\documentclass[11pt]{article}

\usepackage{pratik}

\numberwithin{equation}{section}
\numberwithin{figure}{section}

\newcommand{\titletext}{Optimization-based frequentist confidence intervals for functionals in constrained inverse problems: Resolving the \rustburrus conjecture}

\title{\titletext}

\author{
    Pau Batlle\footremember{caltechcms}{Department of Computing and Mathematical Sciences, California Institute of Technology, CA 91125, USA.} 
    \\ {\small \href{mailto:pbatllef@caltech.edu}{pbatllef@caltech.edu}~} 
    \and Pratik Patil\footremember{berkeleystats}{Department of Statistics, University of California, Berkeley, CA 94720, USA.} 
    \\ {\small \href{mailto:pratikpatil@berkeley.edu}{pratikpatil@berkeley.edu}~} 
    \and Michael Stanley\footremember{cmustats}{Department of Statistics and Data Science, Carnegie Mellon University, Pittsburgh, PA 15213, USA.} 
    \\ {\small \href{mailto:mcstanle@andrew.cmu.edu}{mcstanle@andrew.cmu.edu}~} 
    \and Houman Owhadi\footrecall{caltechcms} 
    \\ {\small \href{mailto:owhadi@caltech.edu}{owhadi@caltech.edu}~} 
    \and Mikael Kuusela\footrecall{cmustats} 
    \\ {\small \href{mailto:mkuusela@andrew.cmu.edu}{mkuusela@andrew.cmu.edu}~}
}

\date{\today}

\begin{document}

\maketitle

\begin{abstract}
We present an optimization-based framework to construct confidence intervals for functionals in constrained inverse problems, ensuring valid one-at-a-time frequentist coverage guarantees.
Our approach builds upon the now-called strict bounds intervals, originally pioneered by \cite{burrus1965utilization,rust_burrus}, which offer ways to directly incorporate any side information about the parameters during inference without introducing external biases. 
This family of methods allows for uncertainty quantification in ill-posed inverse problems without needing to select a regularizing prior. 
By tying optimization-based intervals to an inversion of a constrained likelihood ratio test, we translate interval coverage guarantees into type I error control and characterize the resulting interval via solutions to optimization problems. 
Along the way, we refute the \rustburrus conjecture, which posited that, for possibly rank-deficient linear Gaussian models with positivity constraints, a correction based on the quantile of the chi-squared distribution with one degree of freedom suffices to shorten intervals while maintaining frequentist coverage guarantees. 
Our framework provides a novel approach to analyzing the conjecture, and we construct a counterexample employing a stochastic dominance argument, which we also use to disprove a general form of the conjecture. 
We illustrate our framework with several numerical examples and provide directions for extensions beyond the Rust--Burrus method for nonlinear, non-Gaussian settings with general constraints.
\end{abstract}

\section{Introduction}\label{sec:intro}

Advances in data collection and computational power in recent years have led to an increase in the prevalence of high-dimensional, ill-posed inverse problems, especially within the physical sciences.
These challenges are particularly evident in domains such as remote sensing and data assimilation, where uncertainty quantification (UQ) in inverse problems is of paramount importance.
Many of these inverse problems also come with inherent physical constraints on their parameters.
This paper focuses on constrained inverse problems for which the noise model is known, and the forward model, defined on a finite-dimensional parameter space, can be computationally evaluated.
Our primary objective is to construct a confidence interval for a functional of the forward model parameters.

Formally, we consider statistical models of the form $\by \sim P_{\bx^*}$, where $\by \in \mathbb{R}^m$ is sampled according to a parametric probability distribution. 
Here $\bx^* \in \mathbb{R}^p$ is a fixed unknown parameter, which we know a priori lies within the set $\mathcal X$; see \Cref{fig:diagram-intro} for an illustration.
Our goal is to construct confidence intervals for a known one-dimensional functional $\varphi(\bx^*) \in \RR$.
Ideally, we want the length of these intervals to be as small as possible, while still maintaining a nonasymptotic frequentist coverage guarantee.
In other words, given a prescribed coverage level $1 - \alpha$ for some $\alpha \in (0, 1)$, we want to construct functions of the data $I^-(\by)$ and $I^+(\by)$ such that the following coverage guarantee holds in finite sample\footnote{This form of ``simple'' interval is only for expositional simplicity. 
One can consider more general forms of confidence sets $\cI(\by)$ beyond simple intervals, which we will do when describing the general framework in \Cref{sec:interval_methodology}.}:
\begin{equation} \label{u}
    \inf_{\bx\in\mathcal X}\mathbb{P}_{\by \sim P_{\bx}}\big(\varphi(\bx) \in [I^-(\by), I^+(\by)] \big) \ge 1-\alpha.
\end{equation}
While the requirement \eqref{u} requires that we maintain at least $1 - \alpha$ coverage, we also want it to be approximately accurate by minimizing the slack in the inequality.
Ensuring such \emph{proper calibration}, namely, confidence intervals that do not \textit{undercover} (fail to meet the $1-\alpha$ guarantee for some $\bx \in \mathcal X$) or \textit{overcover} (are too large and therefore exceed the required coverage) is paramount in practical applications.
This is especially true in contexts that require stringent safety and certification standards. 
Intervals that undercover yield unreliable inferences that may expose the system to unforeseen risks.
Conversely, intervals that overcover might lead to excessive economic costs by needing to guard against scenarios that are unlikely to occur.

\begin{figure}[!t]
    \centering
    \includegraphics[width = 0.75\columnwidth]{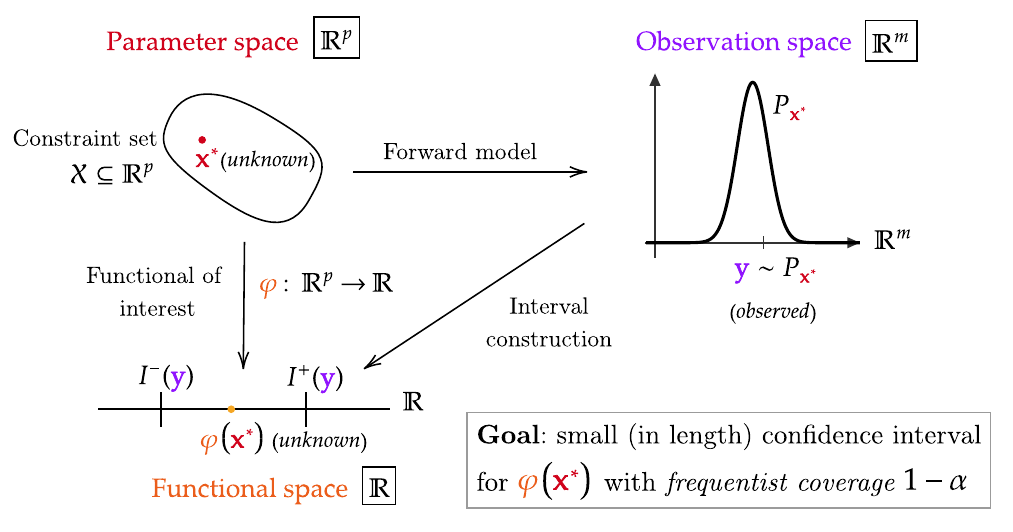}
    \caption{
    Illustration of the problem setup. We seek to construct confidence intervals $[I^-(\by), I^+(\by)] \subseteq \RR$ for $\varphi(\bx^*) \in \RR$ from an observation $\by \in \RR^m$ sampled from $P_{\bx^*}$ that satisfies a frequentist coverage guarantee in finite sample while being as small (in length) as possible.
    }
    \label{fig:diagram-intro}
\end{figure}

In many applied contexts, Bayesian methods constitute a primary set of techniques for uncertainty quantification. 
These methods leverage a prior for regularization, derived either from the intrinsic details of the problem or introduced externally. 
A key advantage of this regularization approach is the natural UQ that emerges from the Bayesian statistical framework. 
Specifically, the combination of a predefined prior and data likelihood results in a posterior distribution via Bayes' theorem. 
This distribution can subsequently be used to derive the intended posterior UQ. 
However, there is a caveat:
Bayesian methods can offer \emph{marginal coverage} (probability over $\bm{x}$ and $\bm{y}$) if the prior is correctly specified.
They do not necessarily provide \emph{conditional coverage} (probability over $\bm{y}$ given $\bm{x}$).
The former notion of coverage is weaker (and, in particular, the latter implies the former, but the converse may not be true), as it replaces the infimum in the coverage requirement \eqref{u} with a probability distribution over $\bx$.
Generally, Bayesian methods may not align with the analyst's expectations due to inherent bias \cite{kuusela_phd_thesis, patil}.
While, in theory, priors present an effective mechanism to incorporate scientific knowledge into UQ, they can inadvertently introduce extraneous information \cite{stark_constraints_priors} and a lack of robustness in the resulting estimates \cite{owhadi2015brittlenessb,owhadi2015brittlenessa, owhadi2017qualitative}.

On the other hand, we could consider a basic worst-case approach that is rooted in the simple observation that: 
\begin{equation}\label{worst_case}
    \varphi(\bx^*) \in \bigg[ \inf_{\bx \in \mathcal X} \varphi(\bx), \; \sup_{\bx \in \mathcal X} \varphi(\bx) \bigg].
\end{equation}
Of course, this method is inherently conservative given the absence of assumptions and any specific knowledge regarding data generation.
More importantly, the method does not use observations $y$ in any way to calibrate the confidence set.
This means that the sets cannot be fine-tuned to approximately achieve the desired $1 - \alpha$ coverage level.
Nevertheless, they illustrate the essential idea of constructing a confidence interval based on the outcomes of two boundary optimization problems, an approach that the more sophisticated methods that we will study in this paper build on.
We shall henceforth refer to such intervals with the notation: 
\begin{equation*}
\begin{aligned}
\inf_{\bx}/\sup_{\bx} \quad & \varphi(\bx) \\
\st \quad & \bx \in \mathcal X 
\end{aligned} := \bigg [ \inf_{\bx \in \mathcal X} \varphi(\bx), \; \sup_{\bx \in \mathcal X} \varphi(\bx) \bigg ].
\end{equation*}

An example of such a more sophisticated method is the so-called ``simultaneous'' approach \cite{stark_strict_bounds,stark1994simultaneous}, which provides intervals with at least $1-\alpha$ frequentist coverage for the functional of interest $\varphi(\bx^*)$ from confidence sets for the parameter $\bx^*$. 
The approach can be summarized in three steps (see \Cref{fig:diagram-intro2} for an illustration): 
\begin{enumerate}[leftmargin=15mm]
    \setlength\itemsep{0em}
    \item[Step 1.] Construct a set $\mathcal{C}(\by) \in \RR^p$ that is a $1-\alpha$ confidence set for $\bx^*$.
    \item[Step 2.] Intersect this set $\mathcal{C}(\by)$ with the constraint set $\mathcal X$.
    \item[Step 3.] Project this intersection through the functional of interest $\varphi$.
\end{enumerate}
The term ``simultaneous'' refers to Steps 1 and 2 being independent of the quantity of interest $\varphi$, so the resulting set from Step 2 can be simultaneously projected to different quantities of interest. 
Under mild assumptions, the resulting intervals can be equivalently written as
\begin{equation}\label{simultaneous}
    \cI_{\SSB}(\by)
    :=
    \bigg [\inf_{\bx \in \mathcal X\cap \mathcal{C}(\by)} \varphi(\bx), \sup_{\bx \in \mathcal X \cap \mathcal{C}(\by)} \varphi(\bx) \bigg ] = \begin{aligned}
\inf_{\bx}/\sup_{\bx} \quad & \varphi(\bx) \\
\st \quad & \bx \in \mathcal X \cap \mathcal{C}(\by).
\end{aligned}
\end{equation}
This illustrates how the simultaneous approach is a refinement of the basic worst-case method \eqref{worst_case}: 
the observation of the data $\by$ shrinks the \say{pre-data set} $\mathcal X$ into a smaller \say{post-data set} $\mathcal X \cap \mathcal{C}(\by)$, which is then projected through $\varphi$ in a worst-case manner. 
Given that this simultaneous framework is broadly encapsulated in \cite{stark_strict_bounds} as ``strict bounds,'' we label these intervals as ``simultaneous strict bounds'' or SSB intervals, for short.

Unlike methods that rely on explicit regularization through a prior, the techniques outlined above leverage only the physical constraints and the functional of interest to address the underlying ill-posedness of the inverse problem. 
This approach allows for uncertainty quantification without the need to assume a prior distribution, circumventing potential biases and miscalibrated coverage issues previously mentioned.

Although the interval \eqref{simultaneous} has guaranteed coverage for $\varphi(\bx^*)$ inherited from the coverage of $\mathcal{C}(\by)$, this method generally suffers from overcoverage, especially when the dimension of $\mathcal X$ is large \cite{patil,stanley_unfolding,kuusela2017shape}. 
This happens due to two main factors: (i) its generality cannot account for the specific structure of $P$, $\varphi$, and $\mathcal X$; and (ii) while the set $\mathcal C(\by)$ being a $1-\alpha$ confidence set is a sufficient condition, it is not necessary for \eqref{simultaneous} to ensure accurate coverage, which implies that smaller sets might also produce valid confidence intervals.
Consequently, an important research direction has been constructing confidence intervals that are shorter than the simultaneous approach, but still maintain nominal coverage for a given $\varphi$. Sometimes, this is achieved by assuming that $P$, $\varphi$, and $\mathcal X$ come from a particular class \cite{stark1994simultaneous, rust_burrus, tenorio2007confidence, patil, stanley_unfolding}.
In the sequel, we discuss one such special class.

\begin{figure}[!t]
    \centering
    \includegraphics[width = 0.8\columnwidth]{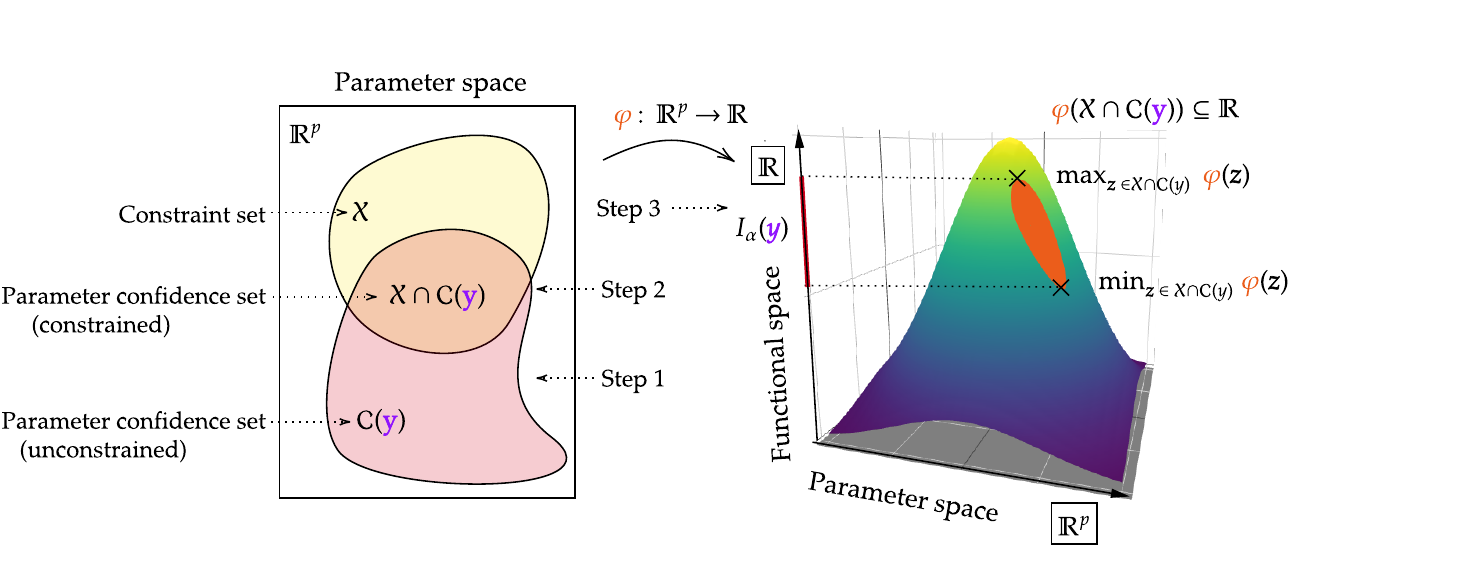}
    \caption{
    Illustration of the simultaneous approach for confidence interval building, which works generically for any $\mathcal X$, $\varphi$ and $P$. 
    The intersection of $\mathcal{X}$ and $\mathcal{C}(\by)$ occurs in the original parameter space $\mathbb{R}^p$, and is then projected via the functional of interest function into the real line. 
    The confidence interval is then constructed using the minimum and maximum of the quantity of interest $\varphi$ over the intersection $\mathcal X \cap \mathcal{C}(\by)$.
    }
    \label{fig:diagram-intro2}
\end{figure}

\subsection{The \rustburrus conjecture}\label{subsec:intro_b_conject}

The Gaussian linear forward model with nonnegativity constraints and a linear functional of interest is a setting that has attracted significant attention, going back to the works of \cite{burrus1965utilization, rust_burrus}. 
These foundational studies consider the applied problem of unfolding gamma-ray and neutron spectra from pulse-height distributions under rank-deficient linear systems.
They demonstrated that incorporating the nonnegativity physical constraint allowed for the computation of nontrivial (i.e., finite length) intervals for linear functionals of the parameters.
In order to describe the construction of these intervals, consider the canonical form of the Gaussian linear model with nonnegativity constraints, along with a linear functional of interest:
\begin{equation}\label{eq:lineargaussianmodel}
\underbrace{\by = \bK\bx^* + \bm{\varepsilon}, \quad \bm{\varepsilon} \sim \mathcal{N}(\bm{0}, \bI_m)}_{\text{model}}, \quad \text{with} \quad \underbrace{\bx^* \geq \bm{0}}_{\text{constraints}} \quad \text{and} \quad \underbrace{\varphi(\bx) = \bh^\tp \bx}_{\text{functional}}.
\end{equation}
Here, $\bK \in \RR^{m \times p}$ is the forward operator\footnote{Note that the forward operator $\bK$ is allowed to be column rank deficient and the overparameterized setting when $p > m$ is allowed.}, $\bx^* \in \RR^p$ is the true parameter vector, and $\bh \in \RR^p$ contains weights for the functional of interest. 
In this setting, \cite{burrus1965utilization, rust_burrus} posed that the following interval construction yields valid $1-\alpha$ confidence intervals, a result now known as the \emph{\rustburrus conjecture} \cite{rust1994confidence}:
\begin{equation}\label{RB1}
\cI_{\OSB}(\by)
:=
\begin{aligned}
\min_{\bx}/\max_{\bx} \quad & \bh^\tp\bx \\
\st \quad &\Vert \by - 
\bK\bx \Vert_2^2 \leq \psi^2_\alpha(\by),\\
  & \bx \geq \bm{0},  \\
\end{aligned}
\end{equation}
where $\psi^2_\alpha = z_{\alpha/2}^2 + s^2(\by)$. Here $z_\alpha$ is the upper quantile of standard normal such that $\mathbb{P}(Z>z_\alpha) = \alpha$ for $Z \sim \mathcal{N}(0,1)$, and $s^2(\by)$ is defined through an optimization problem as follows: 
\begin{equation}\label{RB2}
\begin{aligned}
s^2(\by) :=
\begin{cases}
\begin{aligned}
    \min_{\bx} \:\qquad & \Vert \by - \bK\bx\Vert^2_2 \\ 
    \st \quad & \bx \geq \bm{0}.
\end{aligned}    
\end{cases}
\end{aligned}
\end{equation}
 
Comparison of \eqref{RB1} with \eqref{simultaneous} shows that Rust and Burrus proposed a \say{simultaneous-like} construction. 
In this construction, the set $\{\bx: \Vert \by - \bK \bx \Vert_2^2 \leq \psi^2_\alpha(\by)\}$ plays the role of $\mathcal{C}(\by)$. 
It typically does not represent a $1-\alpha$ confidence set for $\bx^*$, thus relaxing the stringent assumption of the SSB interval construction. Furthermore, a possible simultaneous interval for this setting can be built by observing that $\| \by-\bK\bx^* \|_2^2 \sim \chi^2_m$.
This yields the following valid $1-\alpha$ interval:
\begin{equation}
\begin{aligned}
\min_{\bx}/\max_{\bx} \quad & \bh^\tp\bx \\
\st \quad &\Vert \by - \bK\bx \Vert_2^2 \leq Q_{\chi^2_m}(1-\alpha)\\
  & \bx \geq \bm{0}.  \\
\end{aligned}
\end{equation}
Here, $Q_{\chi^2_m}$ is the quantile function of a $\chi^2_m$ distribution. 
It should be noted that the data-dependent term $\psi^2_\alpha(\by)$ in \eqref{RB1} could be considerably smaller than $Q_{\chi^2_m}(1-\alpha)$, especially when $m$ is large and $\alpha$ is small.
So if the \rustburrus conjecture were true, it would provide a significant reduction in the length of the interval for problems in the class \eqref{eq:lineargaussianmodel}.
For instance, assuming $\alpha = 0.05$ (so that we are after a $95$\% coverage level), \cite{stanley_unfolding} observe an expected length reduction of about a factor of two across a variety of functionals in a particle unfolding application.
The gain in the interval length originates from the fact that these intervals take into account that we are only required to guarantee coverage for \emph{one specific} functional. 
Given that intervals of the form \eqref{RB1} are designed to provide coverage for one functional at a time, following the nomenclature of \cite{stanley_unfolding}, we refer to these intervals as ``one-at-a-time strict bounds'' or OSB intervals, for short\footnote{\cite{patil,stanley_unfolding} also extend the setting and the conjecture to encompass linear constraints of the form $\boldsymbol{A}\bx \leq \boldsymbol{b}$. 
Such constraints are of interest in practical applications such as X$_{\mathrm{CO2}}$ retrieval and particle unfolding. 
For simplicity, we present only the positivity constraint case here. However, our counterexample based on positivity constraints in \Cref{sec:rust_burrus} will also be sufficient to disprove the conjecture in this general case.}.

\cite{rust_burrus} and subsequently \cite{rust1994confidence} investigated the conjecture posed in \cite{burrus1965utilization}. 
The latter work purported to have found definitive proof for the conjecture's validity. 
However, this claim was later refuted by \cite{tenorio2007confidence} through a two-dimensional counterexample.
In this work, we demonstrate that, in fact, this two-dimensional counterexample proposed in \cite{tenorio2007confidence} is not a valid counterexample. 
However, we present and prove another counterexample that refutes the conjecture and we propose ways to fix the previous faulty results by reinterpreting the conjecture. 
We achieve this through a novel hypothesis test-based framework that not only revisits but also broadens the scope beyond the linear Gaussian setting paired with positivity constraints in which the conjecture was originally proposed.

\subsection{Summary and outline}

In this paper, we frame the problem of confidence interval construction for functionals in constrained, ill-posed problems through the inversion of a particular likelihood ratio test. 
This perspective allows us to reinterpret the interval coverage guarantee in terms of type-I error control associated with the test and, subsequently, the distribution of the log-likelihood ratio under the null hypothesis.
We also establish connections between different fields of hypothesis testing with likelihood ratio tests, optimization-based confidence intervals, and chance-constrained optimization.
A detailed summary of contributions in this paper along with an outline for the paper is given below.

\begin{itemize}[leftmargin=7mm]
    \item[(1)]
    \textbf{Strict bounds intervals from test inversion.}
    In \Cref{sec:data_gen_test_def_inv_arg}, we present a general framework to construct strict bounds intervals through test inversion, resulting in two optimization problems for the interval endpoints. 
    This approach generalizes the Rust--Burrus-type interval technique to potentially nonlinear and non-Gaussian settings. 
    Our main result in \Cref{thm:interval_coverage} proves coverage of the test inversion construction and \Cref{thm:interval_coverage_converse} provides sufficient conditions under which the coverage is tight. 
    Examples in \Cref{subsec:examples_theorem} provide straightforward but concrete analytical illustrations of our framework.
 
    \item[(2)]
    \textbf{General interval construction methodology.}
    In \Cref{sec:interval_methodology}, we present a general methodology for computing confidence intervals that builds on the framework in \Cref{sec:data_gen_test_def_inv_arg}.
    We outline the methodology in \Cref{alg:metaalgo} and discuss two key components: the chance-constrained optimization problem and the stochastic dominance argument in \Cref{subsec:CCO} and \Cref{subsec:stochastic_dominance}, respectively.
    The chance-constrained optimization problem allows us to obtain optimal decision values for the proposed framework, while the stochastic dominance argument provides a theoretical tool to find provable upper bounds.
    
    \item[(3)]
    \textbf{Refuting the \rustburrus conjecture.}
    In \Cref{sec:rust_burrus}, we demonstrate that our method successfully recovers previously proposed OSB intervals for the linear Gaussian setting. 
    In \Cref{thm:burrus_false}, we leverage this novel interpretation to disprove the \rustburrus conjecture \cite{rust_burrus, rust1994confidence} in the general case, by refuting a previously proposed counterexample and providing a new, provably correct counterexample in \Cref{lem:3d-counterexample}. 
    Furthermore, we provide a negative result that disproves a natural generalization of the original conjecture in \Cref{lemma:dimbounding}. 
    Our proof technique provides a method to detect when the Rust--Burrus approach is effective and when it falls short and introduces a means to rectify the earlier erroneous examples. 

    \item[(4)]
    \textbf{Illustrative numerical examples.}
    In \Cref{sec:numerical-examples}, we elucidate our findings through a suite of numerical illustrations. 
    These span various scenarios, including the counterexample to the \rustburrus conjecture. 
    We show that test inversion-based intervals, which have provable guarantees, achieve better coverage calibration than previous approaches.
\end{itemize}

\subsection{Other related work}
\label{subsec:RelatedWork}

Given the effectiveness of the strict bounds methodology in high-dimensional ill-posed inverse problems, this paper seeks to deepen our understanding of these intervals and provide related perspectives by connecting them with the broader statistical literature. 
Specifically, we relate these intervals to the well-developed areas of likelihood ratio tests, test inversions, and constrained inference, enabling us to make rigorous statements about their properties and generalize the methodology beyond its earlier confines. 
We provide a brief overview of earlier work in this area below.

\myparagraph{Confidence intervals in penalized inverse problems}
Various optimization-based strategies exist for constructing confidence intervals for functionals in linear inverse problems with constraints.
A first connection between optimization and inference in inverse problems is given by classical approaches seeking to optimize an objective function to balance data misfit with regularization, while adhering to prior constraints \cite{hansen1992analysis,hansen1993use}. 
It is then common to use the variability of the minimizer to quantify uncertainty.
Another closely related strategy employs Bayesian methods to estimate the posterior distribution of model parameters given a regularizing prior and subsequently constructs credible intervals from marginal distributions \cite{tarantola2005inverse,stuart2010inverse}. 
Since these methods effectively quantify uncertainty around the expectation of the regularized estimator, their coverage is highly dependent on the precision of prior information and strength of regularization \cite{patil, kuusela_phd_thesis, kuusela2015}. 
A recent line of work starting with \cite{javanmard2014} attempts to improve the coverage of confidence intervals derived from penalized estimators by ``de-biasing'' the regularized estimators; however, in practice, guarantees can only be obtained asymptotically and finite-sample performance depends on the choice of tuning parameters.

An alternative line of work in optimization-based confidence intervals focuses on ensuring correct frequentist finite-sample coverage.  
This approach is more resistant to the aforementioned challenges associated with relying heavily on prior assumptions or de-biasing and offers a robust framework for uncertainty quantification.
In the following section, we will describe these optimization-based methods in more detail.

\myparagraph{Optimization-based confidence intervals and the \rustburrus conjecture}
This paper is largely motivated by the literature (\cite{rust_burrus, oleary_rust, rust1994confidence, tenorio2007confidence, patil, stanley_unfolding}) surrounding the Burrus conjecture (see \Cref{sec:rust_burrus} for further discussion), which makes a claim about how to set a calibration parameter in an optimization-based confidence interval construction so that the resulting interval has a desired level of coverage for a single functional.
For intervals with a simultaneous coverage guarantee for an arbitrary collection of functionals, \cite{stark1992inference} provides the most general optimization-based confidence interval construction.
While these intervals provide the desired coverage, they are overly conservative in terms of length when compared to intervals calibrated for a specific functional \cite{stanley_unfolding}.
These prior works consider only the Gaussian linear inverse problem and can thus be seen as a particular instance of the more general optimization-based confidence intervals treated in this paper. 

\myparagraph{Inverting likelihood ratio tests and constrained inference}

Traditionally, optimization-based confidence interval constructions in inverse problems have developed somewhat independently of the broader statistical literature, often overlooking the duality between confidence intervals and hypothesis testing \cite{casella_berger, wasserman2004all, lehmann_romano}. 
Our work reinterprets these optimization-based confidence intervals from the inverse problem literature as inverted hypothesis tests and situates them within the realm of constrained testing and inference; see, e.g., \cite{gourieroux, wolak87, robertson, shapiro1988towards, wolak89,molenberghs2007likelihood}, among others.

The constrained inference literature often employs the $\bar{\chi}^2$ distribution, a convex combination of $\chi^2$ distributions with different degrees of freedom, dictated by the problem constraints.
Recent work in \cite{yu2019constrained} has extended these constrained testing frameworks to high-dimensional settings with linear inequality constraints, examining both sparse and non-sparse scenarios. 
Although such tests can be more powerful than their unconstrained counterparts, their definitions typically limit the null hypothesis to linear subspaces, complicating their use in test inversion scenarios~\cite{silvapulle2011constrained}.

Although there have been applications of constrained test inversion (\cite{Feldman_1998}), these are limited in scope due to grid-based inversion approaches. 
The statistics literature contains other approaches to inverting likelihood ratio tests (LRTs), which center around sampling procedures 
\cite{cash1979parameter, venzon1988method, garthwaite_buckland, LikTestInv, neale1997use, schweiger}.
Alternatively, one can sample from the parameter space and the forward model to generate training data for a quantile regression, which can then be used to invert an LRT (\cite{dalmasso2020confidence, waldo,dalmasso2021likelihood, fisher2020efficient}). 
Since these latter approaches require sampling points in the parameter space, they are practically limited to compact parameter spaces and may encounter difficulties with high-dimensional parameters. 
In scenarios where the data can be split, approaches such as Universal Inference \cite{universal_inference} offer a way to obtain confidence sets for irregular likelihoods with nonasymptotic coverage.

\myparagraph{Worst-case and likelihood-free methods}

Most of the approaches and methods referenced and described above make strong assumptions about the underlying data-generating distribution (e.g., linear forward model and Gaussian noise).
To generalize these assumptions, one can either take a worst-case approach within the model class (e.g., \cite{donoho1994statistical} which looks at worst-case confidence intervals for linear inverse problems) or remove distribution assumptions altogether. 
For example, Optimal Uncertainty Quantification (OUQ) (see, e.g., \cite{owhadiOUQ}) does not assume a particular likelihood function to perform statistical inference by relying instead on worst-case bounds.
If one is willing to make boundedness assumptions on the parameter space, simulation-based inference approaches such as \cite{LikFree1, LikFree2,dalmasso2020confidence, waldo,dalmasso2021likelihood,cranmer2020frontier,cranmer2016approximating} explore the use of sampling-only access to the likelihood, typically through a simulator, which has found particular relevance in the physical sciences.
While these likelihood-free methods are advantageous in contexts where the likelihood is uncertain, unknown or accessible only through a simulator, they tend to yield conservative estimates when a well-defined likelihood is available.

\section{Strict bounds intervals from test inversion}
\label{sec:data_gen_test_def_inv_arg}

Suppose that we observe data $\by \in \mathbb{R}^m$ according to a data-generating process $\by \sim P_{\bx^*}$. Here, $P_{\bx^*}$ is a distribution that depends on a fixed but unknown parameter $\bx^* \in \mathbb{R}^p$. 
Furthermore, suppose that we have prior knowledge that this parameter $\bx^*$ lies in a constraint set $\mathcal{X} \subseteq \mathbb{R}^p$, namely $\bx^* \in \mathcal{X}$. 
Given a nominal coverage level $1-\alpha$, where $\alpha \in (0,1)$, this paper investigates methods for constructing a $1-\alpha$ confidence interval for $\varphi(\bx^*)$, where $\varphi: \mathbb{R}^p \to \mathbb{R}$ is a known one-dimensional quantity of interest\footnote{Confidence sets of several functionals of interest with guarantees can be constructed by using the proposed method with, e.g., Bonferroni correction, but studying the performance of that approach is beyond the scope of this work}
(we will also refer to $\varphi$ as a functional of interest).
More precisely, we are interested in constructing an interval 
$\cI_\alpha(\by) \subseteq \mathbb{R}$ 
for $\varphi(\bx^*)$ that satisfies the following coverage requirement: 
\begin{equation}
    \label{eq:coverage_requirement}
    \mathbb{P}_{\by \sim P_{\bx^*}} \big(\varphi(\bx^*) \in \cI_\alpha(\by) \big) \geq 1 - \alpha, \quad \text{for all} \quad \bx^* \in \mathcal{X}.
\end{equation}
Our primary focus lies in intervals that: (i) effectively utilize the information that $\bx^* \in \mathcal X$, (ii) are valid (i.e., satisfying the coverage requirement in \eqref{eq:coverage_requirement}) in the finite data and noisy regimes (rather than, e.g., in the large system or noiseless limits), (iii) do not make overly restrictive assumptions (e.g., identifiability) about the structure of the parametric model $P_{\bx^*}$, and (iv) are short in length\footnote{Note that length will, in general, depend on the unknown parameter $\bx^*$. 
There are several notions for the ``optimality'' of the method with respect to length, such as minimax length \cite{donoho1994statistical,schafer2009constructing} or expected length \cite{stanley_unfolding}, among others.}.
We view the observation vector $\by$ as a single observation in $\RR^{m}$ drawn from a multivariate distribution $P_{\bx^*}$. 
This may include the case of repeated sampling (i.i.d.\ or not) from an experiment and aggregating the samples in a vector. 
In this case, $P_{\bx^*}$ is then defined as the measure that accounts for all the observations.\footnote{For example, in the typical case where a $d$ dimensional vector is observed a total of $n$ times, we aggregate the results in an $m = n \times d$ dimensional vector. 
Throughout, we use $m$ to denote the total dimensionality of the observation vector.}

\subsection{Review: classical test inversion for simple null hypotheses} \label{subsec:test_inversion}

We briefly review the concept of test inversion and the duality between hypothesis testing and confidence sets upon which the subsequent subsections will be built.
After observing $\by \sim P_{\bx^*}$, two classical statistical tasks emerge: (i) determining whether $\bx^* = \bx$ for a particular $\bx \in \mathcal X$ at a significance level $\alpha$ (hypothesis testing), and (ii) constructing a subset of $\mathcal X$ that contains $\bx^*$ with a coverage level $1 -\alpha$ (confidence set building). 
In hypothesis testing, for a given parameter $\bx$, one can consider the hypothesis test:
\begin{equation} \label{eq:h_test_classic}
    H_0 \colon \bx^* = \bx \quad \text{versus} \quad H_1 \colon \bx^* \neq \bx.
\end{equation}
We then build an acceptance region $A(\bx)$ in the data space (the space in which the observations $\by$ live) corresponding to the observations that would not reject $H_0$, with the condition that $H_0$ is rejected with probability at most $\alpha$ when it is true. 
In confidence set building, one builds a subset in parameter space (the space in which the parameters $\bx$ live) as a function of the data, $\cC(\by)$, such that it contains $\bx^*$ with probability at least $1-\alpha$ (over repeated samples of $\by \sim P_{\bx^*}$). 

Lifting to the product space of the data and parameter spaces (see \Cref{fig:duality} for an illustration), both tasks amount to the construction of a compatibility region $\mathcal S$. 
For a fixed observation $\by$, a confidence set is given by $\cC(\by) = \{ \bx \colon (\by, \bx) \in \mathcal S\}$, and for a fixed parameter $\bx$, the acceptance region is given by $\cA(\bx) = \{ \by \colon (\by, \bx) \in \mathcal S\}$. 
Observe that $\mathbb{P}(\by \in \cA(\bx)) = \mathbb{P}(\bx \in \cC(\by))$. 
Therefore, a procedure that forms confidence sets with coverage $1-\alpha$ for any possible data $\by$ also creates a procedure that yields valid hypothesis tests at the level $\alpha$ for any possible parameter value $\bx$, and vice versa. 
This observation can be used to create confidence sets as the set of parameter values that would not be rejected by a hypothesis test, a construction known as test inversion (see, e.g., Chapter~7 of \cite{casella_berger} or Chapter~5 of \cite{Panaretos2016}).

\begin{SCfigure}[50][!ht]
    \centering
    \includegraphics[width = 0.4\columnwidth]{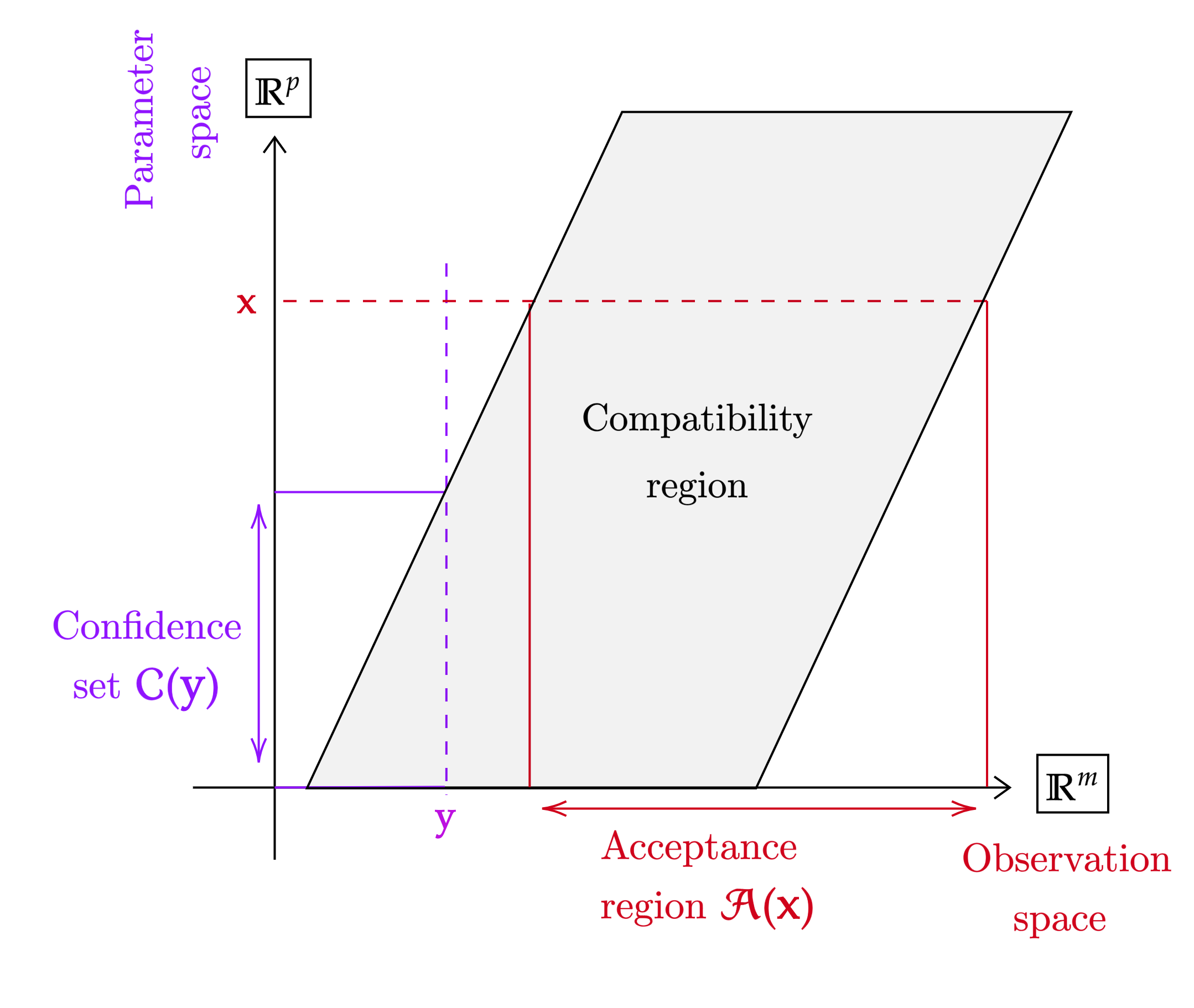}
    \caption{
    Illustration of the classical duality between hypothesis testing and confidence set building as seen in the product space of the data and parameter spaces. 
    Pairs of the dual hypothesis test and the confidence set can be viewed as a set $\mathcal S$ in the product space (the compatibility region). 
    For fixed data $\by$, a confidence set is given by $\cC(\by) = \{ \bx: (\by, \bx) \in \mathcal S\}$, and for a fixed parameter $\bx$, the acceptance region is given by $\cA(\bx) = \{ \by: (\by, \bx) \in \mathcal S\}$.
    \vspace{3em}
    }
    \label{fig:duality}
\end{SCfigure}

\subsection{Formulation and inversion of constrained likelihood ratio tests}\label{subsec:LRT}

The starting point of this work is the inversion of specific hypothesis tests that can incorporate the constraint information $\mathcal X$ and the functional of interest $\varphi$. 
We will establish that the test inversion can be achieved by solving two endpoint optimization problems. 
We note that unlike the simple null versus composite alternative tests \eqref{eq:h_test_classic} described in \Cref{subsec:test_inversion}, the tests we will consider have composite nulls. We focus on the continuous case and assume that the Lebesgue measure dominates the set of distributions $\mathcal P := \{P_{\bx} \mid \bx \in \mathcal X\}$.
However, a discrete analog can also be constructed using a similar approach as in \cite{Feldman_1998}. 
Let $L_\bx$ be the density of $P_\bx$, and let $\ell_\bx := \log L_\bx$. 
For any $\mu \in \mathbb{R}$, denote the level sets of the quantity of interest $\varphi$ by $\Phi_\mu$. These are defined as follows: 
\begin{equation} \label{eq:affine_subspace}
    \Phi_\mu := \{\bx \colon \varphi(\bx) = \mu \} \subseteq \mathbb{R}^p.
\end{equation}
Subsequently, define a hypothesis test $T_\mu$ as follows:
\begin{equation} \label{eq:h_test_oi_full}
    H_0 \colon \bx^* \in \Phi_\mu \cap \mathcal{X} \quad \text{versus} \quad H_1 \colon \bx^* \in \mathcal{X} \setminus \Phi_\mu.
\end{equation}
We can test hypothesis~\eqref{eq:h_test_oi_full} (for a fixed $\mu$) with a Likelihood Ratio (LR) test statistic defined as the following function of the observed data $\by$:
\begin{equation} \label{eq:likelihood_ratio}
    \Lambda(\mu, \by) := \frac{\underset{\bx \in \Phi_\mu \cap \mathcal{X}}{\text{sup}} \; L_{\bx}(\by)}{\underset{\bx \in \mathcal{X}}{\text{sup}} \; L_{\bx}(\by)}.
\end{equation}
The corresponding log-likelihood ratio (LLR) statistic $\lambda(\mu, \by)$ is given by:
\begin{align} \label{eq:log_likelihood_ratio_full}
     \lambda(\mu, \by) = -2 \log \Lambda(\mu, \by) &= -2 ~ \bigg\{\underset{\bx \in \Phi_\mu \cap \mathcal{X}}{\text{sup}} \; \ell_{\bx}(\by)- \underset{\bx \in \mathcal{X}}{\text{sup}} \; \ell_{\bx}(\by) \bigg\} \\ &= \underset{\bx \in \Phi_\mu \cap \mathcal{X}} {\text{inf}} \; -2\ell_{\bx}(\by)- \underset{\bx \in \mathcal{X}}{\text{inf}} \; -2 \ell_{\bx}(\by). \notag
\end{align}
As is standard (see, e.g., \cite{casella_berger,wasserman2004all}), we use the supremum over all $\mathcal X$ in the denominator of \eqref{eq:likelihood_ratio}, instead of over $\mathcal X \setminus \Phi_\mu$\footnote{\cite{schervish} provides conditions for the equality of both test statistics.}. 
The factor of $-2$ helps connect with the standard likelihood ratio test in the context of Wilks' theorem and is needed, together with the optimization being over the whole space, to reinterpret the previous constrained inference intervals as coming from the inversion of this test (see \Cref{sec:rust_burrus}). 

\myparagraph{Motivation behind the choice of test and test statistic}
In addition to the reinterpretation of the previous constrained inference intervals as a result of inverting this test, there are other theoretical and practical reasons that make it a reasonable choice for this work. 
Theoretically, the LR emerges as the optimal test statistic (resulting in the most powerful level-$\alpha$ test) in the simple versus simple hypothesis testing setting via the Neyman--Pearson Lemma \cite{casella_berger, lehmann_romano}. 
Although uniformly most powerful tests do not exist in general, LR tests have been effective in several contexts. 
For example, \cite{wald_1943} provides some optimality properties for the likelihood ratio test in terms of its asymptotic average power.
Although our test of interest does not fall under the simple versus simple paradigm and we are interested in nonasymptotic properties, these two properties support the sensibility of adopting the LR-based test. 
Furthermore, the literature on constrained inference \cite{robertson, silvapulle2011constrained} extensively uses the LR, deriving both asymptotic and nonasymptotic log-likelihood ratio (LLR) distributions in various scenarios, often leading to the $\bar{\chi}^2$ distribution. 
These characterizations indicate that, in very specific situations, it is possible to obtain the distribution of the test statistic under the null hypothesis, either exactly or in an asymptotic sense. 
Our setting extends well beyond those situations because our setup is general, and we do not make particular assumptions on the likelihood model or the constraint set. 

\myparagraph{The distribution of the LLR and test inversion}
In hypothesis testing, we reject the null hypothesis when the values of $\lambda(\mu, \by)$ exceed a threshold. 
This indicates that there is substantial evidence against the data being generated by a distribution in the composite null defined by $\mu$.
To choose a rejection region, we next study the distribution of the LLR, denoted as $\lambda(\mu, \by)$, in the context where $\mu = \varphi(\bx)$ (pertaining to the null hypothesis) and $\by \sim P_{\bx}$, a data sampling model, across various values of $\bx \in \mathcal X$.
Let $F_\bx$ denote the distribution of $\lambda(\varphi(\bx), \by)$ for any $\bx \in \mathcal X$, where $\by \sim P_\bx$.
To simplify the notation, we will write $\lambda \sim F_\bx$ to indicate that an LLR is sampled following the procedure described above.

To ensure an $\alpha$-level test for test inversion, we need to control the distribution of the test statistic under the null hypothesis. 
Since the null is composite, the false positive rate must hold for any parameter under the null hypothesis $H_0$. 

Suppose that we are conducting a test $T_\mu$ to determine whether $\mu^* = \varphi(\bx^*)$ equals a particular $\mu \in \varphi(\mathcal X) \subseteq \mathbb{R}$, that is, $\bx^* \in \Phi_{\mu} \cap \mathcal{X}$.
We use $\lambda > q_\alpha$ as the rejection region, where $q_\alpha$ is a predetermined decision threshold. 
Under the null hypothesis, if the decision threshold satisfies:
\begin{equation} \label{eq:type1_err_c_alpha}
    \sup_{\bx \in \Phi_{\mu} \cap \mathcal{X}} \mathbb{P}_{\lambda \sim F_{\bx}} \left(\lambda > q_\alpha \right) \leq \alpha
\end{equation}
for all $\alpha \in (0, 1)$, then we say $T_\mu$ is a \emph{level-$\alpha$ test}.\footnotemark
\footnotetext{Here $q_\alpha$ is the decision value corresponding to intervals with a coverage probability of 1-$\alpha$, aligning with classical textbook notation (see, e.g., \cite{casella_berger}, \cite{wasserman2004all}). 
For any random variable $Z$, we will denote with the subscript $\alpha$ the cutoff points that satisfy $\mathbb{P}(Z > z_\alpha) = \alpha$.}

Inverting the test with respect to $\mu$ will require choosing an appropriate $q_\alpha$ for all $\mu$; henceforth we will denote it as $q_\alpha(\mu)$. 

We seek to invert this test using a methodology similar to that outlined in \Cref{subsec:test_inversion}, but adapted to accommodate the composite null hypothesis.
The acceptance region is formally defined as:
\begin{equation} \label{eq:acceptance_region_c_alpha}
    \cA_\alpha(\mu) := \left\{ \by: \lambda(\mu, \by) \leq q_\alpha(\mu) \right\}.
\end{equation}
Subsequently, we define the proposed confidence set for $\mu^* = \varphi(\bx^*) \in \mathbb{R}$ through test inversion as follows:
\begin{equation}\label{eq:conf_set_def}
    \cC_\alpha(\by) := \{ \mu \colon \lambda(\mu, \by) \leq q_\alpha(\mu) \}.
\end{equation}

We prove in \Cref{lem:conf_set_coverage} that if \eqref{eq:type1_err_c_alpha} is satisfied for $\mu^* := \varphi(\bx^*)$ (that is, $T_{\mu^*}$ is a level-$\alpha$ test), the resulting confidence set will have the desired $1 - \alpha$ coverage, thus extending the classical test inversion framework to our specific case.

\begin{lemma}[Coverage of the inverted test] \label{lem:conf_set_coverage}
Let $\alpha \in (0, 1)$.
Let $\bx^*$ be the true parameter value and $\mu^*$ its image under $\varphi$. 
If $T_{\mu^*}$ is a level-$\alpha$ test, then 
\[
\mathbb{P}_{\by \sim P_{\bx^*}} \left( \mu^* \in \cC_\alpha(\by) \right) \geq 1 - \alpha.
\]
\end{lemma}
\begin{proof}[Proof sketch]
    The proof is based on a straightforward test inversion argument.
    For a detailed proof, see \Cref{proof:conf_set_coverage}.
\end{proof}

To ensure that condition \eqref{eq:type1_err_c_alpha} holds in practice, when $\bx^*$ and therefore $\mu^*$ are both unknown, we need to satisfy this condition for all possible null hypotheses.
Specifically, we choose an appropriate $q_\alpha(\mu)$ for each $\mu$ to ensure that all hypothesis tests $T_\mu$ are level-$\alpha$.
Formally, this is expressed as:
\begin{equation} \label{eq:combined_type_1}
   \sup_{\mu \in \varphi(\mathcal X)} \sup_{\bx \in \Phi_\mu \cap \mathcal{X}} \mathbb{P}_{\lambda \sim F_{\bx}} \left(\lambda> q_\alpha(\mu) \right) \leq \alpha.
\end{equation}
This condition is equivalent\footnote{
Note that every $\bx \in \mathcal{X}$ is accounted for in \eqref{eq:combined_type_1} since $\mu = \varphi(\bx)$. 
} to: 
\begin{equation} \label{eq:combined_type_1_better}
  \sup_{\bx \in \mathcal{X}} \mathbb{P}_{\lambda \sim F_{\bx}} \left(\lambda> q_\alpha(\varphi(\bx)) \right) \leq \alpha.
\end{equation}
Although \eqref{eq:combined_type_1_better} lacks the interpretation of \eqref{eq:combined_type_1} of having hypothesis tests for each different $\mu \in \varphi(\mathcal{X})$, it simplifies the calculations. 
We refer to a set of values $q_\alpha(\mu)$ that satisfy \eqref{eq:combined_type_1_better} (or equivalently \eqref{eq:combined_type_1}) as \emph{valid values}. 
Since $\mu$ in \eqref{eq:combined_type_1} is equal to $\varphi(\bx)$ as $\bx \in \Phi_\mu$, we can use $q_\alpha(\mu)$ and $q_\alpha(\varphi(\bx))$ interchangeably.

From \Cref{lem:conf_set_coverage}, we know that valid values can be used in \eqref{eq:conf_set_def} to construct a confidence set for $\mu^*$ with the correct $1-\alpha$ coverage. 
Moreover, as argued in the proof of \Cref{lem:conf_set_coverage}, the probability that the set \eqref{eq:conf_set_def} covers the unknown $\mu^*$ is given by:
\begin{equation}\label{eq:exactprobability}
    \mathbb{P}_{\by \sim P_{\bx^*}} \left( \mu^* \in \cC_\alpha(\by) \right) = 1-\mathbb{P}_{\lambda \sim F_{\bx^*}} \left(\lambda > q_\alpha(\mu^*) \right),
\end{equation}
which is guaranteed to be at least $1-\alpha$ by the condition \eqref{eq:combined_type_1}. 
To obtain intervals with the smallest possible size while maintaining coverage, we aim to find the optimal decision values $q_\alpha(\mu)$, which are solutions to optimization problems involving the quantiles $Q_{F_{\bx}}: [0,1] \to \mathbb{R}$ of the distributions of the family $\{F_{\bx}, \bx \in \mathcal{X}\}$.
\begin{lemma}[Optimal decision values]\label{lemma:quantile}
    The optimal (smallest valid) value of $q_\alpha(\mu)$ is given by the maximum quantile (MQ) optimization problem:
    \begin{equation} \label{eq:QuantileOpt1}
    \maxQ_{\mu, 1-\alpha} := \sup_{\bx \in  \Phi_\mu \cap \mathcal{X}} Q_{F_\bx}(1-\alpha).
    \end{equation}
    Furthermore, if one wants to choose a single $q_\alpha$ for all $\mu$, then the optimal value is given by:
    \begin{equation} \label{eq:QuantileOpt2}
        \maxQ_{1-\alpha} := \sup_{\mu \in \varphi(\mathcal{X})} \maxQ_{\mu, 1-\alpha} = \sup_{\bx \in \mathcal{X}}  Q_{F_\bx}(1-\alpha).
    \end{equation}
\end{lemma}
\begin{proof}[Proof sketch] 
It can be directly checked that the proposed quantities are valid and that using any smaller decision value leads to intervals with undercoverage for at least one point. See \Cref{proof:lemma:quantile} for more details.
\end{proof}
In the event that only an upper bound on the quantities defined in \Cref{lemma:quantile} can be obtained, those can also be used as valid decision values, as stated precisely in the following corollary.
\begin{corollary}
    For every $\mu \in \mathbb{R}$, $\nu(\mu) \geq \maxQ_{\mu, 1-\alpha}$ (as defined in \Cref{lemma:quantile})  if and only if $\nu(\mu)$ is a valid decision value for $T_\mu$ for a particular $\alpha$.
    In addition, $\nu \geq \maxQ_{1-\alpha}$ if and only if $\nu$ is a valid decision value for $T_\mu$ for all $\mu \in \mathbb{R}$ for a particular $\alpha$.
\end{corollary}
    The result follows immediately from substituting the proposed $\nu$ values into the probability statements in \Cref{lemma:quantile}.
Theoretical and computational methods for obtaining valid $q_\alpha(\mu)$ are discussed in \Cref{sec:interval_methodology}, and we investigate the computation of $\cC_\alpha(\by)$ via optimization techniques in the next subsection, assuming valid $q_\alpha(\mu)$ are known. 
\subsection{Characterizing the inverted confidence set via optimization problems} \label{subsec:conf_set_as_opt}

The set defined in \eqref{eq:conf_set_def} produces a random collection of real numbers that contains the true functional value with a probability of at least $1 - \alpha$. 
Although this set is not necessarily an interval, it is contained within an interval whose (possibly infinite) extremes are computable through optimization techniques.

Given a valid $q_\alpha(\mu)$, which satisfies either \eqref{eq:combined_type_1} or \eqref{eq:combined_type_1_better}, let us define the following sets:
\begin{align}
\mathcal{D}(\by) &:= \{ \bx \colon -2\ell_{\bx}(\by) \leq q_\alpha(\varphi(\bx)) + \inf_{\bx' \in \mathcal X} -2\ell_{\bx'}(\by) \} \subseteq \mathbb R^p, \label{eq:defD}\\
\bar{\mathcal X}_\alpha(\by) &:= \mathcal X \cap \mathcal{D}(\by). \label{eq:defcalX}
\end{align}
If $\bar{\mathcal X}_\alpha(\by) \neq \emptyset$, we further define:
\begin{equation} \label{intervaldef}
\cI_\alpha(\by) := \bigg [\inf_{\bx \in \bar{\mathcal X}_\alpha(\by)} \varphi(\bx),\; \sup_{\bx \in \bar{\mathcal X}_\alpha(\by)} \varphi(\bx) \bigg ].
\end{equation}
If $\bar{\mathcal X}_\alpha(\by) = \emptyset$, let $\cI_\alpha(\by)$ be the empty interval.
\begin{theorem}[From test inversion to optimization-based intervals] 
\label{thm:interval_coverage} 
    For any $\alpha \in (0, 1)$, and for any $\bx \in \mathcal{X}$, let $\cI_\alpha(\by)$ be the interval constructed according to \eqref{intervaldef}.
    It holds that
    \begin{equation*}
    \mathbb{P}_{\by \sim P_{\bx}} \left(\varphi(\bx) \in \cI_\alpha(\by) \right) \geq 1-\alpha.
    \end{equation*}
    In other words, $\cI_\alpha(\by)$ is a valid $1-\alpha$ confidence interval for $\varphi(\bx^*)$.
\end{theorem}
\begin{proof}[Proof sketch]
We prove that
\begin{equation}
\bigg [\inf_{\lambda(\mu, \by) \leq q_\alpha(\mu)} \mu,\; 
\sup_{\lambda(\mu, \by) \leq q_\alpha(\mu)} \mu \bigg ] = \cI_\alpha(\by).
\end{equation}
The object on the left-hand side is defined by enclosing $\cC_\alpha(\by)$ within the smallest possible interval that contains it, and therefore, it has guaranteed coverage. 
The equality arises from the equivalence between the optimization problems under consideration.
For a complete proof, see \Cref{proof:main}.
\end{proof}

\begin{remark}
    [Comparison with the simultaneous strict bound intervals]
    Observe that the construction of $\cI_\alpha(\by)$ follows the form outlined in \eqref{simultaneous} for the simultaneous strict bound intervals. 
    However, a key distinction lies in not requiring that $\mathcal{D}(\by) \subseteq \mathcal X$ serves as a $1-\alpha$ confidence set for $\bx$.
    This relaxation will translate into shorter intervals when $q_\alpha$ is chosen appropriately.    
\end{remark}

\begin{remark}
[Handling empty constrained sets]
If $\alpha$ is chosen such that $1-\alpha$ becomes too small, the set $\bar{\mathcal{X}}_\alpha(y)$ can be empty. 
In that case, we default to the empty interval, under the interpretation that there are no parameter values that simultaneously agree with the constraint and the observed data (at a particular level $\alpha$). 
However, the actual interval produced under this circumstance does not compromise the $1-\alpha$ coverage level provided by the theorem. 
If a point estimate inside the constraint region is desired, an option is to choose the closest point from $\mathcal X$ to $\mathcal D$. 
This point specifically ensures the continuity of the interval with respect to $\alpha$ in many standard scenarios. 
Generally, an empty set $\bar{\mathcal{X}}_\alpha(\by)$ should inform one of three possibilities: either (i) an outlier event has been observed, or (ii) the initial assumption that $\bx \in \mathcal X$ is flawed, or (iii) the forward model $P_\bx$ is misspecified.
Here, the definition of an ``outlier'' is intrinsically linked to the choice of $\alpha$. 
A larger $\alpha$ will make such events more frequent, as it broadens the range of data considered as outliers.
\end{remark}

We also present a partial converse result, stating that the interval coverage implies the validity of $q_\alpha$, subject to appropriate assumptions on $\varphi$, $P$, and $\mathcal X$. 
This result will be instrumental in refuting the coverage claims of the Rust--Burrus intervals, and consequently, the \rustburrus conjecture, as discussed in \Cref{sec:rust_burrus}.

\begin{proposition}[Coverage implies validity of quantile levels]\label{thm:interval_coverage_converse}
Assume that $\mathcal X$ forms a convex cone, $\ell_\bx(\by)$ is a concave function, and $\varphi(\bx)$ is linear.
Define $\cI_\alpha(\by)$ as in \Cref{thm:interval_coverage}, for a particular choice of $q_\alpha(\mu)$. 
If $\cI_\alpha(\by)$ is a valid $1-\alpha$ confidence interval for all $\bx$, then the values of $q_\alpha(\mu)$ are valid.
\end{proposition}
\begin{proof}[Proof sketch]
    Generally, the values of $q_\alpha(\mu)$ are valid if and only if $\cC_\alpha(\by)$ constitutes a $1-\alpha$ set.
    Since $\cI_\alpha(\by)$ is the smallest interval that contains $\cC_\alpha(\by)$, if $\cC_\alpha(\by)$ is already an interval, then the result holds. 
    The assumptions on $\mathcal X, \ell_\bx(\by)$ and $\varphi$ ensure that this is the case by the convexity of the function 
    \[
        \mu \mapsto \inf_{\substack{\varphi(\bx) = \mu \\ \bx \in \mathcal{X}}} -2\ell_\bx(\by)
    \]
    for any $\by$. 
    For a detailed proof, see \Cref{proof:main_converse}.
\end{proof}

Finally, we remark that the construction presented in this paper provides an approach to uncertainty quantification that does not rely on a specific point estimator, distinguishing it from many other UQ procedures. 
However, if one wishes to obtain a point estimator, it is worth noting that the midpoint of the interval can be justified from a decision-theoretic perspective.
This idea has been discussed in previous works by \cite{micchelli1977survey,Bajgiran2022}, among others.

\subsection{Illustrative examples}\label{subsec:examples_theorem}

To elucidate the general methodology outlined in \Cref{thm:interval_coverage}, we offer two simple illustrative examples where the LLR and its distribution are explicitly computable: a one-dimensional constrained Gaussian scenario and an unconstrained linear Gaussian case.

\myparagraph{Constrained Gaussian in one dimension}
As a tangible example, consider the following one-dimensional model:
\begin{equation} \label{eq:1d_toy_model}
    \underbrace{y= x^* + \varepsilon, \quad \varepsilon \sim \mathcal{N}(0,1)}_{\text{model}} \quad \text{with} \quad \underbrace{x^* \geq 0}_{\text{constraints}} \quad \text{and} \quad \underbrace{\varphi(x) = x}_{\text{functional}}.
\end{equation}
In this case, the distribution of the LLR is precisely known.
Hence, a confidence interval can be constructed without resorting to the techniques introduced in \Cref{sec:interval_methodology}, which are otherwise necessary when such information is not available.

The form of the hypothesis test $T_\mu$, as given in \eqref{eq:h_test_oi_full}, is as follows:
\begin{equation} \label{eq:1d_hypoth}
    H_0: x^* = \mu \quad \text{versus} \quad H_1: x^* \neq \mu \text{ and } x^* \geq 0.
\end{equation}
The LLR as defined in \eqref{eq:log_likelihood_ratio_full} for the test \eqref{eq:1d_hypoth} is given by:
\begin{align} \label{eq:1d_log_lr}
    \lambda(\mu, y)~ &= \inf_{x = \mu, x \geq 0} (y - x)^2 - \inf_{x \geq 0} (y - x)^2 \nonumber \\
    &= \begin{cases}
        (y - \mu)^2, & y \geq 0, \\
        (y - \mu)^2 - y^2, & y < 0.
    \end{cases}
\end{align}
We can also derive its distribution under the null hypothesis (i.e., when $x^* = \mu$, leading to $y = \mu + \varepsilon$) for any $\mu \in [0, \infty)$, as formalized below.
\begin{example}[Distribution of the LLR statistic for a constrained Gaussian in one dimension]
\label{lemma:1d_cdf_part1}
    For $\lambda(\mu, y)$  as defined in \eqref{eq:1d_log_lr} with $\mu \geq 0$, when $y \sim \mathcal N(\mu, 1)$ (null hypothesis),
    for all $c > 0$,
    we have
    \begin{equation*}
        \mathbb{P} (\lambda(\mu, y) \leq c ) = \begin{cases}
        \chi^2_1(c) + \frac{1}{2}, & \mu = 0 \\
        \chi^2_1(c) \cdot \bm{1} \{ c < {\mu}^2 \} + \{ \Phi (\sqrt{c}) - \Phi( {(-{\mu}^2 - c)}/(2 \mu) ) \} \cdot \bm{1} \{ c \geq {\mu}^2 \}, & \mu > 0,
        \end{cases}
    \end{equation*}
    where $\chi^2_1$ and $\Phi$ are the CDFs of a $\chi^2_1$ and a standard Gaussian, respectively.
\end{example}
\begin{proof}
   See \Cref{app:proof-lemma:1d_cdf_part1}.
\end{proof}

The expression for $\lambda(\mu, y)$, with the appropriately scaled log transformation, is equivalent to Equation (4.3) in \cite{Feldman_1998} where the Neyman confidence interval construction for the same problem is considered. 
\cite{Feldman_1998} characterizes this quantity as a likelihood ordering for determining an acceptance region. 

By virtue of the previous result and \Cref{lemma:quantile}, we can take $Q_{\mu}(1-\alpha)$, where $Q_\mu$ is the quantile of the distribution of $\lambda(\mu, y)$ when $\mu$ is fixed, as $q_\alpha(\mu)$ satisfying \eqref{eq:type1_err_c_alpha}. 
A direct computation shows
\begin{equation} \label{eq:1d_quantile_func}
q_\alpha(\mu) = Q_{\mu}(1-\alpha) = \begin{cases}
      Q_{\chi^2_1}(1-\alpha), & 1-\alpha < \chi^2_1(\mu^2), \\
        r_{\mu, \alpha}, & 1-\alpha \geq \chi^2_1(\mu^2), \\ 
\end{cases}
\end{equation}
where $r_{\mu, \alpha}$ is the unique nonnegative root of the function $x \mapsto \Phi(\sqrt{x}) - \Phi({(-\mu^2-x)}/(2\mu)) - (1-\alpha)$, which can be found using numerical methods.
Therefore, $\mathcal D = \{x: (y-x)^2 \leq q_\alpha(x) + \min_{x' \geq 0} (y-x')^2 \}$ and the final form of the confidence interval becomes:
\begin{equation*}
 \cI_\alpha(y) = \bigg [ \min_{\substack{x \in \mathcal D \\ x \geq 0}} \, x ,\; \max_{\substack{x \in \mathcal D \\ x \geq 0}} \, x \bigg].
\end{equation*}
For a numerical comparison of this interval with alternative methods, we refer the reader to \Cref{subsec:1d_numerics}.

\myparagraph{Unconstrained Gaussian linear model}
Consider the following problem setup:
\begin{equation}\label{eq:lineargaussianmodel-unconstrained}
\underbrace{\by = \bK\bx^* + \bm{\varepsilon}, \quad \bm{\varepsilon} \sim \mathcal{N}(\bm{0}, \bI_m)}_{\text{model}} \quad \text{and} \quad \underbrace{\varphi(\bx) = \bh^\tp \bx}_{\text{functional}}.
\end{equation}
Assume $\bK \in \mathbb{R}^{m \times p}$ has full column rank.
The assumption $\text{Cov}(\by) = \bI_m$ is without loss of generality as it is equivalent to assuming a known positive definite covariance for $\by$ and performing a basis change with the Cholesky factor. 
Note that this setup is the same as \eqref{eq:lineargaussianmodel} but the parameter space is not constrained, that is, $\mathcal X = \mathbb{R}^p$, and the forward model $\bK$ is assumed to be full rank. 

Using the framework established in \Cref{subsec:LRT}, our aim is to invert the following family of hypothesis tests:
\begin{equation} \label{eq:hyp_unconstrained}
    H_0: \bh^\tp \bx^* = \mu \quad \text{versus} \quad H_1: \bh^\tp \bx^* \neq \mu.
\end{equation}
The LLR as defined in \eqref{eq:log_likelihood_ratio_full} for the test \eqref{eq:hyp_unconstrained} takes the form:
\begin{equation} \label{eq:llr_full_rank_unconstrain}
    \lambda(\mu, \by) ~:= \min_{\bx\,:\, \bh^\tp \bx = \mu} \lVert \by - \bK \bx \rVert_2^2 - \min_{\bx} \lVert \by - \bK \bx \rVert_2^2.
\end{equation}
In this particular scenario, the LLR admits a closed-form expression and has a straightforward distribution, as formalized below:
\begin{example} 
[Distribution of the LLR statistic for the unconstrained Gaussian linear model]
\label{lemma:unconstrchi2}
$\lambda(\mu, \by)$ for the unconstrained full column rank Gaussian linear model \eqref{eq:llr_full_rank_unconstrain} can be expressed in closed form as 
\begin{equation} \label{eq:test_stat_eq_unconstr}
\lambda(\mu, \by) = \frac{(\bh^\tp (\bK^\tp \bK)^{-1} \bK^\tp \by - \mu)^2}{\bh^\tp (\bK^\tp \bK)^{-1} \bh}.
\end{equation}
Furthermore, for any $\bx^*$, whenever $\by \sim \mathcal{N}(\bK\bx^*, \bI_m)$, $\lambda(\bh^\tp\bx^*, \by)$ is distributed as a chi-squared distribution with $1$ degree of freedom.
\end{example}
\begin{proof}
    See \Cref{proof:unconstrchi2}.
\end{proof}

Using the above results, we can set $q_{\alpha}(\mu) = Q_{\chi^2_1}(1-\alpha)$ for all values of $\mu$. 
Here, $Q_{\chi^2_1}$ represents the quantile function of a chi-squared distribution with $1$ degree of freedom.  
Consequently, we can express the interval in \eqref{intervaldef} as:
\begin{equation} \label{int:unconstrained_full_rank_opt}
    \cI_\alpha(\by) = \bigg[\min_{\bx \in \mathcal{D}(\by)} \bh^\tp\bx,\; \max_{\bx \in \mathcal{D}(\by)} \bh^\tp\bx \bigg],
\end{equation}
where we define $\mathcal{D}(\by) := \{ \bx \colon \lVert \by - \bK \bx \rVert_2^2 \leq  Q_{\chi^2_1}(1-\alpha) + \min_{\bx'} \lVert \by - \bK \bx' \rVert_2^2 \}$.
Similarly, let us define $z_{\alpha} = \Phi^{-1}(1 - \alpha)$, where $\Phi$ is the cumulative distribution function of the standard normal distribution. 
Using the equivalence $z^2_{\alpha/2} = Q_{\chi^2_1}(1-\alpha)$, we can rewrite the expression in terms of the standard normal.
Moreover, as shown in Appendix A of \cite{patil}, the endpoints of the above interval can be calculated in closed form and are given by:
\begin{equation} \label{int:unconstrained_full_rank}
    \cI_\alpha(\by) = \bigg[\bh^\tp \hat{\bx} - z_{\alpha / 2} \sqrt{\bh^\tp \left(\bK^\tp \bK \right)^{-1} \bh},\; \bh^\tp \hat{\bx} + z_{\alpha / 2} \sqrt{\bh^\tp \left(\bK^\tp \bK \right)^{-1} \bh} \bigg],
\end{equation}
where we define the least-squares estimator $\hat{\bx}= (\bK^\tp \bK)^{-1} \bK^\tp \by$. 
This interval is equivalent to the one derived from observing that $\hat{\bx} \sim \mathcal{N}(\bx^*, (\bK^\tp \bK)^{-1})$. 
Therefore, we have $\bh^\tp\hat{\bx} \sim \mathcal{N}(\bh^\tp\bx^*, \bh^\tp (\bK^\tp \bK)^{-1} \bh)$. The interval in \eqref{int:unconstrained_full_rank} is thus a standard construction of a Gaussian $1 - \alpha$ confidence interval. 
Our construction therefore coincides with the classical interval in this case where a guaranteed-coverage interval can be obtained with standard manipulations; however, our framework remains valid in constrained, rank-deficient, non-Gaussian and/or nonlinear problems where few alternative approaches are available.

\section{General interval construction methodology}
\label{sec:interval_methodology}

In this section, we outline the core practical methodology for constructing intervals derived from \Cref{thm:interval_coverage}. 
To summarize the preceding, \Cref{lem:conf_set_coverage} asserts that if we know $q_\alpha(\mu)$ satisfying \eqref{eq:combined_type_1_better}, we can invert the hypothesis test with a composite null hypothesis defined in \eqref{eq:h_test_oi_full} to yield a valid $1 - \alpha$ confidence interval.
\Cref{lemma:quantile} poses two optimization problems that, if solved, yield valid decision values. 
The optimization problems in \Cref{lemma:quantile} give rise to two approaches of increasing complexity: $(i)$ finding a single $q_\alpha$ that is valid for any $\mu$ and $(ii)$ finding valid $q_\alpha(\mu)$ dependent on $\mu \in \mathcal{\varphi(\mathcal X)}$. 
While approach $(ii)$ can lead to tighter intervals, it is usually at the cost of more complex theoretical analysis and computations, including the computational complexity of solving the optimization problems in \eqref{intervaldef}. 
In particular, when we reinterpret previously proposed optimization-based methods in \Cref{sec:rust_burrus}, we will observe that these previously proposed methods are of type $(i)$, which this work expands to accommodate type $(ii)$ generalizations. 
In this section, we briefly analyze the hardness of the optimization problems \eqref{eq:QuantileOpt1} and \eqref{eq:QuantileOpt2} by connecting them to the chance-constrained optimization literature. 
In case solving these problems is impractical, in \Cref{subsec:stochastic_dominance}, we describe using stochastic dominance as a theoretical tool that can be used to create and analyze provable upper bounds to the optimization problems. 
Stochastic dominance will also be used in \Cref{sec:rust_burrus} as the main technique to disprove coverage of the previously proposed Rust--Burrus intervals. This section is summarized in a meta-algorithm, detailed in \Cref{subsec:algorithm}.

\subsection[Computational ways to compute quantile levels via optimization]{Maximum quantile problems as chance-constrained optimization}\label{subsec:CCO}

\Cref{lemma:quantile} presents optimization problems for finding the maximum quantiles, which are crucial for the proposed hypothesis test inversion procedure. 
We show that these problems, of the form $\sup_{\bx \in  \Phi_\mu \cap \mathcal{X}} Q_{F_\bx}(1-\alpha)$ and $\sup_{\bx \in \mathcal{X}} Q_{F_\bx}(1-\alpha)$, can be formulated as chance-constrained optimization (CCO) problems (see \cite{Geng2019} for a review of theory and applications of CCO).
\begin{lemma}[Chance constrained characterization of the max quantile problems]\label{lemma:cco} 
Let $\mathcal S \subseteq \mathcal X$. 
Then the max quantile optimization problem $\sup_{\bx \in \mathcal{S}} Q_{F_\bx}(1-\alpha)$ can be equivalently written as the chance constrained optimization problem:
\begin{equation}\label{eq:cco_problem}
    \begin{aligned}
    \sup_{q, \bx} \quad & q \\
    \st \quad & \bx \in \mathcal{S} \\
    \quad & q \in \mathbb{R} \\
    \quad & \mathbb{P}_{u\sim \mathcal{U}([0,1])}(\mathcal{F}(\bx,u) \leq q) \leq 1-\alpha  \\ 
    \end{aligned}    
\end{equation}
where $\mathcal{F}(\bx, u) = F^{-1}_\bx(u)$, with $F^{-1}_\bx$ being the (possibly generalized) inverse CDF of $F_\bx$
\end{lemma}
\begin{proof}
    By using the definition of $(1-\alpha)$-quantile of $X$ as the maximum $q$ such that $\mathbb{P}(X\leq q) \leq 1-\alpha$, and $\lambda \sim F_\bx \overset{d}{=} \mathcal{F}(\bx, u \sim \mathcal{U}([0,1])$
\end{proof}

Note that \Cref{lemma:cco} applies to both \eqref{eq:QuantileOpt1} and \eqref{eq:QuantileOpt2} by choosing appropriate $\mathcal{S}$. In general, CCO problems are known to be strongly NP-hard \cite{Geng2019}, even with convexity assumptions for $\mathcal{F}$. 
Although various algorithms exist for general chance-constrained optimization, we leave the development of an algorithm specific to this problem and comparison with the aforementioned algorithms for future work. 
In our numerical examples (see \Cref{sec:numerical-examples}), we solve these problems using gradient-free optimizers that do not exploit the chance-constrained structure but instead see the quantile function as a noisy black-box function to be optimized, with evaluations performed by estimating quantiles from large amount of samples. 
In higher dimensional scenarios, more advanced techniques tailored to the chance-constrained structure might be required. 
One can also write optimization problem \eqref{eq:cco_problem} and the interval optimization problem \eqref{intervaldef} jointly as one chance-constrained optimization problem; see \Cref{sec:joint-optimatization}.

\subsection[Analytical ways to obtain quantile levels via stochastic dominance]{Analytical ways to obtain quantile levels via stochastic dominance}
\label{subsec:stochastic_dominance}

In this subsection, we develop an analytical tool to find valid $q_\alpha$ that allows for a straightforward evaluation for any confidence level $1-\alpha \in (0, 1)$. \
We first consider the case where we aim to choose a valid $q_\alpha$ for all $\mu$. 
We propose taking $q_\alpha = Q_X(1-\alpha)$, where $Q_X$ is the quantile function of a random variable $X$ with a known, easy-to-compute distribution. 
We establish that for the resulting confidence interval to maintain a $1 - \alpha$ coverage guarantee for any $\alpha$, $X$ must stochastically dominate the random variable with distribution $F_{\bx^*}$, i.e., $\lambda(\mu^*, \by)$ where $\by$ is a random variable with distribution $P_{\bx^*}$. 
This is denoted as $X \succeq \lambda(\mu^*, \by)$ or, with slight abuse of notation, as $X \succeq F_{\bx^*}$.
Following the classical definition of stochastic dominance for real-valued random variables (see, e.g., \cite{shaked2007stochastic}), we say that $X \succeq Y$ if and only if $\mathbb{P}(X \geq z) \geq \mathbb{P}(Y \geq z)$ \footnote{One can equivalently define stochastic dominance with strict inequalities $X > z$ and $Y > z$, see \cite{shaked2007stochastic}} for all $z \in \mathbb{R}$

\begin{lemma}[Valid quantile level via stochastic dominance]
\label{lemma:SD}
$Q_X(1-\alpha)$ serves as a valid (in the sense of \eqref{eq:combined_type_1_better}) choice for $q_\alpha$ for all $\alpha$ if and only if $X \succeq \lambda(\mu^*, \by)$, where $\by \sim P_{\bx^*}$.
\end{lemma}
\begin{proof}
    See \Cref{sec:proof-subsec:stochastic_dominance}.
\end{proof}

\begin{remark}
[Partial validity of quantile levels]
\label{remark:sd}
If $X$ does not stochastically dominate $\lambda(\mu^*, \by)$, a valid $q_\alpha$ can still be identified for specific $\alpha$ levels, provided that certain conditions are met. 
Specifically, $z$ can serve as a valid $q_\alpha$ where $\alpha = 1 - F_X(z)$ and $F_X$ being the cumulative distribution function of $X$, if and only if $\mathbb{P}(X \leq z) \leq \mathbb{P}(Y \leq z)$ for some value of $z$.
\end{remark}

\begin{remark}
    [Support restriction] 
    Candidates for $X$ can be restricted to the range $[0, \infty)$ without loss of generality, as $\lambda(\mu^*, \by)$ is supported on this range by moving the mass a candidate $X$ might have in $(-\infty, 0)$ to $0$.
\end{remark}

An economic interpretation of our result is that agents with nondecreasing utility functions would prefer a reward drawn from $X$ over one from $\lambda(\mu^*, \by)$. 
In practical scenarios where the true parameter $\bx^*$ is unknown, it is required to establish stochastic dominance for the entire family of distributions $F_{\bx}$, where $\bx \in \mathcal X$.

Although all stochastically dominant distributions provide correct coverage when used to obtain $q_\alpha$, a larger stochastic dominance gap provides more conservative bounds. 
Furthermore, if $X_1, X_2$ both stochastically dominate the family $F_{\bx}$ for all $\bx \in \mathcal X$, we can take the pointwise minimum $q_\alpha = \min\{ Q_{X_1}(1-\alpha), Q_{X_2}(1-\alpha) \}$ which will be no worse than using either $X_1$~or~$X_2$.

The perspective of stochastic dominance also enables the use of coupling arguments to identify stochastically dominating distributions. 
For instance, one approach to find stochastically dominating distributions to a given $F_{\bx}$ is finding a function $g(\varphi(\bx), \by)$ such that for all $z$,
\begin{equation*}
\mathbb P(g(\varphi(\bx), \by) \geq z) \geq \mathbb {P}(\lambda(\varphi(\bx), \by) \geq z),
\end{equation*}
where the randomness is from $\by \sim P_{\bx}$.
A particular case is that of nonrandom bounds. 
If $g(\varphi(\bx), \by) \geq \lambda(\varphi(\bx), \by)$ almost surely in $\by$ (as opposed to when $\by \sim P_{\bx}$), then this implies a coupling of random variables once $\by$ is sampled that implies stochastic dominance (see e.g. Theorem~4.2.3 in \cite{ProbabilityBook}).

This technique can be generalized to find $q_\alpha(\mu)$. 
Instead of finding a stochastic dominant variable $X$ such that $X \succeq F_\bx$ for all $\bx \in \mathcal{X}$, we aim to find a distribution $X_\mu$ for each $\mu$, such that $X_\mu \succeq F_\bx$ for all $\bx \in \Phi_\mu \cap \mathcal{X}$, and then set $q_\alpha(\mu) = Q_{X_\mu}(1-\alpha)$. 
This ensures that $Q_{X_{\mu^*}} \succeq F_{\bx^*}$, providing the desired coverage guarantees.

\begin{figure}[!t]
    \centering
    \begin{subfigure}[b]{0.45\textwidth}
        \includegraphics[width=\textwidth]{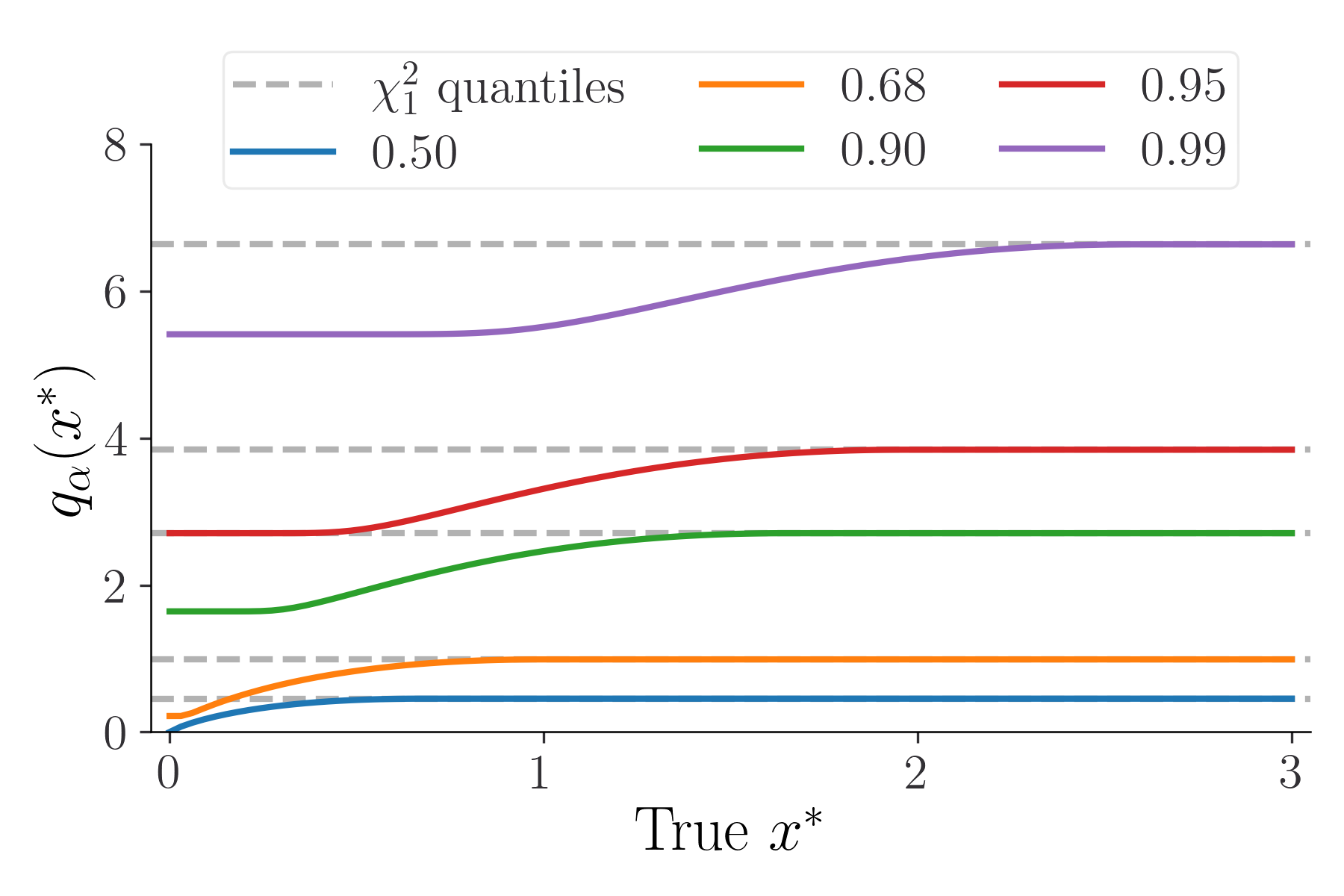}
        \caption{}
        \label{fig:1d_quantiles}
    \end{subfigure}
    \begin{subfigure}[b]{0.45\textwidth}
        \includegraphics[width=0.95\textwidth]{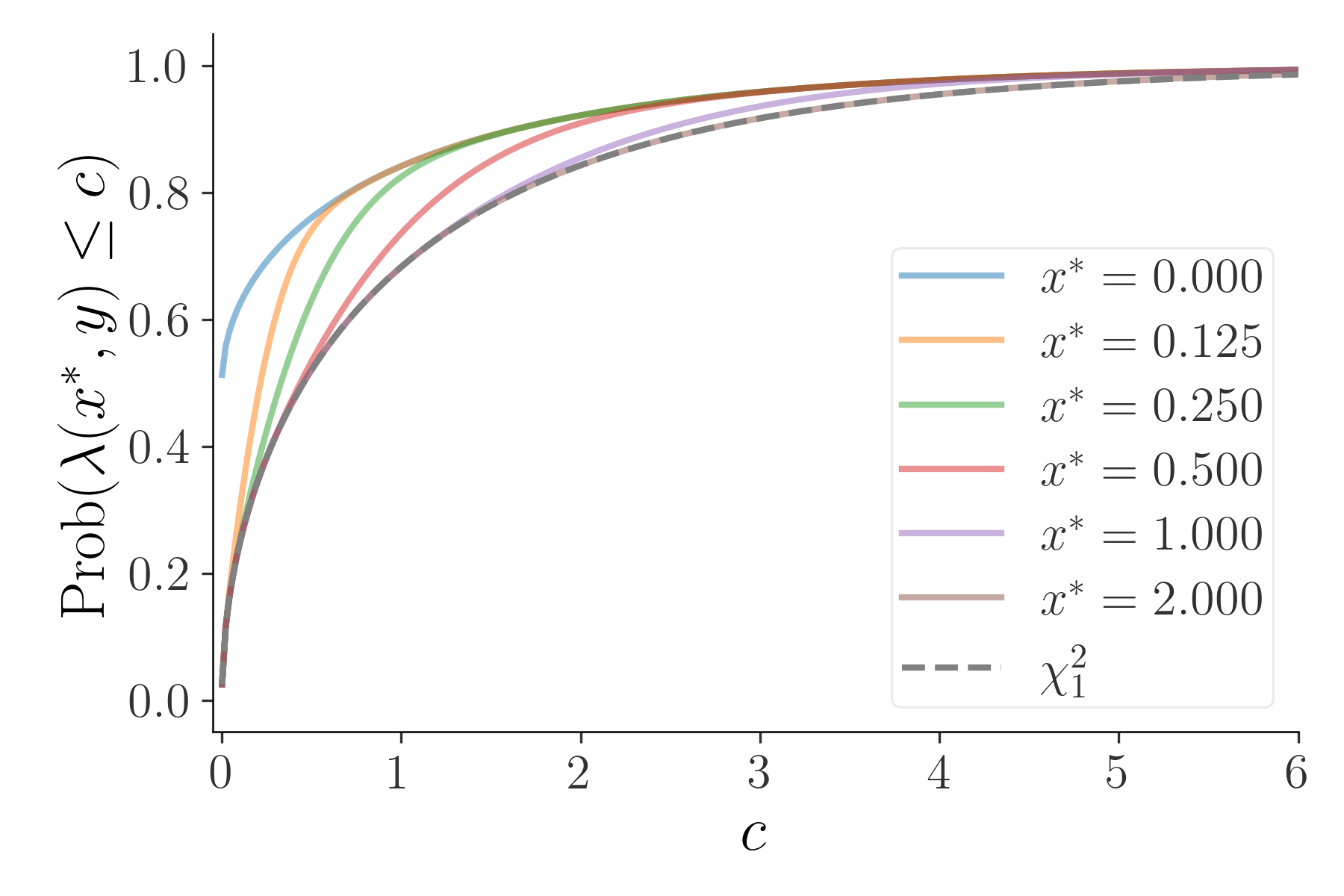}
        \caption{}
        \label{fig:1d_cdfs}
    \end{subfigure}
    \caption{
    Comparison of quantile functions and CDFs for LLRs with different true parameter values.
    The \textbf{left} panel provides the true values of the quantile function as a function of $x^*$ across different confidence levels. 
    As proven in \Cref{lemma:1d_cdf}, the quantile of $\chi^2_1$ is greater than the true quantile for all $x^*$ and all levels. 
    The \textbf{right} panel shows the CDFs for LLRs under different values of the true parameters, $x^*$. 
    From \Cref{lemma:1d_cdf}, as the true parameter increases, the CDF is increasingly dominated by its $\chi^2_1$ component, so it follows that as $x^*$ increases, the CDF approaches the $\chi^2_1$ CDF. 
    This figure also provides a visual explanation of why using the true quantile or the true quantile function to compute the interval in \eqref{intervaldef} produces shorter intervals compared to those computed with the $\chi^2_1$ quantile.
    }
    \label{fig:quantile-comparison}
\end{figure}

As an illustration, we revisit the one-dimensional constrained example discussed in \Cref{subsec:examples_theorem}. 
We consider the model $y = x^* + \varepsilon$, where $\varepsilon \sim \mathcal{N}(0,1)$, $x^* \geq 0$, and $\varphi(x) = x$. 
We recall that we have $\lambda(\mu, y) = (y-\mu)^2 - \bm{1}(y < 0)y^2$ and an analytical solution for the quantile of the distribution of $\lambda(\mu, y)$ for every $\mu$, which we can use as valid $q_\alpha$. 
We extend the results in \Cref{lemma:1d_cdf_part1} below to prove that the distribution of $\lambda(\mu, y)$ is stochastically dominated by a $\chi^2_1$.
See \Cref{fig:quantile-comparison} for an illustration.

\begin{example}
    [Stochastic dominance for LLR for constrained Gaussian in one dimension]
     \label{lemma:1d_cdf}
    For the LLR $\lambda(\mu, y)$, when $y \sim \mathcal{N}(\mu, 1)$ under the null hypothesis, we have that, for $Z \sim \chi^2_1$, $Z \succeq \lambda(\mu, y)$ for all $\mu \geq 0$.
\end{example}
\begin{proof}
   See \Cref{app:constrained_1d_gaussian}.
\end{proof}

For this example, given the stochastic dominance result, we can define $1-\alpha$ confidence intervals using $\chi^2_{1, 1-\alpha}$ instead of using $q_\alpha(\mu)$. 
This produces larger intervals than using the true quantile, but the true quantile in the closed form will generally be unavailable in more complex examples, while the presented stochastic dominance tools can still be used. 
The intervals using the $\chi^2_1$ quantile are:
\begin{equation}\label{opt4new1d_quantile_sec3}
\cI_\alpha(y) 
:=~
\begin{aligned}
\min_{x}/\max_{x} \quad & x \\
\st \quad & x \geq 0\\
& (x-y)^2 \leq \chi^2_{1, 1-\alpha} + \min_{x' \geq 0} ~ (x'-y)^2.
\end{aligned}
\end{equation}

\subsection{General confidence interval construction}
\label{subsec:algorithm}

In this section, we present our meta-algorithm that uses the methodologies described in the preceding sections. 
The goal of this meta-algorithm is to construct a $1-\alpha$ confidence interval for a given quantity of interest $\varphi(\bx^*)$. 
The algorithmic steps are outlined in \Cref{alg:metaalgo}.

\begin{algorithm}[!ht]
\caption{Meta-algorithm for confidence interval construction}
\label{alg:metaalgo}
\begin{algorithmic}[1]
\REQUIRE Observed data $\by$, log-likelihood model $\ell_\bx(\by)$, quantity of interest functional $\varphi$, constraint set $\mathcal X$, miscoverage level $\alpha$.
\vspace{0.5em}
\STATE \textbf{Test statistic}: Write down the LLR test statistic
\begin{equation}
\lambda(\mu, \by) ~= \underset{\bx \in \Phi_\mu \cap \mathcal{X}}{\text{inf}} \; -2\ell_{\bx}(\by)- \underset{\bx \in \mathcal{X}}{\text{inf}} \; -2 \ell_{\bx}(\by).
\end{equation}
\STATE \textbf{Distribution control}: Control $F_{\bx}$, the distribution of $\lambda(\varphi(\bx), \by)$ where $\by \sim P_{\bx}$, for all $\bx \in \mathcal X$, by either:
\begin{enumerate}[leftmargin=7mm,label=\Alph*.]
    \itemsep0em 
    \item \emph{Explicit solution}: 
    Obtain $F_\bx$ explicitly, and let $q_\alpha(\mu) := \sup_{\bx \in \Phi_\mu \cap \mathcal{X}} Q_{F_\bx}(1-\alpha)$.
    \item \emph{Computational way to directly find valid $q_\alpha$} (\Cref{subsec:CCO}): Solve $\sup_{\bx \in \Phi_\mu \cap \mathcal{X}} Q_{F_\bx}(1-\alpha)$ (to set $q_\alpha(\mu)$) or $\sup_{\bx \in \mathcal{X}} Q_{F_\bx}(1-\alpha)$ (to set $q_\alpha$) numerically.
    \item \emph{Analytical way using stochastic dominance} (\Cref{subsec:stochastic_dominance}): 
    Construct a distribution $X$ that stochastically dominates $F_{\bx}$ for all $\bx \in \mathcal X$, and let $q_\alpha := Q_X(1-\alpha)$, or construct distributions $X_\mu$ that stochastically dominate $F_{\bx}$ for all $\bx \in \Phi_\mu \cap \mathcal{X}$ and let $q_\alpha(\mu) := Q_{X_\mu}(1-\alpha)$.
\end{enumerate}
\STATE \textbf{Confidence interval calculation}: Obtain the confidence intervals by solving the pair of optimization problems that is easier in the particular case:
\begin{itemize}[leftmargin=7mm]
    \item[I.]
    Parameter space formulation:
            \begin{equation}\label{optA}
            \begin{aligned}
            \min_{\bx}/\max_{\bx} \quad & \varphi(\bx) \\
            \st \quad & \bx \in \mathcal X\\
            \quad & -2\ell_{\bx}(\by) \leq q_\alpha(\varphi(\bx))+ \inf_{\bx' \in \mathcal X} -2\ell_{\bx'}(\by). 
            \end{aligned}
            \end{equation}
    \item[II.]
    Functional space formulation:
            \begin{equation}\label{optB}
            \begin{aligned}
            \min_\mu/\max_{\mu} \quad &\mu \\
            \st \quad & \mu \in \varphi(\mathcal{X}) \subseteq \mathbb{R} \\ \quad & \underset{\bx \in \Phi_\mu \cap \mathcal{X}} {\text{inf}} \; -2\ell_{\bx}(\by)- \underset{\bx \in \mathcal{X}}{\text{inf}} \; -2 \ell_{\bx}(\by)\leq q_\alpha(\mu).
            \end{aligned}
            \end{equation}
\end{itemize}

\ENSURE Confidence interval with coverage $1-\alpha$.
\end{algorithmic}
\end{algorithm}

It is worth noting that the optimization problems defined in \eqref{optA} and \eqref{optB} may not always be convex or straightforward to solve. 
However, their dual formulations can be constructed, offering provably valid confidence intervals for any feasible dual solution \cite{stark1992inference}.
We defer the exploration of specialized optimization techniques specifically tailored to solve \eqref{optA} and \eqref{optB} to future work.

\section{Refuting the \rustburrus conjecture}\label{sec:rust_burrus}

As discussed in \Cref{sec:intro}, the family of constrained problems that has received the most attention is the positivity-constrained version of the problem as described in \Cref{subsec:examples_theorem}. 
To recap, the model is defined as follows:
\begin{equation}\label{eq:BurrusConjectureModel} 
\by \sim \mathcal{N}(\bK \bx^*, \bI_m) \quad \text{with} \quad \mathcal{X} = \{\bx: \bx \geq \bm{0}\} \quad \text{and} \quad \varphi(\bx) = \bh^\tp \bx.
\end{equation}
Here $\bK \in \mathbb{R}^{m \times p}$ is the forward linear operator.
We again emphasize here that $\bK$ need not have full column rank, so we can for example have $p > m$.
It was initially conjectured in \cite{burrus1965utilization,rust_burrus} that a valid $1-\alpha$ confidence interval could be obtained as
\begin{equation}\label{opt4burrus}
\begin{aligned}
\min_{\bx}/\max_{\bx} \quad & \bh^\tp\bx \\
\st \quad & \Vert \by - \bK\bx \Vert_2^2 \leq \psi^2_\alpha\\
  \quad & \bx \geq \bm{0}.
\end{aligned}
\end{equation}
Here, $\psi^2_\alpha = z_{\alpha/2}^2 + s^2(\by)$, with $z_{\alpha/2}$ being the previously defined standard Gaussian quantile, and $s^2(\by)$ is defined as the optimal value of
\begin{equation*}
\begin{aligned}
\min_{\bx} \quad & \Vert \by - \bK\bx\Vert^2_2 \\ 
\st \quad & \bx \geq \bm{0}.
\end{aligned}
\end{equation*}
Although initially believed to be proved in \cite{rust1994confidence}, an error in the proof was later identified in \cite{tenorio2007confidence}, along with a counterexample. 
However, we demonstrate that this counterexample actually satisfies the conjecture, leaving the conjecture unresolved until now prior to our work, to the best of our knowledge.

The main result of this section is the construction of a new valid counterexample using the test inversion perspective developed in \Cref{sec:data_gen_test_def_inv_arg} and the stochastic dominance approach of \Cref{subsec:stochastic_dominance}, disproving the conjecture. 

\begin{theorem}
[Refutation of the \rustburrus conjecture]
\label{thm:burrus_false}
    The \rustburrus conjecture is false in general. 
    The two-dimensional example previously proposed of a particular instance of \eqref{eq:BurrusConjectureModel} in \cite{tenorio2007confidence},
    \begin{equation*}
    \bK = \bI_2 \quad \text{and} \quad \bh = (1,-1)^\tp \quad \text{with} \quad \bx^* = (a, a)^\tp \text{ such that } a \geq 0,
    \end{equation*}
    does not constitute a valid counterexample to the \rustburrus conjecture. 
    However, the following constitutes a valid counterexample for the \rustburrus conjecture:
  \begin{equation*}
    \bK = \bI_3 \quad \text{and} \quad \bh = (1,1,-1)^\tp \quad \text{with} \quad \bx^* = (0,0,1)^\tp.
    \end{equation*}
\end{theorem}

The main idea of the proof is to first connect the conjecture to our framework, identifying the conjectured intervals as a particular case of our construction with a particular choice of $q_\alpha$. 
We then apply \Cref{thm:interval_coverage_converse} to show that coverage is equivalent to a valid choice of $q_\alpha$. 
Finally, we present a counterexample to prove that the proposed $q_\alpha$ is not universally valid.
The proof is divided into several lemmas for clarity.

Our approach is novel in that it diverges from previous geometric perspectives on the Gaussian likelihood \cite{rust_burrus,rust1994confidence, oleary_rust}, instead leveraging the test inversion and stochastic dominance perspectives developed in \Cref{sec:data_gen_test_def_inv_arg} and \Cref{subsec:stochastic_dominance}.

\subsection{Proof outline of \Cref{thm:burrus_false}}

This subsection provides a structured outline of the proof for \Cref{thm:burrus_false}, which refutes the \rustburrus conjecture.
We break down the proof into several key lemmas.

\begin{lemma}
[Framing the \rustburrus conjecture as test inversion]
\label{lemma:burruschi2}
The construction of intervals in \eqref{opt4burrus} for a particular instance of the problem $(\bx^*, \bK, \bh)$ is equivalent to the general construction in \Cref{thm:interval_coverage} for the model $\by \sim \mathcal{N}(\bK \bx^*, \bI_m)$, with $\bx^* \geq \bm{0}$ component wise, and $\varphi(\bx) = \bh^\tp \bx$, using the threshold $q_\alpha(\mu) = z_{\alpha/2}^2$ independent of $\mu$. 
Therefore, it is equivalent to inverting a hypothesis test $H_0: \bh^\tp \bx = \mu \text{ versus } H_1: \bh^\tp \bx \neq \mu$ with LLR
\begin{equation}\label{eq:llrburrus}
\lambda(\mu, \by) ~:= \min_{\substack{\bh^\tp \bx = \mu  \\ \bx \geq \bm{0}}} \lVert \by - \bK \bx \rVert_2^2 - \min_{\bx \geq \bm{0}} \lVert \by - \bK \bx \rVert_2^2.
\end{equation}
Furthermore, the interval has correct coverage if and only if $q_\alpha=z^2_{\alpha/2}$ is valid in the sense of satisfying the false positive guarantee \eqref{eq:combined_type_1}.
\end{lemma}
\begin{proof}
    See \Cref{sec:proof:lemma:burruschi2}.
\end{proof}

\begin{lemma}
[Reducing the \rustburrus conjecture to stochastic dominance]
\label{lemma:burruschi2B}
The construction of intervals in \eqref{opt4burrus} has the right coverage for any $\alpha$ (and hence the conjecture holds) for a particular instance of the problem $(\bx^*, \bK, \bh)$ if and only if the log-likelihood ratio test statistic
    \begin{equation*}
    \lambda(\mu = \bh^\tp\bx^*, \by) ~:= \min_{\substack{\bh^\tp \bx = \bh^\tp\bx^*  \\ \bx \geq \bm{0}}} \lVert \by - \bK \bx \rVert_2^2 - \min_{\bx \geq \bm{0}} \lVert \by - \bK \bx \rVert_2^2
    \end{equation*}
    is stochastically dominated by a $\chi^2_1$ distribution whenever $\by \sim \mathcal N(\bK \bx^*, \bI_m)$.
\end{lemma}
\begin{proof}
    See \Cref{sec:proof:lemma:burruschi2B}.
\end{proof}

As an example, the constrained one-dimensional example considered in \Cref{subsec:examples_theorem} satisfies the stochastic dominance result and hence the conjecture. 
Furthermore, using \Cref{lemma:unconstrchi2}, an alternative characterization of the conjecture is the stochastic dominance of the unconstrained LLR test statistic
$\min_{\bh^\tp \bx = \bh^\tp \bx^*} \lVert \by - \bK \bx \rVert_2^2 - \min_{\bx} \lVert \by - \bK \bx \rVert_2^2$ over the constrained test statistic $\min_{\substack{\bh^\tp \bx = \bh^\tp\bx^*  \\ \bx \geq \bm{0}}} \lVert \by - \bK \bx \rVert_2^2 - \min_{\bx \geq \bm{0}} \lVert \by - \bK \bx \rVert_2^2$.

We use \Cref{lemma:burruschi2B} to prove both that the example in \cite{tenorio2007confidence} obeys the conjecture and that our new counterexample does not. 
\myparagraph{Invalidity of a previous counterexample in two dimensions}
The previously proposed counterexample from \cite{tenorio2007confidence} is a two-dimensional problem with $\bK = \bI_2$, $\bx^* = (a,a)^\tp$ with $a \geq 0$, $\bh = (1,-1)^\tp$ (and therefore $\mu^* = \bh^\tp \bx^* = 0$).
The LLR test statistic is 
\begin{equation*}
\lambda(\mu^* = 0, \by) = \min_{\substack{x_1 = x_2 \\ \bx \geq \bm{0}}} \Vert \bx -\by \Vert^2_2 - \min_{\bx \geq \bm{0} } \Vert \bx-\by \Vert^2_2 
\end{equation*}
which, after solving the optimization problems, is equal to  
\begin{equation*}
\lambda(\mu^*, \by) = 
  \begin{cases} 
     y_1^2 + y_2^2 -(y_1 - \max(y_1, 0))^2 - (y_2 - \max(y_2, 0))^2, & y_1 + y_2 < 0, \\
    \frac 1 2 (y_1-y_2)^2 -(y_1 - \max(y_1, 0))^2 - (y_2 - \max(y_2, 0))^2, &  y_1 + y_2 \geq 0,
    \end{cases}
\end{equation*}
which we can equivalently write as 
\begin{align}\label{tenorioce}
    \lambda(\mu^*, \by) =\ & (y_1^2 + y_2^2)\mathbbm{1}\{y_1+y_2 < 0\} + \frac 1 2 (y_1-y_2)^2 \mathbbm{1}\{y_1+y_2 \geq 0\} \\&- y_1^2 \mathbbm{1}\{y_1<0\} -y_2^2\mathbbm{1}\{y_2 < 0\}.\notag
\end{align}
\begin{lemma}
    [Invalidity of a previous counterexample]
    \label{lemma:invalidity-counterexample}
    The LLR statistic $\lambda(\mu^*, \by)$ in \eqref{tenorioce} is stochastically dominated by a $\chi^2_1$ random variable whenever $\by \sim \mathcal N(\bx^*, \bI_2)$, $\bx^* = (a, a)^\tp$ for $a \ge 0$, and $\bh = (1, -1)^\top$. 
    Therefore, it does not constitute a valid counterexample to the conjecture.
\end{lemma}
\begin{proof}[Proof sketch]
The proof follows from a coupling argument between the LLR and a $\chi^2_1$ random variable. 
See \Cref{proof:invalidity-counterexample} for proof details.
\end{proof}

In summary, we used \Cref{lemma:burruschi2B} to demonstrate that the previously proposed counterexample actually satisfies the conjecture.

\myparagraph{A new provably valid counterexample in three dimensions}
We now present a new counterexample in $\mathbb{R}^3$ to refute the \rustburrus conjecture. 
Specifically, we consider $\bK = \bI_3$, $\bx^* = (0,0,1)^\tp$, and $\bh = (1,1,-1)^\tp$, yielding $\mu^* = -1$. 
We prove that $\chi^2_1$ does not stochastically dominate $\lambda(\mu^*, \by)$, which in this case is
\begin{equation}\label{ource}
\lambda(\mu^* = -1, \by) ~= \min_{\substack{x_1 + x_2 -x_3 = -1\\ \bx \geq \bm{0}}} \Vert \bx -\by \Vert^2_2 - \min_{\bx \geq \bm{0} } \Vert \bx-\by \Vert^2_2.
\end{equation}
We prove that $\mathbb{E}[\lambda(\mu^*, \by)] > \mathbb{E}[\chi^2_1] = 1$. 
Here, the expectation is taken with respect to $\by \sim \mathcal N(\bx^*, \bI_3)$, and the inequality is a general sufficient condition to refute stochastic dominance and hence for the conjecture to break.
\begin{lemma}
    [Validity of a new counterexample]
    \label{lem:3d-counterexample}
    $\lambda(\mu^*, \by)$ in \eqref{ource} is \emph{not} stochastically dominated by a $\chi^2_1$ random variable whenever $\by \sim \mathcal N(\bx^*, \bI_3)$ with $\bx^* = (0,0,1)^\tp$. 
    Therefore, it constitutes a valid counterexample to the general conjecture.
\end{lemma}
\begin{proof}[Proof sketch]
We compute the expected value and show that it is greater than 1 (the expected value of a $\chi^2_1$), therefore proving stochastic dominance. 
See \Cref{ProofCounter} for the proof details.
\end{proof}

\begin{remark}
    [A more general counterexample]
    The validity of the counterexample does not hinge on $\bx^*$ being on the boundary of the constraint set. 
    In fact, the example remains valid for $\bx^* = (\varepsilon, \varepsilon, 1)^\tp$ with $\varepsilon > 0$ sufficiently small. 
    We choose $\varepsilon = 0$ for the simplicity of the proof. 
    See \Cref{fig:3d_quantiles} for numerical evidence, where the quantiles over the dashed line correspond to valid counterexamples.
\end{remark}

\Cref{fig:sdominance} shows the difference between the two examples. 
By plotting the difference between the CDF of $\lambda$ (obtained numerically with $N = 10^6$ samples) and the CDF of a $\chi^2_1$ distribution, we observe stochastic dominance for the two-dimensional example in \Cref{fig:sdominance} (left panel) and no stochastic dominance (hence breaking of the conjecture) for the three-dimensional example in \Cref{fig:sdominance} (right panel). \Cref{sec:numerical-examples} contains numerical coverage studies for both scenarios agreeing with the observation made here.

\begin{figure}[!t]
    \centering
    \begin{subfigure}{.5\textwidth}
      \centering
        \includegraphics[width=0.95\columnwidth]{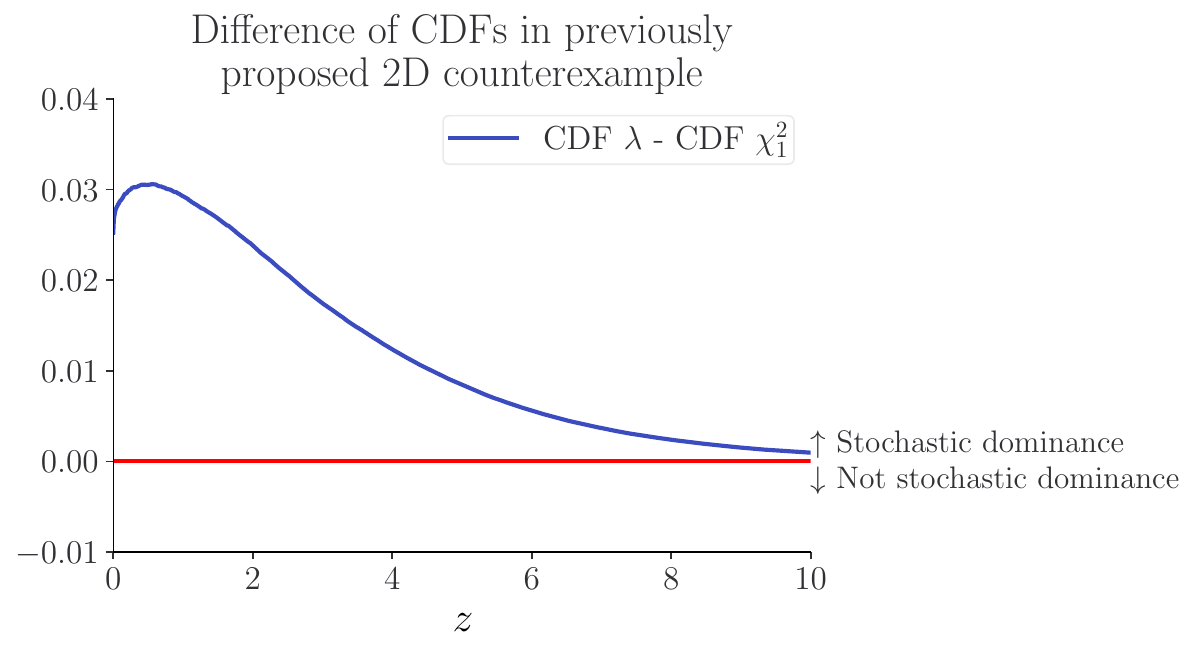}
      \label{fig:sdominanceA}
    \end{subfigure}%
    \begin{subfigure}{.5\textwidth}
      \centering
        \includegraphics[width=0.95\columnwidth]{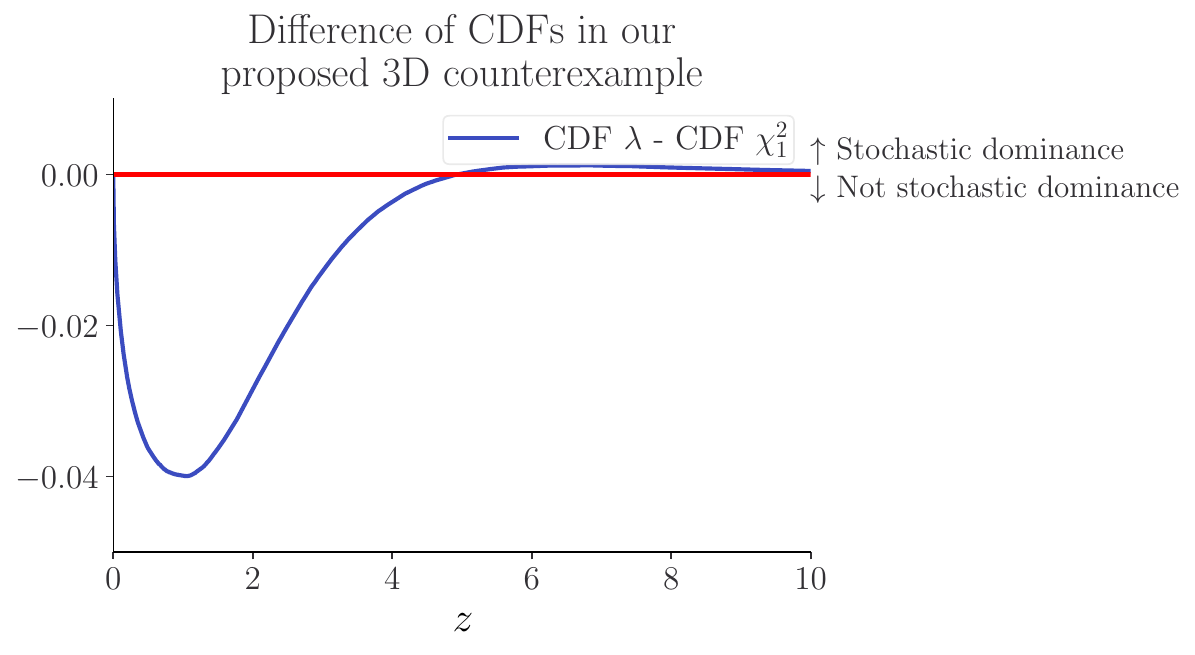}
      \label{fig:sdominanceB}
    \end{subfigure}
    \caption{
    Difference of cumulative distribution functions between the LLR test statistic and $\chi_1^2$ distribution for the statistics defined in \eqref{tenorioce} \textbf{(left)} and \eqref{ource} \textbf{(right)}. 
    Stochastic dominance, which is equivalent to the \rustburrus conjecture, is broken in the right example only. 
    There is a direct correspondence between the points at which the CDF difference is negative and confidence levels $1-\alpha$ that fail to hold (see \Cref{remark:sd}).
    }  
    \label{fig:sdominance}
\end{figure}

\subsection{A negative result in high dimensions}

After establishing that the $\chi^2_1$ distribution fails to stochastically dominate the constrained log-likelihood ratio, a natural question arises: Is there another distribution, possibly within the $\chi^2_k$ family, that can stochastically dominate the constrained LLR? 
If such a distribution exists, it would allow us to redefine $\psi^2_\alpha$ in \eqref{opt4burrus} as $s^2 + Q_X(1-\alpha)$, making the $Q_X$ term in the optimization problem dimension independent, leading to intervals with shorter length in large dimensions.
It is worth noting that in the unconstrained scenario, the LLR distribution is precisely $\chi^2_1$, regardless of the dimensionality of the problem. However, the following proposition shows that no such dimension-independent distribution exists for the constrained case.
\begin{proposition}
[A negative result in high dimensions]
\label{lemma:dimbounding}
The family of constrained LLRs for general $\bK, \bh$ in arbitrary dimensions, defined as
\begin{equation*}
\lambda(\mu = \bh^\tp\bx^*, \by) ~= \min_{\substack{\bh^\tp \bx = \bh^\tp\bx^*  \\ \bx \geq \bm{0}}} \lVert \by - \bK \bx \rVert_2^2 - \min_{\bx \geq \bm{0}} \lVert \by - \bK \bx \rVert_2^2,
\end{equation*}
cannot be stochastically dominated a dimension-independent way by any finite-mean distribution (including all $\chi^2_k$ for $k \ge 1$).
\end{proposition}
\begin{proof}[Proof sketch]
We construct a sequence of examples with increasing dimensions and demonstrate that the expected value of the constrained LLR grows unbounded as the dimension increases. 
This result negates the possibility of stochastic dominance by any finite-mean distribution. 
For a detailed proof, see \Cref{proof:dimbounding}.
\end{proof}

\section{Numerical examples}
\label{sec:numerical-examples}

In this section, we provide numerical illustrations for the procedures and theoretical results described above. 
In particular, we analyze coverage properties of four types of intervals. The first two intervals come from previous works: $\cI_{\SSB}$ \eqref{simultaneous}, which has provably correct coverage but is known to overcover when inferring a single functional, and $\cI_{\OSB}$ \eqref{RB1}, which comes from the Burrus conjecture and, as we proved in this work, does not correctly cover in general. 
The last two are the interval constructions described in this work by solving the max quantile problems \eqref{eq:QuantileOpt1} and \eqref{eq:QuantileOpt2} to find $q_\alpha(\mu)$ (\Cref{subsec:algorithm}): $\cI_{\MQ}$ (where $q_\alpha$ does not depend on $\mu$) and the more refined $\cI_{\mathsf{MQ\mu}}$ (where $q_\alpha$ depends on $\mu$). Both of these intervals have provable coverage.

We analyze four settings in which the Burrus conjecture applies, including a one-dimensional example in \Cref{subsec:1d_numerics}, the previously proposed invalid counterexample in \Cref{subsec:2d_gauss_num} (for which we prove in \Cref{lemma:invalidity-counterexample} that $\cI_{\OSB}$ correctly covers), our proposed counterexample in \Cref{subsec:3d_numerics}  (for which we prove in \Cref{lem:3d-counterexample} that $\cI_{\OSB}$ does not cover for at least some level $1-\alpha$ and a certain $\bm{x}^*$), and a setting with a bounded constraint set\footnote{Strictly speaking, this setting is not included in the original \rustburrus conjecture but is later extended in \cite{patil,stanley_unfolding}.} $\mathcal X$ in \Cref{subsec:bounded_domain}.

Throughout the examples, the observed coverage (or lack thereof) agrees with the developed theoretical results, and we observe that our interval $\cI_{\mathsf{MQ\mu}}$ consistently fixes the miscalibration of the other interval types: when $\cI_{\OSB}$ undercovers, $\cI_{\mathsf{MQ\mu}}$ is on average longer than $\cI_{\OSB}$ to obtain coverage, and when $\cI_{\OSB}$ overcovers, $\cI_{\mathsf{MQ\mu}}$ is on average shorter than $\cI_{\OSB}$ with coverage closer to the prescribed level $1-\alpha$. 

\subsection{Constrained Gaussian in one dimension} 
\label{subsec:1d_numerics}

We revisit the constrained Gaussian model in one dimension \eqref{eq:1d_toy_model} described in \Cref{subsec:examples_theorem}, $y= x^* + \varepsilon, \, \varepsilon \sim \mathcal{N}(0,1), \, x^* \geq 0$ and $\varphi(x) = x$
We perform a simulation experiment using six true parameter settings of $x^* \in \{0, 2^{-3}, 2^{-2}, 2^{-1}, 2^0, 2^1 \}$. 
We focus on settings closer to the boundary since that is where the biggest differences between the considered intervals exist. 
For each of these settings, we simulate $10^5$ observations according to the model \eqref{eq:1d_toy_model} and compute three different $95\%$ confidence intervals for each sample: interval \eqref{intervaldef} using the actual quantile function given in \eqref{eq:1d_quantile_func} ($\cI_{\MQ_\mu}$\footnote{Since this is a one-dimensional problem with $\varphi(x) = x$, this is equivalent to being able to solve all the $\mu-$dependant quantile optimization problems}), interval \eqref{intervaldef} using the stochastically dominating $Q_{\chi^2_1}(1 - \alpha)$ quantile ($\cI_{\OSB}$, which in this problem is equal to $\cI_{\MQ}$), and the standard Truncated Gaussian interval, which equals the SSB interval in this case ($\cI_{\SSB}$). 
The intervals computed with the true quantile function are characterized by:
\begin{equation}\label{opt4new1d_quantile}
\begin{aligned}
\cI_{\MQ_\mu}(y) 
:=~ 
\min_{x}/\max_{x} \quad & x \\
\st \quad & x \geq 0\\
& (x-y)^2 \leq q_\alpha(x) + \min_{x \geq 0} (x-y)^2,
\end{aligned}
\end{equation}
where $q_\alpha(x)$ is given by \eqref{eq:1d_quantile_func}.
For the stochastically dominating $Q_{\chi^2_1}(1 - \alpha)$, the interval in \eqref{intervaldef} becomes:
\begin{equation}\label{opt4new1d}
\begin{aligned}
\cI_{\OSB}(y) 
=  \min_{x}/\max_{x} \quad & x \\
\st \quad & x \geq 0\\
& (x-y)^2 \leq Q_{\chi^2_1}(1 - \alpha) + \min_{x \geq 0} (x-y)^2.
\end{aligned}
\end{equation}
Finally, the truncated Gaussian interval, which is shown below to be equivalent to the SSB interval in this case, is defined as:
\begin{equation}\label{eq:TG1D}
    \cI_{\SSB}(y) := \left[ y - z_{\alpha / 2},\; y + z_{\alpha / 2} \right] \cap \mathbb{R}_{\ge 0}.
\end{equation}
Observe that \eqref{opt4new1d} admits an explicit solution:
\begin{align}\label{eq:explicit1d}
\cI_{\OSB}(y) = 
\begin{cases}
[y - \sqrt{Q_{\chi^2_1}(1 - \alpha)},\ y + \sqrt{Q_{\chi^2_1}(1 - \alpha)} ]  \cap \mathbb{R}_{\ge 0}, & y \geq 0 \\     
[y - \sqrt{Q_{\chi^2_1}(1 - \alpha) + y^2},\; y + \sqrt{Q_{\chi^2_1}(1 - \alpha) + y^2} ]  \cap \mathbb{R}_{\ge 0}, & y < 0.
\end{cases}
\end{align}
Furthermore, note that $\sqrt{Q_{\chi^2_1}(1 - \alpha)} = z_{\alpha/2}$, so that \eqref{eq:explicit1d} is always larger than or equal to \eqref{eq:TG1D}. 
Conversely, we can express \eqref{eq:TG1D} as the solution to optimization problems, illustrating that the truncated Gaussian interval is equivalent to the SSB interval for this case:
\begin{equation}\label{TGopt1D}
\begin{aligned}
\cI_{\SSB}(y) =  \min_{x}/\max_{x} \quad & x \\
\st \quad & x \geq 0\\
& (x-y)^2 \leq z_{\alpha/2}^2.
\end{aligned}
\end{equation}
To empirically estimate coverage, for each $x^*$ setting and each interval type, we compute $10^5$ intervals and keep track of their coverage of the true parameter. 
The left panel in \Cref{fig:1d_interval_coverages_and_lengths} shows how $\cI_{\OSB}$ based on $Q_{\chi^2_1}(1 - \alpha)$ over-covers when the true parameter is on the boundary, which makes sense as this setting of $q_\alpha$ holds for all $x^*$, and therefore is a conservative quantile.
As expected, the interval computed with $q_\alpha(x)$ maintains the nominal $95\%$ coverage over all considered $x^*$ values. This shows that knowing the quantile function means that we can compute an interval with exact nominal coverage that is adaptive to the unknown true parameter value.
Additionally, we note that as $x^*$ grows, the estimated coverage values across these methods converge, illustrating the intuition that when $x^*$ gets sufficiently far from the constraint boundary, the problem is essentially unconstrained, and all considered methods produce nearly identical results. 
The right panel in \Cref{fig:1d_interval_coverages_and_lengths} shows each interval's expected length as a function of $x^*$.
Again, we observe the tightness of the interval \eqref{intervaldef} constructed with the true quantile function $q_\alpha(x)$ compared to the interval constructed with the stochastically dominating quantile, $Q_{\chi^2_1}(1 - \alpha)$. 
Similarly to coverage, as $x^*$ grows, the expected interval lengths of $\cI_{\MQ_\mu}$ and $\cI_{\OSB}$ converge, and the methods become indistinguishable. 
Also, observe that the truncated Gaussian intervals have a smaller expected length compared to the intervals computed with $q_\alpha(x)$.
We note that this length observation is particular to this one-dimensional example, as OSB intervals have been shown to be shorter than SSB intervals in higher dimensional problems \cite{oleary_rust, rust_burrus, stanley_unfolding}.
\begin{figure}[!t]
    \centering
    \includegraphics[width=0.9\columnwidth]{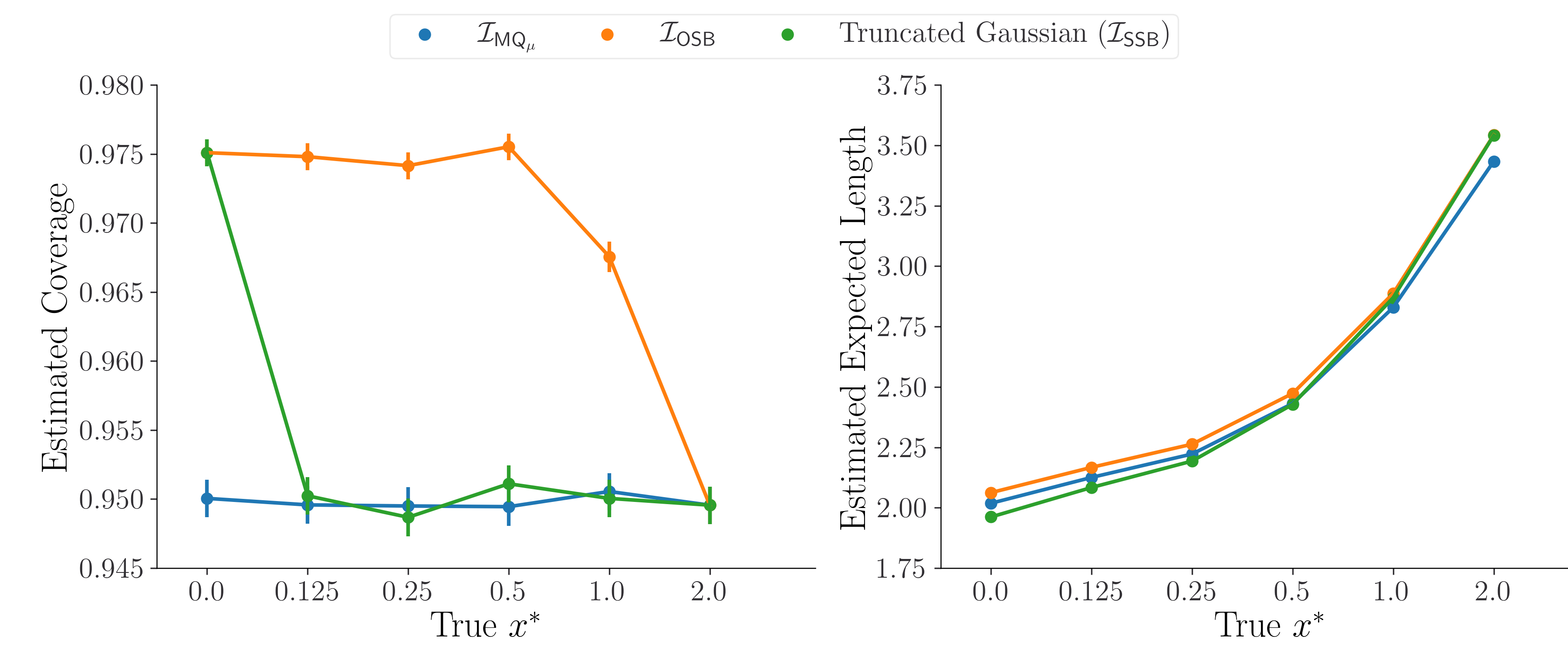}
    \caption{
    The \textbf{(left)} figure shows estimated coverage for each $95\%$ confidence interval and each true $x^*$ for the one-dimensional constrained Gaussian model. 
    Both the Truncated Gaussian (SSB) and OSB intervals overcover when $x^* = 0$ while the MQ$\mu$ interval predictably achieves nominal coverage. 
    All the coverage values converge as $x^*$ gets larger, as the problem moves toward the unconstrained problem where all the intervals are effectively the same. 
    The intervals surrounding the estimated values are $95\%$ Clopper--Pearson intervals, expressing the Monte Carlo uncertainty of each coverage estimate. 
    The \textbf{(right)} figure shows an estimate of the expected interval length for each method (with $95\%$ confidence intervals that are nearly length zero since the standard error of each estimate is nearly zero with $10^5$ realizations each). 
    Similarly to the coverage results in the left panel, as $x^*$ gets larger, the expected OSB and MQ$\mu$ interval lengths converge while the SSB intervals remain slightly larger.
    }
    \label{fig:1d_interval_coverages_and_lengths}
\end{figure}

\subsection{Constrained Gaussian in two dimensions} \label{subsec:2d_gauss_num}

We consider the Gaussian linear model in \eqref{eq:lineargaussianmodel} with $\bK = \bI_2$, $\varphi(\bx) = \bh^\tp \bx = x_1 - x_2$ and $\mathcal X = \{\bx \in \mathbb{R}^2: \bx \geq 0 \}$.
The work \cite{tenorio2007confidence} proposes this scenario as a counterexample to the \rustburrus conjecture, but as shown in \Cref{lemma:invalidity-counterexample}, it is in fact a case where the $\chi^2_1$ distribution stochastically dominates the LLR for all true $\bx^* \in \{\bx: \bh^\tp \bx = 0, \, \bx \geq 0\}$, so the OSB intervals proposed by the conjecture have provably correct coverage for $\bx^*$ in this set. 

We estimate interval coverage with $1-\alpha = 0.95$ for the four types of intervals ($\cI_{\SSB}$, $\cI_{\OSB}$, $\cI_{\MQ}$ and $\cI_{\mathsf{MQ\mu}}$) for three true parameter values, two of them inside the $\{\bx \colon \bh^\tp \bx = 0, \, \bx \geq 0\}$ region in which \Cref{lemma:invalidity-counterexample} applies, and one outside. In this and all the examples that follow, we solve the max quantile optimization problems \eqref{eq:QuantileOpt1}, \eqref{eq:QuantileOpt2} using Bayesian Optimization (see \cref{fig:2d_heatmap} for an illustration of the quantile function being optimized in this particular example), and we solve the outer optimization problems with a convex optimization package or a root-finding numerical algorithm (in the case of $\cI_{\mathsf{MQ\mu}}$, using the functional space formulation). For each parameter value and all intervals, coverage and expected length are estimated by drawing $5 \times 10^4$ observations from the data generating process, computing all interval types for each generated observation, and then checking coverage and length.  The coverage confidence intervals are $95\%$ Clopper--Pearson intervals for a binomial parameter, whereas the length confidence intervals are standard asymptotic Gaussian intervals using sample means and standard errors.

The results are shown in \cref{fig:2d_coverage_and_length}. 
We observe correct coverage for the OSB intervals, in agreement with \Cref{lemma:invalidity-counterexample} (which applies for $\bx^* = (0,0)^\top$ and $\bx^* = (0.33, 0.33)^\top$). We observe that the SSB intervals are the longest on average and tend to overcover and that the MQ and MQ$\mu$ intervals have nearly identical properties to the OSB intervals. This is because, for this problem setting, solving the optimization problems \eqref{eq:QuantileOpt1} and \eqref{eq:QuantileOpt2} recovers the $\chi^2_1$ quantile (up to the numerical precision of the optimization solvers) of the Burrus conjecture as the maximum quantile (both for all $\mathcal{X}$ and for $\bh^\top\bx =\mu$ for any $\mu$), so both of those intervals actually recover the OSB intervals in this case.

\begin{figure}[!t]
    \centering
    \includegraphics[width = 0.9\columnwidth]{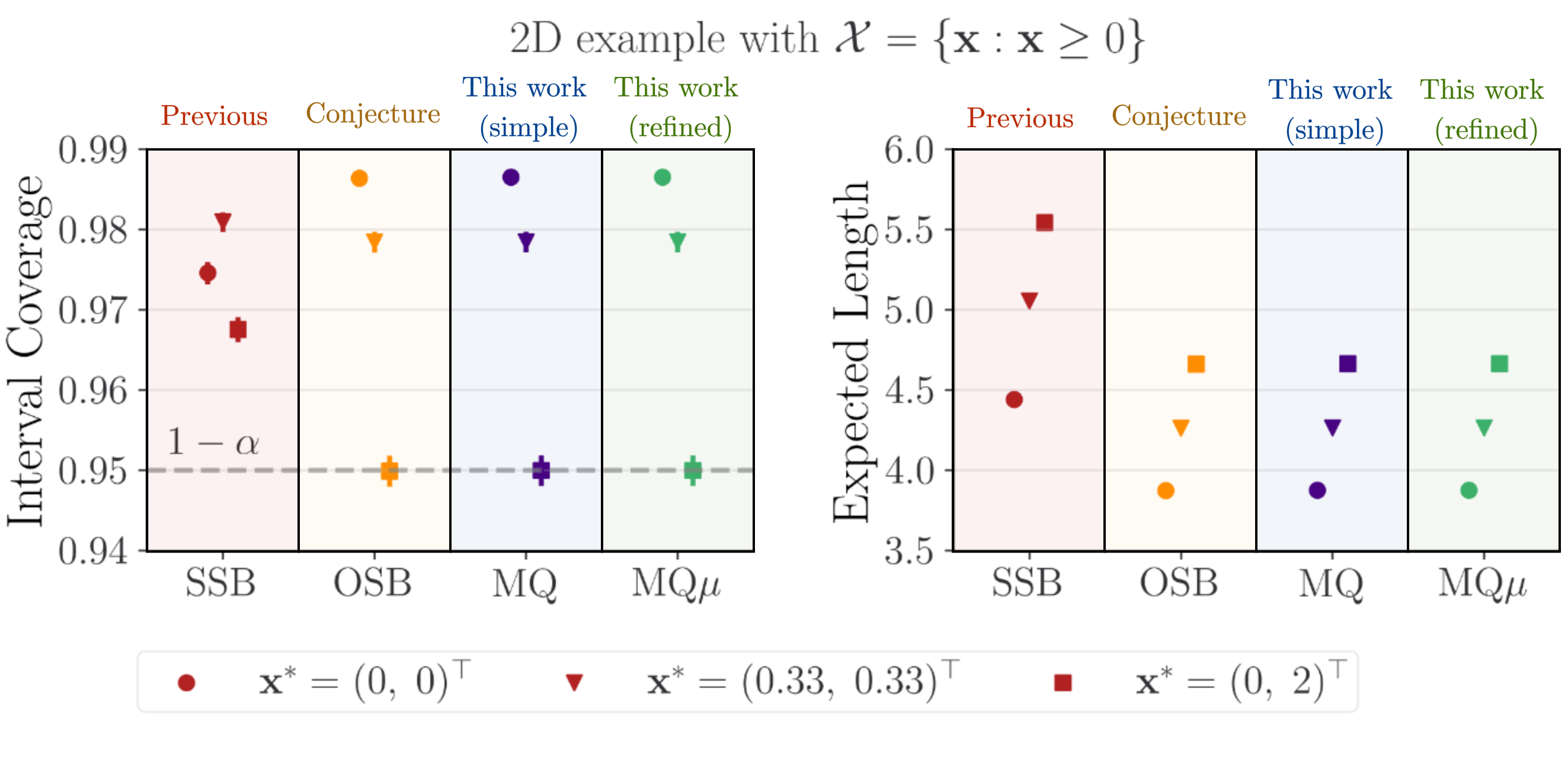}
    \caption{
    Estimated interval coverage \textbf{(left)} and expected lengths \textbf{(right)} for $95\%$ intervals resulting from the SSB, OSB, MQ and MQ$
    \mu$ methods for the Gaussian linear model in \eqref{eq:lineargaussianmodel} with $\bK = \bI_2$, $\varphi(\bx) = \bh^\tp \bx = x_1 - x_2$ and $\mathcal X = \{\bx \in \mathbb{R}^2: \bx \geq 0 \}$.
    }
    \label{fig:2d_coverage_and_length}
\end{figure}

\begin{SCfigure}[50][!t]
    \centering
    \includegraphics[width=0.35\columnwidth]{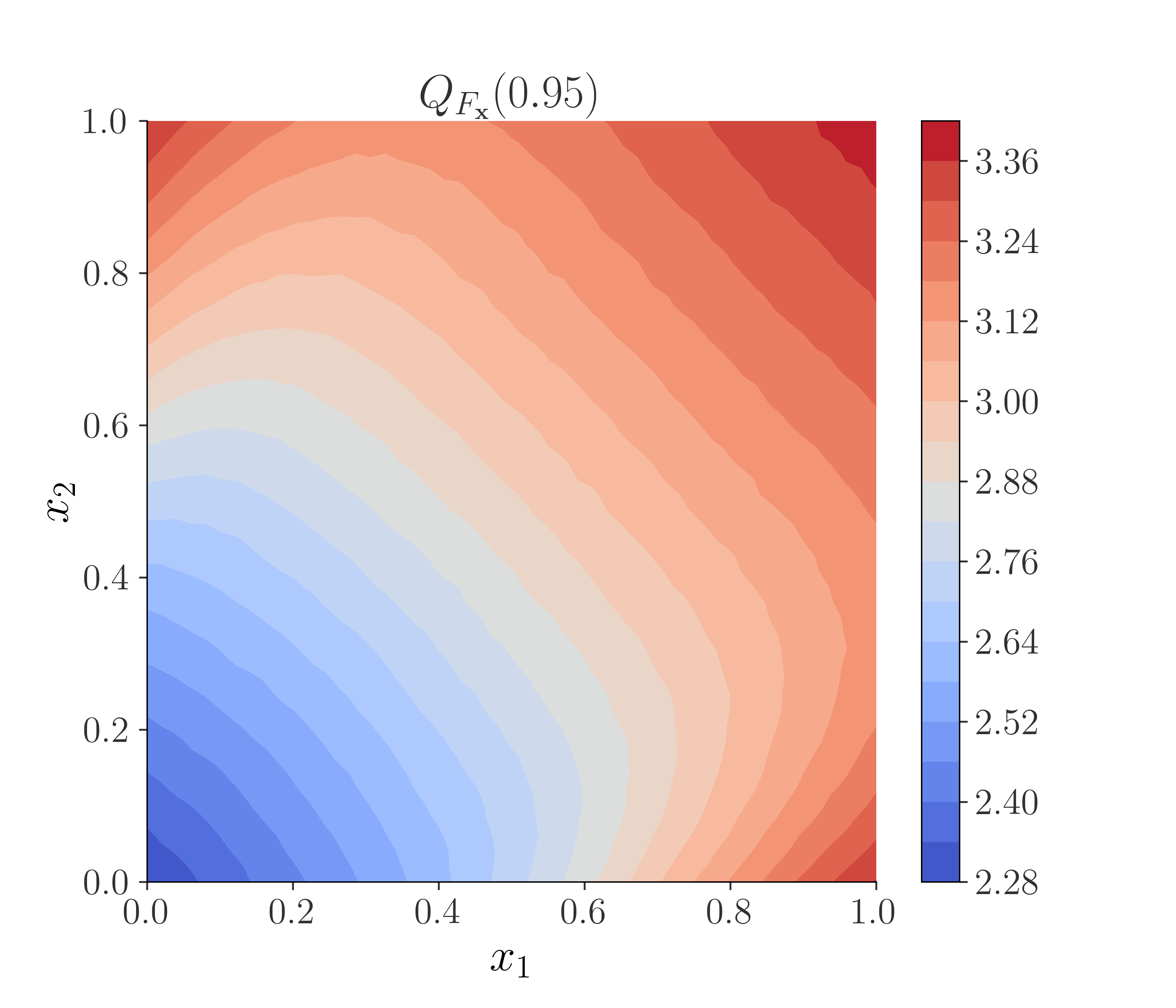}
    \caption{
    Estimated $95$th quantiles from the LLR test statistic null distributions in the region $[0, 1]^2 \subset \mathbb{R}^2_+$ where color shows the estimated quantiles. In contrast to the unconstrained case, the quantile is dependent on the true parameter, and the quantile surface is non-trivial.
    \vspace{4em}
    }
    \label{fig:2d_heatmap}
\end{SCfigure}

\subsection{Constrained Gaussian in three dimensions} \label{subsec:3d_numerics}

We use the three-dimensional counterexample of the \rustburrus conjecture from \Cref{sec:rust_burrus} to numerically show that the $\cI_\OSB$ intervals undercover in this example and that the max-quantile intervals are able to fix the undercoverage. Concretely, we consider the Gaussian linear model in \eqref{eq:lineargaussianmodel} with $\bK = \bI_3$, $\varphi(\bx) = \bh^\tp \bx = x_1 + x_2 - x_3$ and $\mathcal X = \{\bx \in \mathbb{R}^3: \bx \geq 0 \}$. We repeat the same experimental setup as in \Cref{subsec:2d_gauss_num}, comparing the four interval types for $\bx^* = (0,0,0)^\top$ and $\bx^* = (0,0,1)^\top$; this last parameter value is the one analyzed in \Cref{lem:3d-counterexample} and for which we know the OSB interval can undercover for some $\alpha$. \Cref{fig:3d_coverage_length68} shows the results for $1-\alpha = 0.68$ (one sigma interval coverage), and we include in \Cref{app:sec:numerical-examples} the results for $1-\alpha = 0.95$, which lead to the same conclusions. 
Furthermore, to illustrate that the conjecture breaks in an area around the studied point $(0,0,1)^\top$ in \Cref{fig:3d_quantiles}, we plot the numerically estimated $68\%$ LLR test statistic quantiles for parameters of the form $(t, t, 1)$, showing that there is a range of points with a larger quantile than the $\chi^2_1$ quantile, implying undercoverage.

The results agree with our theoretical findings in \Cref{lem:3d-counterexample}, as we observe that the OSB intervals undercover both at $95\%$ and $68\%$ confidence levels, thus invalidating the \rustburrus conjecture. In contrast, the SSB, MQ and MQ$\mu$ intervals all have provable coverage in this scenario, which is reflected in the estimated coverage values. Notably, the MQ$\mu$ intervals are not much longer than the OSB intervals, but they obtain the required coverage, enlarging the conjectured intervals just enough. The simpler MQ intervals overcover a bit more, which illustrates the benefit of solving \eqref{eq:QuantileOpt2} over \eqref{eq:QuantileOpt1} when computationally feasible. Nevertheless, their length is not much larger than MQ$\mu$ and significantly smaller than for SSB, the other simple method with coverage guarantees.

\begin{figure}[!t]
    \centering
    \includegraphics[width = 0.9\columnwidth]{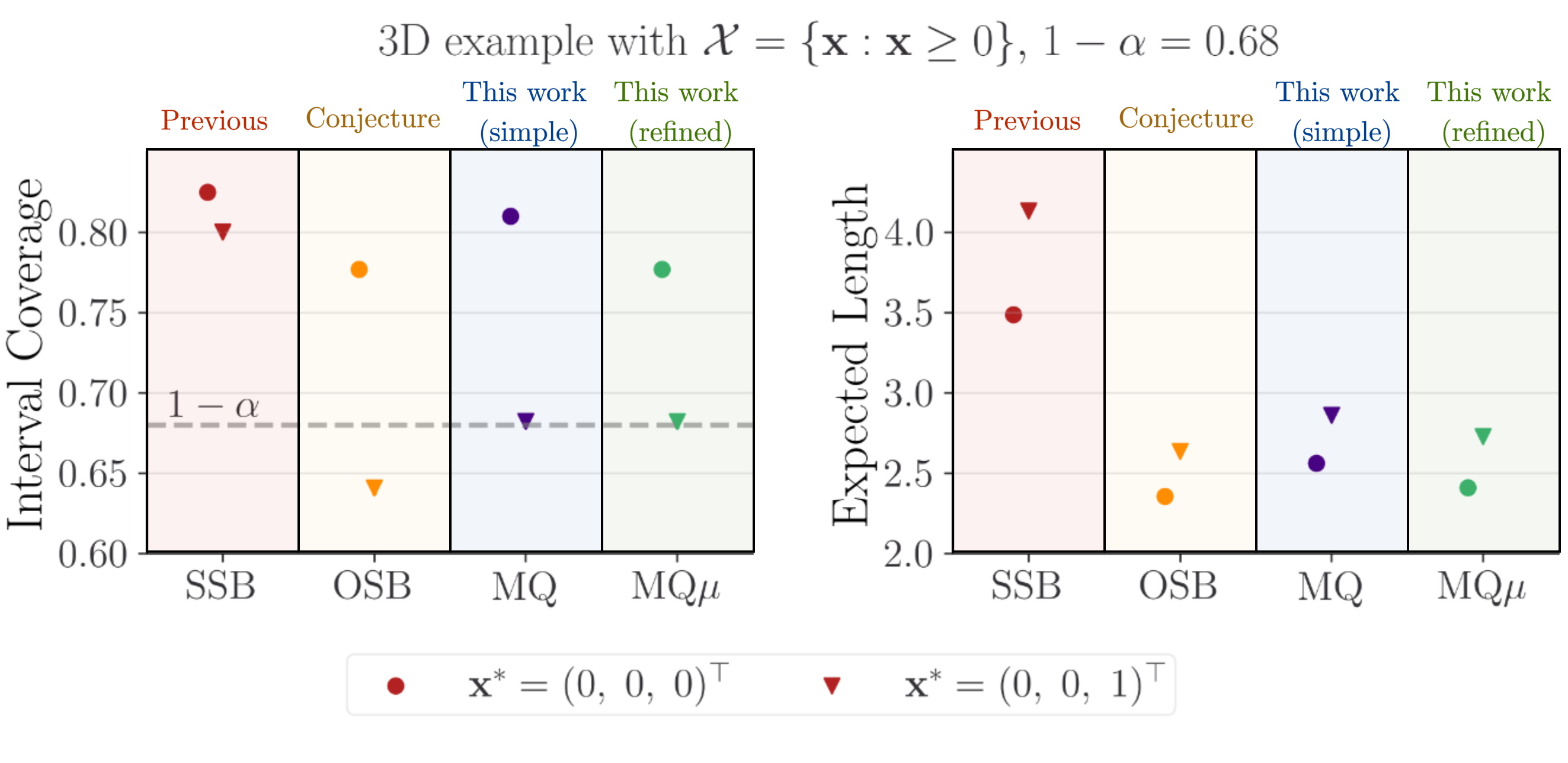}
    \caption{
    Estimated interval coverage \textbf{(left)} and expected lengths \textbf{(right)} for 68\% intervals resulting from the SSB, OSB, MQ and MQ$\mu$ methods for the Gaussian linear model in \eqref{eq:lineargaussianmodel} with $\bK = \bI_3$, $\varphi(\bx) = \bh^\tp \bx = x_1 + x_2 - x_3$ and $\mathcal X = \{\bx \in \mathbb{R}^3: \bx \geq 0 \}$.
    }
    \label{fig:3d_coverage_length68}
\end{figure}

\begin{SCfigure}[50][!t]
    \centering
    \includegraphics[width=0.49\columnwidth]{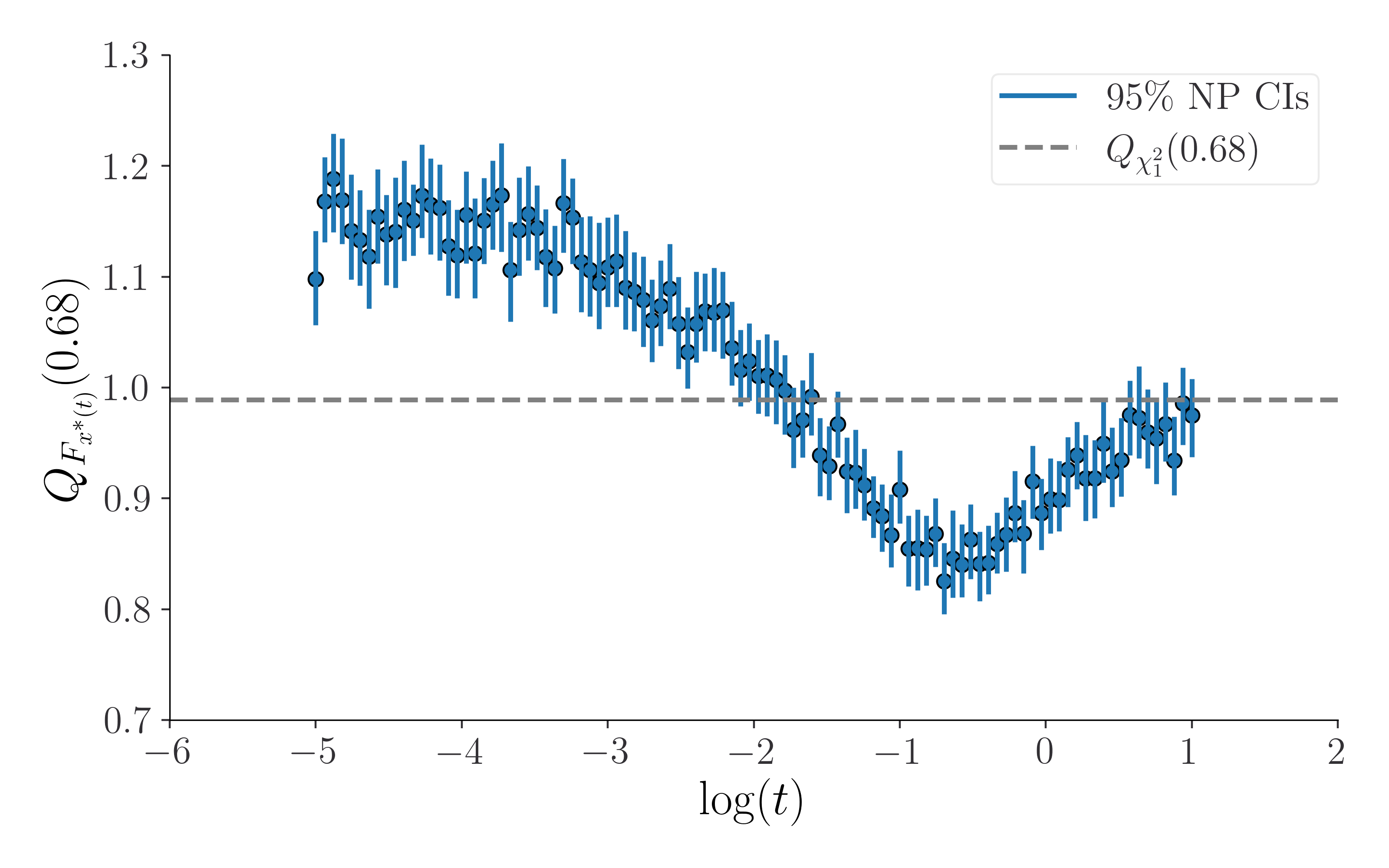}
    \caption{
    For the numerical example in \Cref{subsec:3d_numerics}, considering $\bx^* \in \left\{\bx^*(t) = (t, t, 1)^\tp: 0 < t \leq e \right\} \subset \mathbb{R}^3$, we estimate the $68\%$ LLR test statistic quantiles along with $95\%$ nonparametric (NP) confidence intervals for percentiles \cite{hahn_meeker}.
    The test statistic quantiles exceeding $\chi^2_{1, 0.32}$ correspond to the \rustburrus conjecture failing in this scenario. 
    We note that for this example, the Burrus conjecture fails close to the constraint boundary while the $Q_{\chi^2_1}(0.68)$ quantile becomes valid once sufficiently far from the boundary.
    \vspace{1em}
    }
    \label{fig:3d_quantiles}
\end{SCfigure}

\subsection{Bounded constraint set in two dimensions} \label{subsec:bounded_domain}

As a last case study, we consider a modification of the example in \Cref{subsec:2d_gauss_num} in which the constraint set is chosen to be the bounded set $\mathcal{X} = [0,1]^2$. While not in the original scope of the \rustburrus conjecture, which only considers $\mathcal{X}=\{\bx:\bx\geq 0\}$, as mentioned in \Cref{subsec:intro_b_conject}, the $\cI_{\OSB}$ interval construction has since then been used with constraints of the form $\mathbf{A}\bx\leq \mathbf{b}$ by replacing $\bx \geq 0$ with $\mathbf{A}\bx\leq\mathbf{b}$ in the optimization problems \eqref{RB1} and \eqref{RB2} \cite{patil, stanley_unfolding}. We use the same experimental setup as in \Cref{subsec:2d_gauss_num}, taking into account that the change in $\mathcal{X}$ affects both the interval optimization problem and the optimizations required for $q_\alpha(\mu)$ (since the LLR statistic changes as well). 

 The results are shown in \Cref{fig:box_results}. We observe that in this setting, as opposed to the previous three-dimensional problem, the OSB intervals \emph{overcover}, because the maximum quantile of the LLR test statistic over the constraint set is \emph{smaller} than the $\chi^2_1$ quantile. Furthermore, our intervals MQ, and especially MQ$\mu$, are able to exploit this fact to obtain shorter intervals than OSB with coverage closer to $1-\alpha$ by using the actual max quantiles over $\mathcal{X}$ instead of the $\chi^2_1$ quantile used by OSB.

 \begin{figure}
     \centering
     \includegraphics[width = 0.9\columnwidth]{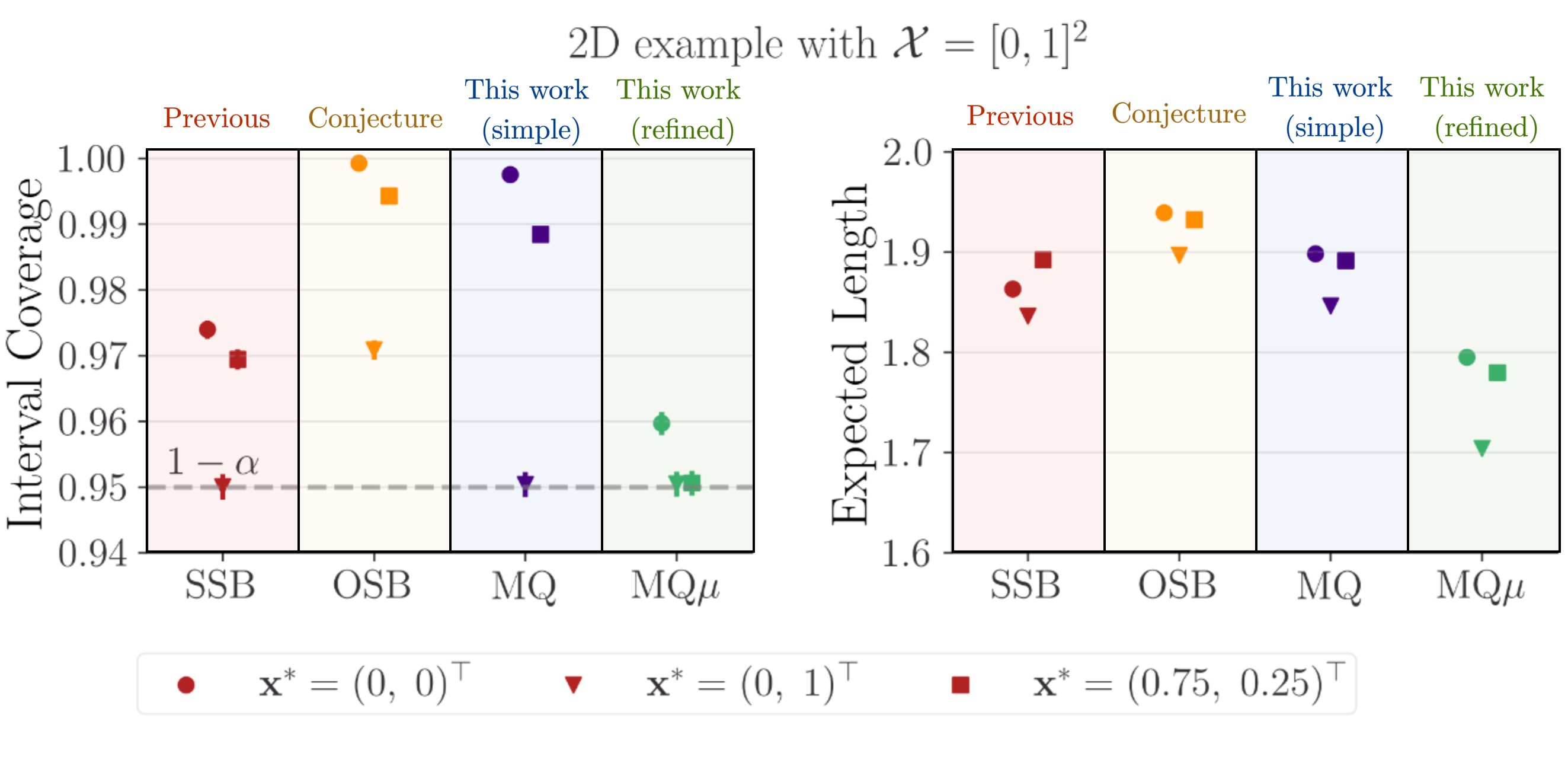}
     \caption{
     Estimated interval coverage \textbf{(left)} and expected lengths \textbf{(right)} for 95\% intervals resulting from the SSB, OSB, MQ and MQ$
    \mu$ methods for the Gaussian linear model in \eqref{eq:lineargaussianmodel} with $\bK = \bI_2$, $\varphi(\bx) = \bh^\tp \bx = x_1 - x_2$ and $\mathcal X = [0,1]^2 \subset \mathbb{R}^2$.
    }
     \label{fig:box_results}
 \end{figure}

\section{Discussion}
\label{sec:discussion_and_conclusion}

This paper presents a framework for constructing confidence intervals with guaranteed frequentist coverage for a given functional of forward model parameters in the presence of constraints.
For the specific case of the Gaussian linear forward model with nonnegativity constraints, we refute the \rustburrus conjecture \cite{burrus1965utilization} by providing a counterexample and propose a more general approach for interval construction. 
Our approach hinges on the inversion of a specific likelihood ratio test, and we offer theoretical and practical insights into the properties of the constructed intervals via illustrative examples.
Our framework is versatile, accommodating potentially nonlinear, non-Gaussian, and rank-deficient settings.

At a high level, the practical effectiveness of UQ methods depends on the (sometimes implicit) assumptions of the method. 
Different methods come into play depending on what we assume or know, be it the likelihood, the constraints, or the prior in Bayesian settings.
In classical statistics, confidence intervals serve as a valuable tool for UQ, especially for one-dimensional quantities of interest. 
These intervals are constructed to offer guaranteed coverage under repeated sampling, aligning with frequentist principles. 
While frequentist coverage guarantees are a useful criterion, especially in contexts where repeatability is essential, we acknowledge that the ``best'' UQ method is often context-dependent. 
For example, this frequentist approach is the most natural in applications like remote sensing (\cite{patil}), where repeatability is a key requirement. 
Conversely, when it is natural to think of the parameter as arising from a prior distribution, Bayesian methods are well-motivated and have desirable properties.

A key aim of this paper is to serve as a basis for the future development of these UQ procedures. 
We conclude this paper by discussing a few possible directions for future work:

\begin{itemize}[leftmargin=7mm]
    \item \textbf{Data-adaptive calibration procedure.}
    We saw in \Cref{sec:numerical-examples} that MQ is valid where OSB is not and can leverage smaller quantiles to produce tighter intervals. 
    In the event that one does not have the assumed true parameter bounding box, it may be possible to create a data-generated one and adjust the error budget accordingly. 
    Such a procedure would enable us to expand the use of the MQ intervals to scenarios with unbounded parameter constraints. 
    We are currently investigating this approach, which will be the subject of a future follow-up paper.

    \item \textbf{Exploration of high-dimensional problems.} 
    A key benefit of the optimization-based interval construction is that it promises to provide a solution to uncertainty quantification in high-dimensional and ultra-high-dimensional problems, where most alternative approaches (including sampling-based ones) are infeasible. 
    For example, in \cite{stanley_unfolding}, we applied a precursor of the methods presented here in a problem where $p=80$, and we are currently exploring the use of these methods in a data-assimilation setting where $p \approx 10^4$. 
    Indeed, optimization is one of the only computational techniques known to work in ultra-high-dimensional problems, such as those in 4DVar data assimilation \cite{kalmikov,ecco,liu_2016}. 
    While optimization has been successfully used for point estimation in such problems, the approaches developed here may enable modifying the existing programs to obtain confidence intervals in addition to point estimates.
    
    \item \textbf{Joint confidence sets for multiple functionals.}
    Since our framework is devised for UQ of a single functional,
    its application to collections of functionals (a higher-dimensional quantity of interest), would be a natural and desirable extension. 
    Trivially, given a collection of $K$ functionals, one could apply this methodology $K$ times and use the Bonferroni correction to adjust the confidence levels so that they all cover at the desired coverage level. 
    Although this approach might be practically reasonable when $K$ is small, it becomes markedly inefficient as $K$ becomes large. 
    Furthermore, this approach would create a $K$-dimensional hyper-rectangle for the quantity of interest, which may not be the optimal geometry for bounding the quantity of interest. 
    As such, extending the framework of \Cref{sec:interval_methodology} to simultaneously consider the $K$ functionals of interest would be the first step to creating a more nuanced approach.
    One way this can be achieved is by appropriately adjusting the definition of $H_0$ in the hypothesis test in \eqref{eq:h_test_oi_full}.
    
    \item \textbf{Choice of test statistics beyond LLR.} 
    The log-likelihood ratio test statistic considered in this work connects with the Rust--Burrus intervals and is observed to perform well in practice, but other choices can be explored in future work. 
    While the LLR is a natural choice for the generic problem, improving the interval length on particular families of problems with different test statistics might be possible. 
    Since the main theoretical machinery comes from the test inversion framework, which is independent of the actual form of the test statistic, alternate versions of \Cref{thm:interval_coverage} can be constructed as long as the test statistic constructs valid level-$\alpha$ hypothesis tests; the resulting intervals of which could be explored theoretically and numerically.
    For instance, the construction of confidence intervals in \Cref{subsec:conf_set_as_opt}, originally written for $\lambda(\mu, \by) = \inf_{\bx \in \Phi_\mu \cap \mathcal{X}} -2\ell_{\bx}(\by) - \inf_{\bx \in \mathcal{X}} -2 \ell_{\bx}(\by)$, readily generalizes to test statistics of the form $\lambda_{f,g}(\mu, \by) = \inf_{\bx \in \Phi_\mu \cap \mathcal{X}} f(\bx, \by) - g(\by)$. 
    In that case, \eqref{eq:defD} becomes $\{\bx: f(\bx, \by) \leq q_\alpha(\varphi(\bx)) + g(\by)\}$, where $q_\alpha$ must be valid for the particular choice of $f$ and $g$, and can be obtained by analyzing the distribution of $\lambda_{f,g}$.
        
    \item \textbf{Generalization to simulation-based problems.}
    An extension of our methodology to settings in which the likelihood is not exactly known can be considered, ranging from only partial knowledge of the form of the likelihood to full simulation-based (likelihood-free) settings where the likelihood is not known explicitly but can be sampled from. 
    A possible avenue is to develop robust worst-case approaches with respect to possible likelihoods. 
    In a fully likelihood-free setting, approaches such as \cite{dalmasso2020confidence, heinrich2022learning, waldo,dalmasso2021likelihood} provide ways to invert hypothesis tests to obtain confidence sets in these scenarios in multidimensional parameter spaces. 
    Projections of these sets could produce confidence intervals for a functional of the model parameters, as seen for the SSB intervals. 
    However, as we have explored, orienting the hypothesis test to the given functional of interest can have dramatic length benefits for the resulting confidence interval (as seen for the OSB, HSB, and MQ intervals). 
    Since the log-likelihood plays a key role in the definition of our intervals, extensions providing ways to relax that dependence would be a necessary first step.
\end{itemize}


\section*{Acknowledgments}

We warmly thank Larry Wasserman, Ann Lee, and other members of the STAMPS (Statistical Methods in the Physical Sciences) working group at Carnegie Mellon University for fruitful discussions on this work.
We also warmly thank Andrew Stuart, George Karniadakis, and their groups at the California Institute of Technology and Brown University for useful comments and stimulating discussions.
We are also grateful to Jon Hobbs and Amy Braverman at the Jet Propulsion Laboratory at the National Aeronautics and Space Administration for many discussions related to this work, and for hosting visits of PP, MS, and MK, and for their kind hospitality.

PB and HO gratefully acknowledge support from the Air Force Office of Scientific Research under MURI award number FA9550-20-1-0358 (Machine Learning and Physics-Based Modeling and Simulation), the Jet Propulsion Laboratory, California Institute of Technology, under a contract with the National Aeronautics and Space Administration, and Beyond Limits (Learning Optimal Models) through CAST (The Caltech Center for Autonomous Systems and Technologies). MS and MK were supported by NSF grants DMS-2053804 and PHY-2020295, JPL RSAs No.\ 1670375, 1689177 \& 1704914 and a grant from the C3.AI Digital Transformation Institute.

\bibliographystyle{apalike}
\bibliography{references}

\clearpage
\appendix

\newcommand{\removelinebreaks}[1]{%
      \def\\{\relax}#1}
\def\titleRLB{\removelinebreaks{\titletext}}

\begin{center}
\Large
{
{\bf
\underline{Supplementary Material}}
}
\end{center}

\bigskip

This document serves as a supplement to the paper entitled ``Optimization-based frequentist confidence intervals for functionals in constrained inverse problems: Resolving the \rustburrus conjecture''.
We outline below the structure of the supplement and summarize key notation used both in this supplement and the main paper.

\subsection*{Organization}

The organization this supplement is summarized in \Cref{tab:supplement-organization}.

\begin{table}[!ht]
\centering
\begin{tabularx}{\textwidth}{L{2.05cm}L{15cm}}
\toprule
\textbf{Appendix} & \textbf{Description} \\
\midrule
\Cref{app:sec:data_gen_test_def_inv_arg_proofs}  & Proofs of \Cref{lem:conf_set_coverage}, \Cref{thm:interval_coverage}, 
\Cref{thm:interval_coverage_converse}, and \Cref{lemma:1d_cdf_part1,lemma:unconstrchi2}  
\\ 
\addlinespace[0.5ex] \arrayrulecolor{black!25}  \midrule
\Cref{app:sec:interval_methodology_proofs} & Proofs of \Cref{lemma:SD,lemma:1d_cdf} 
\\
\arrayrulecolor{black!25} \midrule \addlinespace[0.5ex] \arrayrulecolor{black!25}
\Cref{app:sec:rust_burrus_proofs} & Proofs of \Cref{thm:burrus_false}, \Cref{lemma:burruschi2,lemma:burruschi2B,lemma:invalidity-counterexample,lem:3d-counterexample}, and \Cref{lemma:dimbounding} 
\\
\arrayrulecolor{black!25} \midrule \addlinespace[0.5ex] \arrayrulecolor{black!25}
\Cref{app:sec:numerical-examples} & Additional numerical illustrations in \Cref{sec:numerical-examples}
\\
\addlinespace[0.5ex] \arrayrulecolor{black} \bottomrule
\end{tabularx}
\caption{Roadmap of the supplement.}
\label{tab:supplement-organization}
\end{table}

\subsection*{Notation}

Some of the general notation used throughout this paper summarized in \Cref{tab:notations}.
(Any specific notation needed is explained in respective sections as necessary.)

\begin{table}[ht]
\centering
\begin{tabularx}{\textwidth}{L{2.75cm}L{12.5cm}}
\toprule
\textbf{Notation} & \textbf{Description} \\
\midrule
Non-bold lower or upper case & Denotes scalars (e.g., $\alpha$, $\mu$, $Q$). \\
Bold lower case & Denotes vectors (e.g., $\bx$, $\by$, $\bh$). \\
Bold upper case & Denotes matrices (e.g., $\bK$, $\bI$). \\
Calligraphic font & Denotes sets (e.g., $\mathcal{X}$, $\mathcal{C}$, $\mathcal{D}$). \\
\arrayrulecolor{black!25}\midrule
$\mathbb{R}$ & Set of real numbers. \\
$\mathbb{R}_{\ge 0}$ & Set of non-negative real numbers. \\
$[n]$ & Set $\{1, \dots, n\}$ for a positive integer $n$. \\
$(x, y, z)$ & (Ordered) tuple of elements $x$, $y$, $z$. \\
$\{x, y, z\}$ & Set of elements $x$, $y$, $z$. \\
\midrule
$\lceil x \rceil$, $\lfloor x \rfloor$,  $(x)_{+}$ & Ceiling, floor, and positive part of a real number $x$. \\
$\bm{1}\{A\}$, $\PP(A)$ & Indicator random variable associated with an event $A$ and probability of $A$. \\
$\EE[X], \Var(X)$ & Expectation and variance of a random variable $X$. \\
\midrule
$\bA^{-1}$ & Inverse of a square invertible matrix $\bA \in \RR^{m \times m}$. \\
$\bb^\top$, $\bB^\top$ & Transpose of a column vector $\bb \in \RR^{m}$ and a rectangular matrix $\bB \in \RR^{m \times p}$. \\
$\mathrm{rank}(\bC)$ & Rank of a matrix $\bC \in \RR^{m \times p}$ \\
$\mathrm{null}(\bC)$ & Nullity of a matrix $\bC \in \RR^{m \times p}$ \\
$\bD^{1/2}$ & Principal square root of a positive semidefinite matrix $\bD \in \RR^{p \times p}$.  \\
$\bI$, $\bm{1}$, $\bm{0}$ & The identity matrix, the all ones vector, the all zeros vector of appropriate dimensions \\
\midrule
$\langle \mathbf{u}, \mathbf{v} \rangle$ & The inner product of vectors $\mathbf{u}$ and $\mathbf{v}$. \\
$\| \bu \|_{\bD}$ & The induced $\ell_2$ norm of a vector $\bu$ with respect to a positive semidefinite matrix $\bD$ \\
$\| \mathbf{u} \|_{q}$ & The $\ell_q$ norm of a vector $\mathbf{u}$ for $q \ge 1$. \\
$\| f \|_{L_q}$ & The $L_q$ norm of a function $f$ for $q \ge 1$. \\
\midrule
$\bu \le \bv$ & Lexicographic ordering for vectors $\bu$ and $\bv$. \\
$\bm{A} \preceq \bm{B}$ & Loewner ordering for symmetric matrices $\bm{A}$ and $\bm{B}$.\\
$X \preceq Y$ & Stochastic dominance order for random variables $X$ and $Y$ (see \Cref{subsec:stochastic_dominance} for details). \\
\midrule
$Y = \cO_\alpha(X)$ & Deterministic big-O notation, indicating that $Y$ is bounded by $| Y | \le C_\alpha X$ for some constant $C_\alpha$ that may depend on the ambient parameter $\alpha$. \\
$\cO_p$ & Probabilistic big-O notation. \\
$\dto$ & Convergence in distribution. \\
$\pto$ & Convergence in probability. \\
\arrayrulecolor{black}\bottomrule
\end{tabularx}
\caption{Summary of general notation used throughout the paper and the supplement.}
\label{tab:notations}
\end{table}

A note about $\min/\max$ versus $\inf/\sup$:
We use $\min/\max$ when the optimal value of an optimization problem is attained, otherwise we use $\inf/\sup$.

A note about uniqueness of optimization problems: 
When we express an equality involving the minimizer of an optimization problem, this is intended to signify a set inclusion. 
The guarantees presented in our paper are applicable to all solutions, and we make no distinction among multiple solutions.

\clearpage
\section{Proofs in \Cref{sec:data_gen_test_def_inv_arg}}
\label{app:sec:data_gen_test_def_inv_arg_proofs}

\subsection{Proof of \Cref{lem:conf_set_coverage}} \label{proof:conf_set_coverage}

To prove the lemma, we need to show that the probability of \( \mu^* \) being in the confidence set \( \cC_\alpha(\by) \) is at least \( 1 - \alpha \). 
Towards this end, observe that
\begin{align*}
    \mathbb{P}_{\by \sim P_{\bx^*}} (\mu^* \in \cC_\alpha(\by)) &= \mathbb{P}_{\by \sim P_{\bx^*}} (\by \in A_\alpha(\mu^*))  \\
    &= 1 - \mathbb{P}_{\by \sim P_{\bx^*}} (\by \notin A_\alpha(\mu^*)) \\
    &\geq 1 - \sup_{\bx \in \Phi_{\mu^*} \cap \mathcal{X}} \mathbb{P}_{\by \sim P_{\bx}} (\by \notin A_\alpha(\mu^*))  \\
    &\geq 1 - \alpha,
\end{align*}
as desired.
This completes the proof.

\subsection{Proof of \Cref{lemma:quantile}} \label{proof:lemma:quantile}

The value $\maxQ_\mu$ is a valid decision value for any given $\mu$, since for all $\bx \in \Phi_\mu \cap \mathcal X$ it holds that
\begin{equation}\label{eq:quantileproof-proof}
    \mathbb{P}_{\lambda \sim F_{\bx}} \left(\lambda > \sup_{\bx' \in \Phi_\mu \cap \mathcal{X}}Q_{F_{\bx'}}(1-\alpha) \right) \leq  \mathbb{P}_{\lambda \sim F_{\bx}} \left(\lambda > Q_{F_\bx}(1-\alpha) \right) = \alpha.
\end{equation}
For any $v < \maxQ_\mu$ that one could use as a decision value, there exists $\tilde{\bx} \in \Phi_\mu \cap \mathcal X$ such that $Q_{F_{\tilde{\bx}}}(1-\alpha) > v$. 
Therefore, $\mathbb{P}_{\lambda \sim F_{\tilde{\bx}}} (\lambda > v) > \alpha$, and thus $v$ is not a valid decision value. 
$\maxQ$ is clearly valid for all $\mu$, and a similar argument shows that choosing any smaller $v$ would make it not valid, as there exists a $\tilde{\bx} \in \mathcal{X}$ with a larger quantile than $v$, making $v$ invalid as a decision value for $\mu = \varphi(\tilde{\bx})$. 
The equality between the two formulations comes from the same argument that shows \eqref{eq:combined_type_1} is equivalent to \eqref{eq:combined_type_1_better}.

\subsection{Proof of \Cref{thm:interval_coverage}}\label{proof:main}
Assume $\bar{\mathcal X}_\alpha(\by)$ is nonempty and write as shorthand $\inf_{\bx \in \mathcal X}/\sup_{\bx \in \mathcal X} f(\bx)$ for the interval
\[
    \left[\inf_{\bx \in \mathcal X} f(\bx), \; \sup_{\bx \in \mathcal X} f(\bx)\right].
\]
Observe that
\begin{equation} 
    \cC_\alpha(\by) \subseteq \inf_{\mu \in \cC_\alpha(\by)}/\sup_{\mu \in \cC_\alpha(\by)} \mu.
\end{equation}
From \Cref{lem:conf_set_coverage}, $\cC_\alpha(\by) \subseteq \cI_\alpha(\by)$ implies that $\cI_\alpha(\by)$ is also a $1 - \alpha$ confidence interval. 
We prove this interval exactly equals the defined $\cI_\alpha(\by)$ in \eqref{intervaldef}. 
Unpacking the definition of $\cC_\alpha(\by)$, we write the interval
\begin{equation}\label{opt1}
\begin{aligned}
\inf_\mu/\sup_{\mu} \quad & \mu \\
\st \quad & \mu \in \mathbb{R} \\ & -2\log\Lambda(\mu, \by) \leq q_\alpha(\mu).
\end{aligned}
\end{equation}
We can write different optimization problems which are equivalent to the optimization problem \eqref{opt1}. 
First, we use the definition of $\Lambda$ to write:
\begin{equation}
\label{opt2}
\begin{aligned}
\inf_\mu/\sup_{\mu} \quad & \mu \\
\st \quad & \mu \in \mathbb{R} \\ & \inf_{\varphi(\bx) = \mu, \bx \in \mathcal X} -2\ell_{\bx}(\by) - \inf_{\bx \in \mathcal X} -2\ell_{\bx}(\by)\leq q_\alpha(\mu) .
\end{aligned}
\end{equation}
Notice that we can rewrite the feasibility condition of $\mu$ as follows:
\[
\inf_{\varphi(\bx) = \mu, \bx \in \mathcal X} -2\ell_{\bx}(\by)\leq q_\alpha(\mu) + \inf_{\bx \in \mathcal X} -2\ell_{\bx}(\by)
\]
as there exists $\bx \in \mathcal X$ such that $\varphi(\bx) = \mu$ and 
\[
-2\ell_{\bx}(\by) \leq q_\alpha(\mu) + \inf_{\bx \in \mathcal X} -2\ell_{\bx}(\by).
\]
Therefore, the optimization problem can be rewritten with $\bx$ and $\mu$ as the optimization variables:
\begin{equation}
\label{opt3}
\begin{aligned}
\inf_{\mu, \bx}/\sup_{\mu, \bx} \quad & \mu \\
    \st \quad & \bx \in \mathcal X,~ \mu \in \mathbb{R} \\
& \varphi(\bx) = \mu \\
& -2\ell_{\bx}(\by) \leq q_\alpha(\mu) + \inf_{\bx \in \mathcal X} -2\ell_{\bx}(\by).
\end{aligned}
\end{equation}
And $\mu$ can be eliminated using the constraint, yielding
\begin{equation}\label{param_space}
\begin{aligned}
\inf_{\bx}/\sup_{\bx} \quad & \varphi(\bx) \\
\st \quad & \bx \in \mathcal X\\
& -2\ell_{\bx}(\by) \leq q_\alpha(\varphi(\bx))+ \inf_{\bx \in \mathcal X} -2\ell_{\bx}(\by),
\end{aligned}
\end{equation}
that is, $\inf_{\bx \in  \bar{\mathcal X}_\alpha(\by)}/\sup_{\bx \in  \bar{\mathcal X}_\alpha(\by)} \varphi(\bx)$. 
The choice when $\bar{\mathcal X}_\alpha(\by)$ is empty does not affect coverage properties. 
An alternative proof comes by observing the set equality $\varphi(\{\bx \in \mathcal{X}: -2\ell_{\bx}(\by) \leq q_\alpha(\varphi(\bx))+ \inf_{\bx \in \mathcal X} -2\ell_{\bx}(\by)\}) = \{\mu \in \varphi(\mathcal{X}):  \inf_{\varphi(\bx) = \mu, \bx \in \mathcal X} -2\ell_{\bx}(\by) - \inf_{\bx \in \mathcal X} -2\ell_{\bx}(\by)\leq q_\alpha(\mu)\}$, and the result follows. 
This finishes the proof.

\subsection{Proof of \Cref{thm:interval_coverage_converse}}\label{proof:main_converse}

We have, by definition and test inversion, that $q_\alpha(\mu)$ are valid if and only if 
\[
    \cC_\alpha(\by) := \{\mu \colon \lambda(\mu, \by) \leq q_\alpha(\mu)\}
\]
is a valid $1-\alpha$ confidence interval for any $\bx \in \mathcal{X}$. 
Since $\cI_\alpha(\by)$ is the smallest interval that contains $\cC_\alpha(\by)$, we aim to prove that  $\cC_\alpha(\by)$ is already an interval (including singletons or empty sets), so that $\cC_\alpha(\by) = \cI_\alpha(\by)$ and the result holds. 
Define the function:
\begin{equation}
    \mu \mapsto \mathcal{F}(\mu) = \inf_{\substack{\varphi(\bx) = \mu \\ \bx \in \mathcal{X}}} -2\ell_\bx(\by)
\end{equation}
for a given $\by$, supported in all $\mu$ such that $\Phi_\mu \cap \mathcal X \neq \emptyset$. 
Write $\cC_\alpha(\by)$ explicitly using \eqref{eq:log_likelihood_ratio_full}, we get
\begin{equation}
    \cC_\alpha(\by) := 
    \Bigg \{\mu \colon  \underset{\substack{\varphi(\bx) = \mu \\ \bx \in \mathcal{X}}} {\text{inf}} \; -2\ell_{\bx}(\by)- \underset{\bx \in \mathcal{X}}{\text{inf}} \; -2 \ell_{\bx}(\by) \leq q_\alpha \Bigg \}.
\end{equation}
The second term on the left-hand side does not depend on $\mu$, so it is enough to prove that any set of the form $\{\mu \colon \mathcal{F}(\mu) \leq z\}$ is an interval, which is implied by the function $\mathcal{F}(\mu)$ being convex in $\mu$ (for a fixed $\by$). 
Indeed, if the set is not an interval, we have $\mu^- < \mu < \mu^+$ with $\mu^-, \mu^+ \in \cC_\alpha(\by)$ and $\mu \notin \cC_\alpha(\by)$ which contradicts convexity, since 
\[
    \mathcal{F}(\mu) \geq z > \gamma \mathcal{F}(\mu^-) + (1-\gamma)\mathcal{F}(\mu^+).
\]

To see convexity and finish the proof, let $\mu_1 \neq \mu_2$ and let $\mathcal{G}(\bx) := -2\ell_\bx(\by)$, a convex function by assumption. 
Write for $i = 1,2$:
\[
    x_i \in \argmin_{\substack{\varphi(\bx) = \mu \\ \bx \in \mathcal{X}}} -2\ell_\bx(\by), 
\]
with $x_i$ being any possible element in the set of minimizers, so that $\mathcal{F}(\mu_i) = \mathcal{G}(\bx_i)$. 
For any $0 < \gamma < 1$, $\gamma \bx_1 + (1-\gamma) \bx_2 \in \mathcal X$ since $\mathcal X$ is a convex cone and 
\[
    \varphi(\gamma \bx_1 + (1-\gamma) \bx_2) = \gamma \mu_1 + (1-\gamma)\mu_2,
\]
since $\varphi$ is linear, so $\gamma \bx_1 + (1-\gamma) \bx_2$ is a feasible point of the optimization problem 
\[
\inf_{\substack{\varphi(\bx) = \gamma \mu_1 + (1-\gamma)\mu_2 \\ \bx \in \mathcal{X}}} -2\ell_\bx(\by),
\]
that has optimal value $\mathcal{F}(\gamma\mu_1 + (1-\gamma)\mu_2)$.
Therefore, by convexity of $\mathcal G$ and definition of the $\bx_i$, we have that:
\begin{equation}
    \mathcal{F}(\gamma\mu_1 + (1-\gamma)\mu_2) \leq \mathcal G(\gamma \bx_1 + (1-\gamma) \bx_2) \leq \gamma \mathcal G(\bx_1) + (1-\gamma) \mathcal{G}(\bx_2) = \gamma \mathcal F(\mu_1) + (1-\gamma) \mathcal{F}(\mu_2).
\end{equation}
This completes the proof.

\subsection{Proof of \Cref{lemma:1d_cdf_part1}}
\label{app:proof-lemma:1d_cdf_part1}

Since the case when $\mu^* = 0$ is of particular interest, we show the result in this specific case and then generalize to the case of $\mu^* > 0$.

\myparagraph{\underline{Case of $\mu^*$ = 0}}
When $\mu^* = 0$, we can argue from symmetry of the standard Gaussian about the origin to write down the CDF in closed form. 
For $c \geq 0$, we have
\begin{align}
    \mathbb{P}_{\mu_0} \left(\ell_0 \leq c \right) &= \mathbb{P}_{\mu_0} \left(\ell_0 \leq c, y < 0 \right) + \mathbb{P}_{\mu_0} \left(\ell_0\leq c, y \geq 0 \right) \nonumber \\
    &= \mathbb{P}_{\mu_0} \left(\ell_0 \leq c \mid y < 0 \right)\mathbb{P}_{\mu_0}(y < 0) + \mathbb{P}_{\mu_0} \left(\ell_0 \leq c \mid y \geq 0 \right)\mathbb{P}_{\mu_0}(y \geq 0). \label{eq:tot_prob_form}
\end{align}
By definition, $\mathbb{P}_{\mu_0}(y < 0) = \mathbb{P}_{\mu_0}(y \geq 0) = \frac{1}{2}$, so only the conditional probabilities remain. 
By \eqref{eq:1d_log_lr}, we have
\begin{align}
    \mathbb{P}_{\mu_0} \left(\ell_0 \leq c \mid y < 0 \right) &= \mathbb{P}_{\mu_0} \left(0 \leq c \mid y < 0 \right) = 1 \nonumber \\
    \mathbb{P}_{\mu_0} \left(\ell_0 \leq c \mid y \geq 0 \right) &= \mathbb{P}_{\mu_0} \left(y^2 \leq c \mid y \geq 0 \right). \label{eq:cond_cdf_2}
\end{align}
In \eqref{eq:cond_cdf_2}, we immediately observe that
\begin{equation}
    \mathbb{P}_{\mu_0} \left(y^2 \leq c \mid y \geq 0 \right) = \mathbb{P}_{\mu_0} \left(y^2 \leq c, y \geq 0 \right) \mathbb{P}_{\mu_0} \left(y \geq 0 \right)^{-1} = 2 \mathbb{P}_{\mu_0} \left(0 \leq y \leq \sqrt{c} \right) = 2 \Phi(\sqrt{c}) - 1.
\end{equation}
But we also have that 
\[
\mathbb{P}_{\mu_0} \left(y^2 \leq c \right) = \mathbb{P}_{\mu_0} \left(-\sqrt{c} \leq y \leq \sqrt{c} \right) = 2 \Phi(\sqrt{c}) - 1.
\]
So we have 
\[
\mathbb{P}_{\mu_0} \left(y^2 \leq c \mid y \geq 0 \right) = \mathbb{P}_{\mu_0} \left(y^2 \leq c \right).
\]
Hence, we obtain
\[
\mathbb{P}_{\mu_0} \left(\ell_0 \leq c \mid y \geq 0 \right) = \chi^2_1(c).
\]
Note this independence on the sign of $y$ means that the magnitude of $y$ is statistically independent of its direction. 
Thus, when $\mu_0 = 0$, the log-likelihood ratio has the following distribution:
\begin{equation}
    \ell_0 \sim \frac{1}{2} + \frac{1}{2} \chi^2_1.
\end{equation}
This completes the case when $\mu^* = 0$.

\myparagraph{\underline{Case of $\mu^* > 0$}}
When $\mu > 0$, the closed-form solution to the CDF of $\ell_0$ becomes more complicated, as we can no longer use symmetry around the origin. 
Picking up at \eqref{eq:tot_prob_form}, we first note that when $y \sim \cN(\mu_0, 1)$, we have
\begin{align*}
    \mathbb{P}_{\mu_0} \left(y < 0 \right) = \Phi(- \mu_0)
    \quad
    \text{and}
    \quad
    \mathbb{P}_{\mu_0} \left(y \geq 0 \right) = \Phi(\mu_0).
\end{align*}
Next, we must find the conditional probabilities. 
Starting with the case when $\{y < 0 \}$, we obtain
\begin{align}
    &\mathbb{P}_{\mu_0} \left((y - \mu_0)^2 - y^2 \leq c \mid y < 0 \right) \\ 
    &= \mathbb{P}_{\mu_0} \left(-2 y \mu_0 + \mu_0^2 \leq c \mid y < 0 \right) \nonumber \\
    &= \mathbb{P}_{\mu_0} \left(y \geq \frac{\mu_0^2 - c}{2 \mu_0} \mid y < 0 \right) \nonumber \\
    &= \Phi(- \mu_0)^{-1} \mathbb{P}_{\mu_0} \left(y \geq \frac{\mu_0^2 - c}{2 \mu_0}, y < 0 \right) \nonumber \\
    &= \Phi(- \mu_0)^{-1} \left\{ 0 \cdot \bm{1} \{c \leq \mu_0^2 \} + \mathbb{P}_{\mu_0} \left( \frac{\mu_0^2 - c}{2 \mu_0} \leq y \leq 0 \right) \bm{1} \{c > \mu_0^2 \} \right\} \nonumber \\
    &= \Phi(- \mu_0)^{-1} \mathbb{P}_{\mu_0} \left( \frac{-\mu_0^2 - c}{2 \mu_0} \leq y - \mu_0 \leq - \mu_0 \right) \bm{1} \{ c > \mu_0^2 \} \nonumber \\
    &= \Phi(- \mu_0)^{-1} \left\{\Phi(- \mu_0) - \Phi\left(\frac{-\mu_0^2 - c}{2 \mu_0} \right) \right\}\bm{1} \{ c > \mu_0^2 \}. \label{eq:cond_y_l_0}
\end{align}
Then, when $\{y \geq 0 \}$, we have
\begin{align}
    & \mathbb{P}_{\mu_0} \left((y - \mu_0)^2 \leq c \mid y \geq 0 \right) \\ 
    &= \mathbb{P}_{\mu_0} \left(- \sqrt{c} \leq y - \mu_0 \leq \sqrt{c} \mid y \geq 0 \right) \nonumber \\
    &= \Phi(\mu_0)^{-1} \mathbb{P}_{\mu_0} \left(- \sqrt{c} \leq y - \mu_0 \leq \sqrt{c}, y \geq 0 \right) \nonumber \\
    &= \Phi(\mu_0)^{-1} \mathbb{P}_{\mu_0} \left(0 \leq y \leq \sqrt{c} + \mu_0 \right) \bm{1} \{ - \sqrt{c} + \mu_0 \leq 0 \} \nonumber \\
    &\quad + \Phi(\mu_0)^{-1} \mathbb{P}_{\mu_0} \left(-\sqrt{c} + \mu_0 \leq y \leq \sqrt{c} + \mu_0 \right) \bm{1} \{ - \sqrt{c} + \mu_0 > 0 \} \nonumber \\
    &= \Phi(\mu_0)^{-1} \left\{ \left( \Phi(\sqrt{c}) - \Phi(- \mu_0) \right) \bm{1} \{ c \geq \mu_0^2 \} + (2 \Phi(\sqrt{c}) - 1) \bm{1} \{ c < \mu_0^2 \} \right\}. \label{eq:cond_y_geq_0}
\end{align}
Putting together \eqref{eq:cond_y_l_0} and \eqref{eq:cond_y_geq_0}, we obtain the following CDF:
\begin{equation} 
    \mathbb{P}_{\mu_0} \left( \ell_0 \leq c \right) = \chi^2_1(c) \cdot \bm{1} \{ c < \mu_0^2 \} + \left\{ \Phi(\sqrt{c}) - \Phi\left( \frac{-\mu_0^2 - c}{2 \mu_0} \right) \right\} \cdot \bm{1} \{ c \geq \mu_0^2 \}.
\end{equation}
This completes the case of $\mu^* > 0$.

\subsection{Proof of \Cref{lemma:unconstrchi2}}\label{proof:unconstrchi2}

We derive this result using a duality argument inspired by \cite{gourieroux}.
By definition, we have
\begin{equation} \label{eq:log_like_def}
    \lambda(\mu^*, \by) = \min_{\bx: \varphi(\bx) = \mu^*} \lVert \by - \bK \bx \rVert_2^2 - \min_{\bx} \lVert \by - \bK \bx \rVert_2^2.
\end{equation}
For ease of notation, let $\hat{\bx}^* = \underset{\bx \colon \bh^\tp\bx = \mu^*}{\text{argmin}} \; \lVert \by - \bK \bx \rVert_2^2$. 
Consider the Lagrangian for the first optimization in \eqref{eq:log_like_def}:
\begin{equation} \label{eq:lagrangian_unconstr}
    L(\bx, \lambda) = \lVert \by - \bK \bx \rVert^2_2 + \lambda (\bh^\tp \bx - \mu^*).
\end{equation}
First-order optimality allows solving for $\hat{\bx}^*$ as a function of the dual variable $\lambda$:
\begin{align*}
    \nabla_{\bx} L(\bx, \lambda) &= -2 \bK^\tp (\by - \bK \bx) + \lambda \bh = 0 \\
    &\implies -2 \bK^\tp \by + 2 \bK^\tp \bK \bx + \lambda \bh = 0 \\
    &\implies \hat{\bx}^* = (\bK^\tp \bK)^{-1} \bK^\tp \by - \frac{1}{2} \lambda (\bK^\tp \bK)^{-1} \bh \\
    &\implies \hat{\bx}^* = \hat{\bx} - \frac{1}{2} \lambda (\bK^\tp \bK)^{-1} \bh.
\end{align*}
Substituting back into the LLR, we obtain
\begin{align}
    \lambda(\mu^*, \by) &= \lVert \by - \bK \hat{\bx}^* \rVert_2^2 - \lVert \by - \bK \hat{\bx} \rVert_2^2 \nonumber \\
    &= \lVert \by - \bK \hat{\bx} + \frac{1}{2} \lambda \bK(\bK^\tp \bK)^{-1} \bh \rVert^2_2 - \lVert \by - \bK \hat{\bx} \rVert_2^2. \label{eq:reduction_log_lik}
\end{align}
Performing some algebra, we note that
\begin{align*}
    \lVert \by - \bK \hat{\bx} + \frac{1}{2} \lambda \bK(\bK^\tp \bK)^{-1} \bh \rVert^2_2 &= \lVert \by - \bK \hat{\bx} \rVert_2^2 \\
    & \quad + \lambda (\by - \bK \hat{\bx})^\tp \bK (\bK^\tp \bK)^{-1} \bh + \frac{1}{4} \lambda^2 \bh^\tp (\bK^\tp \bK)^{-1} \bh.
\end{align*}
Thus, we have
\begin{align*}
    \lambda ( \by - \bK \hat{\bx} )^\tp \bK (\bK^\tp \bK)^{-1} \bh &= \lambda \by^\tp \bK (\bK^\tp \bK)^{-1} \bh - \lambda \hat{\bx}^\tp \bK^\tp \bK (\bK^\tp \bK)^{-1} \bh \\
    &= \lambda \hat{\bx}^\tp \bh - \lambda \hat{\bx}^\tp \bh \\
    &= 0.
\end{align*}
So the substitution in \eqref{eq:reduction_log_lik} can be further simplified such that:
\begin{equation} \label{eq:simplified_log_lik_unconstr}
    \lambda(\mu^*, \by) = \frac{1}{4} \lambda^2 \bh^\tp (\bK^\tp \bK)^{-1} \bh.
\end{equation}
We now turn our attention to finding $\lambda$. Note that this optimization defining the Lagrangian~\eqref{eq:lagrangian_unconstr} is convex with an affine equality constraint. 
Therefore, strong duality holds. We then define the dual function as follows:
\begin{align}
    g(\lambda) &= \min_{\bx} L(\bx, \lambda) = L(\hat{\bx}^*, \lambda) \nonumber \\
    &= \lVert \by - \bK \hat{\bx}^* \rVert^2_2 + \lambda (\bh^\tp \hat{\bx}^* - \mu^*) \nonumber \\
    &= \lVert \by - \bK \hat{\bx} + \frac{1}{2} \lambda \bK(\bK^\tp \bK)^{-1} \bh \rVert^2_2 + \lambda \left(\bh^\tp \hat{\bx} - \frac{1}{2} \lambda \bh^\tp (\bK^\tp \bK)^{-1} \bh - \mu^* \right). 
\end{align}
We note that we can make many of the same simplifications above to arrive at the simplified dual function:
\begin{equation}
    g(\lambda) = \lVert \by - \bK \hat{\bx} \rVert_2^2 + \lambda \bh^\tp - \frac{1}{4} \lambda^2 \bh^\tp (\bK^\tp \bK)^{-1} \bh + \lambda \bh^\tp \hat{\bx} - \lambda \mu^*.
\end{equation}
To maximize $g(\lambda)$, we again use the following first order optimality condition:
\begin{align}
    \frac{d g}{d \lambda} &= - \frac{1}{2} \lambda \bh^\tp (\bK^\tp \bK)^{-1} \bh + \bh^\tp \hat{\bx} - \mu^* = 0 \nonumber \\
    \implies \hat{\lambda} &= \frac{2 \left( \bh^\tp \hat{\bx} - \mu^* \right)}{\bh^\tp (\bK^\tp \bK)^{-1} \bh}. \label{eq:lambda_hat_unconstr}
\end{align}
Substituting \eqref{eq:lambda_hat_unconstr} back into \eqref{eq:simplified_log_lik_unconstr}, we obtain
\begin{align*}
    \lambda(\mu^*, \by) 
    &= \frac{1}{4} \left(\frac{2 \left( \bh^\tp \hat{\bx} - \mu^* \right)}{\bh^\tp (\bK^\tp \bK)^{-1} \bh} \right)^2 \bh^\tp (\bK^\tp \bK)^{-1} \bh \\
    &= \frac{(\bh^\tp \hat{\bx} - \mu^*)^2}{\bh^\tp (\bK^\tp \bK)^{-1} \bh}.
\end{align*}
 
For the second part, observe that when $\by \sim \mathcal{N}(\bK\bx^*, \bI_m)$, we have
\[
    \bh^\tp (\bK^\tp \bK)^{-1} \bK^\tp \by \sim \mathcal{N}(\bh^\tp\bx^*, \bh^\tp 
    (\bK^\tp \bK)^{-1} \bh),
\] hence \eqref{eq:test_stat_eq_unconstr} is the square of a one-dimensional standard Gaussian distribution.
This finishes the proof.

\section{Proofs in \Cref{sec:interval_methodology}}
\label{app:sec:interval_methodology_proofs}

\subsection{Proof of \Cref{lemma:SD}}
\label{sec:proof-subsec:stochastic_dominance}

Let $Y := \lambda (\by, \mu^*)$.
Recall the validity of $q_\alpha$ can be written as $\mathbb{P}(Y\leq q_\alpha) \geq 1-\alpha$ from \eqref{eq:type1_err_c_alpha} as:
\begin{align*}
X \succeq Y &\iff \mathbb{P}(X\geq \gamma) \geq \mathbb{P}(Y \geq \gamma), \; \; \text{for all } \gamma 
\\ 
&\iff \alpha = \mathbb{P}(X\geq Q_X(1-\alpha)) \geq \mathbb{P}(Y\geq Q_X(1-\alpha)), \; \; \text{for all } \alpha 
\\
&\iff 1-\alpha = \mathbb{P}(X\leq Q_X(1-\alpha))\leq \mathbb{P}(Y \leq Q_X(1-\alpha)), \; \; \text{for all } \alpha 
\\
&\iff Q_X(1-\alpha) \text{ is a valid } q_\alpha \text{for all } \alpha.
\end{align*}
This finishes the proof.

\subsection[Joint formulation of change-constrained optimization problem]{Joint formulation of change-constrained optimization problem \eqref{eq:cco_problem}}
\label{sec:joint-optimatization}

In this section, we provide details on formulating the optimization problem \eqref{eq:cco_problem} and the interval optimization problem \eqref{intervaldef} as a single chance-constrained optimization problem.
By joining \eqref{eq:cco_problem} and \eqref{intervaldef} as a single optimization problem, we can use a similar argument as in the proof of \Cref{thm:interval_coverage}.
Starting with problem \eqref{opt2}:
\begin{equation}\label{opt2_again}
\begin{aligned}
\inf_\mu/\sup_{\mu} \quad & \mu \\
\st \quad & \mu \in \mathbb{R} \\ & \inf_{\varphi(\bx) = \mu, \bx \in \mathcal X} -2\ell_{\bx}(\by) - \inf_{\bx \in \mathcal X} -2\ell_{\bx}(\by) \leq q_\alpha(\mu),
\end{aligned}
\end{equation}
and substituting $q_\alpha(\mu) = \sup_{\bx \in \Phi_\mu \cap \mathcal{X}}Q_{F_\bx}(1-\alpha) = -\inf_{\bx \in \Phi_\mu \cap \mathcal{X}} -Q_{F_\bx}(1-\alpha)$, we obtain:
\begin{equation}\label{opt2bisagain}
\begin{aligned}
\inf_\mu/\sup_{\mu} \quad & \mu \\
\st \quad & \mu \in \mathbb{R} \\ & \inf_{\begin{subarray}{l}
\varphi(\bx_1) = \mu, \bx_1 \in \mathcal X, \\
\varphi(\bx_2) = \mu, \bx_2 \in \mathcal X
\end{subarray}} \left [-2\ell_{\bx_1}(\by) -Q_{F_{\bx_2}}(1-\alpha) \right ]
\leq \inf_{\bx \in \mathcal X} -2\ell_{\bx}(\by),
\end{aligned}
\end{equation}
which can be transformed to parameter space as:
\begin{equation}\label{paramspaceagain}
\begin{aligned}
\inf_{\bx_1, \bx_2}/\sup_{\bx_1, \bx_2} \quad & \varphi(\bx_1) \\
\st \quad & \bx_1, \bx_2 \in \mathcal X\\
& \varphi(\bx_1) = \varphi(\bx_2) \\ 
& -2\ell_{\bx_1}(\by) - Q_{F_{\bx_2}}(1-\alpha) \leq \inf_{\bx \in \mathcal X} -2\ell_{\bx}(\by).
\end{aligned}
\end{equation}
Further unpacking $Q_{F_{\bx_2}}(1-\alpha)$ as in \Cref{lemma:cco}, we obtain the chance-constrained optimization problem:
\begin{equation}\label{eq:ccojoint}
\begin{aligned}
\inf_{\bx_1, \bx_2, q}/\sup_{\bx_1, \bx_2, q} \quad & \varphi(\bx_1) \\
\st \quad & \bx_1, \bx_2 \in \mathcal X\\
& \varphi(\bx_1) = \varphi(\bx_2) \\ 
& -2\ell_{\bx_1}(\by) \leq \inf_{\bx \in \mathcal X} -2\ell_{\bx}(\by) + q \\ 
& \mathbb{P}_{u \sim \mathcal{U}([0,1])}( \mathcal{F}(\bx_2, u) \leq q) \leq 1-\alpha,
\end{aligned}
\end{equation}
where $\mathcal{F}(\bx, u) = F^{-1}_\bx(u)$.

\subsection{Proof of \Cref{lemma:1d_cdf}}  \label{app:constrained_1d_gaussian}

Similar to \Cref{lemma:1d_cdf_part1} in \Cref{app:proof-lemma:1d_cdf_part1}, this proof is divided into two cases.

\myparagraph{\underline{Case of $\mu^*$ = 0}}
Since the case when $\mu^* = 0$ is of particular interest, we show the result in this specific case and then generalize. 
Thus, when $\mu_0 = 0$, the log-likelihood ratio has the following distribution:
\begin{equation}
    \ell_0 \sim \frac{1}{2} + \frac{1}{2} \chi^2_1.
\end{equation}
Additionally, this distribution implies the following stochastic dominance:
\begin{equation} \label{eq:llr_cdf_mu_0}
    \mathbb{P}_{\mu_0} \left(\ell_0 \leq c \right) = \frac{1}{2} \left( 1 + \chi^2_1(c) \right) \geq \chi^2_1(c),
\end{equation}
that is, the log-likelihood ratio CDF is stochastically dominated by the chi-squared with one degree of freedom distribution. This means that the type-I error of the test can be controlled at the $\alpha$ level.

When $\mu > 0$, the closed-form solution to the CDF of $\ell_0$ becomes more complicated, as we can no longer use symmetry around the origin. 
From the result of \Cref{lemma:1d_cdf_part1},
we have the following CDF:
\begin{equation} 
    \label{eq:llr_cdf_non_zero_1}
    \mathbb{P}_{\mu_0} \left( \ell_0 \leq c \right) = \chi^2_1(c) \cdot \bm{1} \{ c < \mu_0^2 \} + \left\{ \Phi(\sqrt{c}) - \Phi\left( \frac{-\mu_0^2 - c}{2 \mu_0} \right) \right\} \cdot \bm{1} \{ c \geq \mu_0^2 \}.
\end{equation}
Note, a quick check of \eqref{eq:llr_cdf_non_zero_1} when $\mu_0 = 0$ reveals agreement with \eqref{eq:llr_cdf_mu_0} such that
\begin{equation}
    \mathbb{P}_{\mu_0} \left(\ell_0 \leq c \right) = \Phi(\sqrt{c}) = \Phi(\sqrt{c}) - \frac{1}{2} + \frac{1}{2} = \frac{1}{2}\left(2 \Phi(\sqrt{c}) - 1 \right) + \frac{1}{2} = \frac{1}{2} \chi^2_1(c) + \frac{1}{2}.
\end{equation}
This completes the case of $\mu^* = 0$.

\myparagraph{\underline{Case of $\mu^* > 0$}}
We already demonstrated above the chi-squared with one degree of freedom dominates the log-likelihood ratio when $\mu_0 = 0$. 
We now show that the dominance holds when $\mu_0 > 0$. Clearly, when $c < \mu_0^2$, $\mathbb{P}_{\mu_0} \left(\ell_0 \leq 0 \right) = \chi^2_1(c)$, making it in fact equal to the chi-squared with one degree of freedom. Suppose $c \geq \mu_0^2$. 
Define 
\[h(c) := \Phi(\sqrt{c}) - \Phi\left( \frac{-\mu_0^2 - c}{2 \mu_0} \right) - \chi^2_1(c).
\]
The stochastic dominance occurs if and only if $h(c) \geq 0$ for all $c \geq \mu_0^2$.

Note first that $\chi^2_1(c) = \Phi(\sqrt{c}) - \Phi(-\sqrt{c})$ and therefore $h(c) = \Phi(-\sqrt{c}) - \Phi\left( \frac{-\mu_0^2 - c}{2 \mu_0} \right)$. Since $\Phi(\cdot)$ is a monotonically increasing function, it is sufficient to show that $-\sqrt{c} - \frac{-\mu_0^2 - c}{2 \mu_0} \geq 0$ for all $c \geq \mu_0^2$. We do so below.

Define a function $f$ as follows:
\[
    f(c) = -\sqrt{c} - \frac{-\mu_0^2 - c}{2 \mu_0}.
\]
Observe that when $c = \mu_0^2$, $f(c) = 0$. 
Consider when $c > \mu_0^2$. 
We obtain the following first and second derivatives:
\[
    f'(c) = \frac{-\mu_0 + \sqrt{c}}{2 \mu_0 \sqrt{c}}
    \quad
    \text{and}
    \quad
    f''(c) = \frac{1}{4}c^{-3/2}. 
\]
By the constraint $c > \mu_0^2$, it follows that $-\mu_0 + \sqrt{c} > 0$, and therefore, $f'(c) > 0$ for all $c > \mu_0^2$. 
Additionally, $f''(c) > 0$ for all $c > \mu_0^2$, so $f$ is convex. 
Hence, we conclude that $f$ is a monotonically increasing function for $c > \mu_0^2$, which starts at $0$ when $c = \mu_0^2$, and thus $f(c) \geq 0$ for all $c \geq \mu_0^2$. 
It therefore follows that 
\[
   \Phi(-\sqrt{c}) \geq \Phi\left(\frac{-\mu_0^2 - c}{2 \mu_0}\right),
\]
and hence $h(c) \geq 0$ for all $c \geq \mu_0^2$. 
As such, we conclude that $\mathbb{P}_{\mu_0}\left(\ell_0 \leq c \right) \geq \chi^2_1(c)$ for all $c \geq 0$.
In other words, that the sampling distribution for the log-likelihood ratio is stochastically dominated by a chi-squared distribution with one degree of freedom. 
This completes the case of $\mu^* > 0$.

\section{Proofs in \Cref{sec:rust_burrus}}
\label{app:sec:rust_burrus_proofs}

\subsection{Proof of \Cref{thm:burrus_false}}

The proof follows by combining \Cref{lemma:burruschi2,lemma:burruschi2B}.

\subsection{Proof of \Cref{lemma:burruschi2}}
\label{sec:proof:lemma:burruschi2}

The proof follows by direct inspection and substitution of $q_\alpha$ and $-2\ell_{\bx}(\by) = \lVert \by - \bK \bx \rVert_2^2$. The interval has the coverage if the $q_\alpha$ is valid by \Cref{thm:interval_coverage} and only if by \Cref{thm:interval_coverage_converse}.

\subsection{Proof of \Cref{lemma:burruschi2B}}
\label{sec:proof:lemma:burruschi2B}

The proof follows by observing $z^2_{\alpha/2} = Q_{\chi^2_1}(1-\alpha)$ and applying \Cref{lemma:SD}

\subsection{Proof of \Cref{lemma:invalidity-counterexample}}
\label{proof:invalidity-counterexample}

We argue by coupling. Note that $\frac 1 2 (y_1-y_2)^2 \sim \chi^2_1$, so that it suffices to show $\lambda \leq \frac 1 2 (y_1-y_2)^2$ for every $y$ to constitute a valid coupling that proves stochastic dominance. 
This is clearly true when $y_1 + y_2 \geq 0$, since $\lambda- \frac 1 2 (y_1-y_2)^2$ is equal to non-positive terms only. 
When $y_1+y_2 < 0$, if both are strictly negative, then $\lambda = 0 \leq \frac 1 2 (y_1-y_2)^2$. 
Then assume without loss of generality that $y_1$ is non-negative, then $y_2$ has to be negative. 
Then $\lambda = y_1^2$, but $y_1 \geq 0$, $y_2 < 0$ and $y_1 < - y_2$ imply that $|y_1 - y_2| = y_1 - y_2 \geq 2y_1 \geq \sqrt{2}y_1$, squaring both sides gives $\frac{1}{2}(y_1-y_2)^2 < y_1^2 = \lambda$. 
This finishes the proof.

\subsection{Proof of \Cref{lem:3d-counterexample}}\label{ProofCounter}

Consider the LLR
\begin{equation}\label{ource_appendix}
\lambda(\mu^* = -1, \by) ~= \min_{\substack{x_1 + x_2 -x_3 = -1\\ \bx \geq \bm{0}}} \Vert \bx -\by \Vert^2_2 - \min_{\bx \geq \bm{0} } \Vert \bx-\by \Vert^2_2.
\end{equation}
The goal of this proof is to show that $\chi^2_1$ does not stochastically dominate 
\eqref{ource_appendix} when $\by \sim \mathcal{N}(x^* = (0,0,1), \bI_3)$. By Corollary 4.26 in \cite{ProbabilityBook}, $X \succeq Y$ implies $\mathbb E[x] > \mathbb E[y]$, so it suffices to show that 
\[
    \mathbb E[\lambda(\mu^* = -1, \by)] > \mathbb E[\chi^2_1] = 1
\]
to complete the proof.

Observe that
\begin{equation}
    \mathbb E [\lambda(\mu^* = -1, \by)] = \mathbb E \bigg [\min_{\substack{\bh^\tp \bx = -1 \\ \bx \geq \bm{0}}} \Vert \bx -\by \Vert^2_2\bigg ] - \mathbb{E} \bigg [\min_{\bx \geq \bm{0} } \Vert \bx -\by \Vert^2_2\bigg ].
\end{equation}
We begin by computing the second term. 
Since 
\[
\min_{\bx \geq \bm{0} } \Vert \bx -\by \Vert^2_2 = \sum_{i=1}^3 (y_i - \max\{y_i, 0\})^2,
\]
we have
\begin{align*}
     \mathbb{E} \bigg [\min_{\bx \geq \bm{0} } \Vert \bx -\by \Vert^2_2\bigg ] &=  \sum_{i=1}^3  \mathbb{E} \bigg [ (y_i - \max\{y_i, 0\})^2\bigg ] \\
     &= 2 \mathbb{E}_{z \sim \mathcal N(0,1)} \bigg[ (z-\max\{z,0\})^2 \bigg ] + \mathbb{E}_{z \sim \mathcal N(1,1)} \bigg[ (z-\max\{z,0\})^2 \bigg ]
\end{align*}
Let $g(z) := (z-\max\{z, 0\})^2$. Using in both cases, we obtain 
\begin{align*}
\mathbb E[g(z)] 
= {\underbrace{\mathbb E[g(z)\mid z\geq 0]}_{0}} \cdot {\mathbb P(z \geq 0)} + \mathbb E[g(z)\mid z < 0] \cdot \mathbb P(z < 0) = \mathbb E[z^2 \mid z < 0] \cdot \mathbb P(z < 0).
\end{align*}
Note that
\begin{align*}
\mathbb E[z^2 \mid z < 0] 
&= (\mathbb E[z\mid z < 0])^2 + \mathrm{Var}[z\mid z<0] \\ 
&= \begin{dcases}
      \bigg(-\frac{\phi(0)}{\Phi(0)}\bigg)^2 + \bigg(1 - \Big(\frac{\phi(0)}{\Phi(0)}\Big)^2\bigg),\; z \sim \mathcal N(0,1)\\ 
    \bigg(1- \frac{\phi(-1)}{\Phi(-1)}\bigg)^2 + 1 + \frac{\phi(-1)}{\Phi(-1)} - \Big(\frac{\phi(-1)}{\Phi(-1)}\Big)^2,\; z \sim \mathcal N(1,1)
\end{dcases} \\ 
&= \begin{dcases}
      1, \; z \sim \mathcal N(0,1)\\ 
    2- \frac{\phi(-1)}{\Phi(-1)}, \; z \sim \mathcal N(1,1),
\end{dcases} 
\end{align*}
where we used the formulas for mean and variance of a truncated Gaussian. 
Finally,
\begin{align*}
\mathbb{E} \bigg [\min_{\bx \geq \bm{0} } \Vert \bx -\by \Vert^2_2\bigg ]  
&= 2 \cdot 1/2 \cdot 1 + (2- \phi(-1)/\Phi(-1))\cdot(\Phi(-1)) \\ 
&= 1 + 2\Phi(-1) - \phi(-1) \\ 
&\approx 1.0753.
\end{align*}
It suffices to prove that 
\[
\mathbb E \bigg [\min_{\substack{\bh^\tp \bx = -1 \\ \bx \geq \bm{0}}} \Vert \bx -\by \Vert^2_2 \bigg] > 2 + 2\Phi(-1) - \phi(-1) \approx 2.0753. 
\]
We will prove that 
\[
\mathbb E \bigg [\min_{\substack{\bh^\tp \bx = -1 \\ \bx \geq \bm{0}}}\Vert \bx -\by \Vert^2_2 \bigg]  = 13/6 \approx 2.166.
\]

Note that the intersection of the plane $\bh^\tp \bx = x_1 +x_2 - x_3 = -1$ and $\bx \geq \bm{0}$ is the parametric surface $\mathcal S = \left \{(u,\ v,\ u+v+1), u\geq 0, v \geq 0 \right\}$, so we can write
\begin{equation}
    \min_{\bx \in \mathcal S} \Vert \bx -\by \Vert^2_2 = \min_{u \geq 0, v \geq 0} (y_1-u)^2 + (y_2-v)^2 + (y_3 - u - v - 1)^2.
\end{equation}
It is convenient to define a new variable $z_3  = 1-y_3 \sim \cN(0,1)$, so that $(y_1, y_2, z_3)$ is sampled from a standard three dimensional Gaussian. Abusing notation we will still write $y_3$ for $z_3$ and then $\by \sim \cN((0,0,0), \bI)$. The optimization problem becomes:
\begin{equation}
\min_{u \geq 0, v \geq 0} (y_1-u)^2 + (y_2-v)^2 + (-y_3 - u - v)^2.
\end{equation}
This can be explicitly solved to yield: 
\begin{align}\label{formula}
    \min_{\substack{\bh^\tp \bx = -1 \\ \bx \geq \bm{0}}} \Vert \bx -\by \Vert^2_2\ = \begin{dcases}
 y_1^2+y_2^2+y_3^2 & y_1-y_3\leq 0 \text{ and } y_2-y_3\leq 0 \\
 \frac{1}{2} \left(y_1^2+2 y_1 y_3+2 y_2^2+y_3^2\right) & y_1-y_3\geq 0 \text{ and } y_1-2 y_2+y_3\geq 0 \\
 \frac{1}{2} \left(2 y_1^2+y_2^2+2 y_2 y_3+y_3^2\right) & y_2-y_3\geq 0 \text{ and }2 y_1-y_2-y_3\leq 0 \\
 \frac{1}{3} \left(y_1+y_2+y_3\right)^2 & \begin{cases}
     2y_1 - y_3 \geq y_2 \geq \max\{y_1, y_3\} \\
     2y_2 - y_3 \geq y_1 \geq \max\{y_2, y_3\} 
 \end{cases}.
\end{dcases}
\end{align}
We split 
\[
\int_{\mathbb R^3} \min_{\substack{\bh^\tp \bx = -1 \\ \bx \geq \bm{0}}}\Vert \bx -\by \Vert^2_2\ \phi(y_1)\phi(y_2)\phi(y_3) \, \mathrm{d}y
\]
into the different domains given by \eqref{formula}, with the value of the expectation being equal to the sum of the different integrals, which we proceed to compute.
\paragraph{\underline{Region 1}:}
\begin{align*}
    I_1 = \int_{y_3 \geq y_1, y_3 \geq y_2}  \frac{e^{-\frac{1}{2}(y_1^2 + y_2^2 +y_3^2)}}{2 \sqrt{2} \pi ^{3/2}} (y_1^2 + y_2^2 + y_3^2) \, \mathrm{d}y.
\end{align*}
Note that by symmetry of the variables in the integrand,
we have
\begin{align*}
    I_1 
    &= \int_{y_2 \geq y_1, y_2 \geq y_3}  \frac{e^{-\frac{1}{2}(y_1^2 + y_2^2 +y_3^2)}}{2 \sqrt{2} \pi ^{3/2}} (y_1^2 + y_2^2 + y_3^2) \, \mathrm{d}y \\
    &= \int_{y_1 \geq y_3, y_1 \geq y_2}  \frac{e^{-\frac{1}{2}(y_1^2 + y_2^2 +y_3^2)}}{2 \sqrt{2} \pi ^{3/2}} (y_1^2 + y_2^2 + y_3^2) \, \mathrm{d}y.
\end{align*}
And since one of the $y_i$ will always be the largest one, the sum of the domains is $\mathbb R^3$ (modulo measure zero intersections that do not affect integration) and we can write 
\begin{align*}
        I_1 = \dfrac 1  3 \int_{\mathbb R^3}  \frac{e^{-\frac{1}{2}(y_1^2 + y_2^2 +y_3^2)}}{2 \sqrt{2} \pi ^{3/2}} (y_1^2 + y_2^2 + y_3^2) \, \mathrm{d}y = \dfrac 1 3 \cdot 3 = 1.
\end{align*}
Here we used that the integral is the expected value of $y_1^2 + y_2^2 + y_3^2$, which is 3
 since the $y_i$ are centered with unit variance.
 
\paragraph{\underline{Region 2}:}
\begin{align}\label{eqi2}
    I_2 = \int_{y_1-y_3\geq 0, y_1-2 y_2+y_3\geq 0 }  \frac{e^{-\frac{1}{2}(y_1^2 + y_2^2 +y_3^2)}}{4 \sqrt{2} \pi ^{3/2}}  \left(y_1^2+2 y_1 y_3+2 y_2^2+y_3^2\right) \, \mathrm{d}y.
\end{align}

Partition $\mathbb R^3$ in four spaces with measure zero intersection, and we aim to argue that the integral of the integrand in \eqref{eqi2} has the same value when integrating over any of them: 
\begin{align*}
A &:= \left\{\by: y_1 \geq y_3, y_2 \geq \frac{y_1 + y_3}{2}\right\}\\
B &:= \left\{\by: y_1 \geq y_3, y_2 \leq \frac{y_1 + y_3}{2}\right\}\\
C &:= \left\{\by: y_1 \leq y_3, y_2 \geq \frac{y_1 + y_3}{2}\right\}\\
D &:= \left\{\by: y_1 \leq y_3, y_2 \leq \frac{y_1 + y_3}{2}\right\}.
\end{align*}
Clearly $I_2 = I_B = \int_A \bh(y_1, y_2, y_3) \, \mathrm{d}y$. Since $\bh$ satisfies $\bh(x_1, x_2, x_3) = \bh(x_3, x_2, x_1)$, we can exchange $y_1$ and $y_3$ in the definitions of the sets, so $I_A = I_C$ and $I_B = I_D$. And since $\bh$ is even with respect to $x_2$ and odd with respect to $x_1, x_3$ we can exchange $y_i$ to $-y_i$ for $i = 1,2,3$ without the result changing. This flips both inequalities, proving $I_A = I_D$ and $I_B = I_C$. We therefore have
\begin{align*}
    I_2 = \dfrac 1 4 \int_{\mathbb R^3}  \frac{e^{-\frac{1}{2}(y_1^2 + y_2^2 +y_3^2)}}{4 \sqrt{2} \pi ^{3/2}}  \left(y_1^2+2 y_1 y_3+2 y_2^2+y_3^2\right) \, \mathrm{d}y = \dfrac{1}{4} \cdot 2 = \dfrac 1 2.
\end{align*}
Here, in the integral, we factor out the sum and using that, the expected value of $y_iy_j$ is $\delta_{ij}$.
\paragraph{\underline{Region 3}:}
\begin{align*}
    I_3 = \int_{y_2-y_3\geq 0, y_2-2 y_1+y_3\geq 0 }  \frac{e^{-\frac{1}{2}(y_1^2 + y_2^2 +y_3^2)}}{4 \sqrt{2} \pi ^{3/2}}  \left(2y_1^2+2 y_2 y_3+y_2^2+y_3^2\right) \, \mathrm{d}y.
\end{align*}
This is exactly the same integral as $I_2$ by switching $y_2$ with $y_1$, so $I_3 = \dfrac 1 2$.

\paragraph{\underline{Region 4}:}
\begin{align}
        I_4 
        &= \int\displaylimits_{ 2y_1 - y_3 \geq y_2 \geq \max(y_1, y_3)}  \frac{e^{-\frac{1}{2}(y_1^2 + y_2^2 +y_3^2)}}{6 \sqrt{2} \pi ^{3/2}} (y_1 + y_2 + y_3)^2 \, \mathrm{d}y \nonumber \\
        & \qquad + \int\displaylimits_{ 2y_2 - y_3 \geq y_1 \geq \max(y_2, y_3)}  \frac{e^{-\frac{1}{2}(y_1^2 + y_2^2 +y_3^2)}}{6 \sqrt{2} \pi ^{3/2}} (y_1 + y_2 + y_3)^2 \, \mathrm{d}y. \label{eqi4}
\end{align}
We partition $\mathbb R^3$ in 12 subspaces with measure 0 intersection and we aim to argue that the integral of the integrand in \eqref{eqi4} (considering one of the integrals only) has the same value when integrating over any of them. For $\sigma$ a permutation of $(y_1, y_2, y_3)$, we define the first 6 subsets as:
\[
\{\by \colon 2y_{\sigma(1)} - y_{\sigma(2)} \geq y_{\sigma(3)} \geq \max\{y_{\sigma(1)}, y_{\sigma(2)}\}\},
\]
and the last 6 subsets as: 
\[\{\by \colon 2y_{\sigma(1)} - y_{\sigma(2)} \leq y_{\sigma(3)} \leq \min\{y_{\sigma(1)}, y_{\sigma(2)}\}\}.
\]
We need to prove that the integral has the same value in any of the 12 subsets. 
Since that the integrand 
\[
\bh(y_1, y_2, y_3) := \frac{e^{-\frac{1}{2}(y_1^2 + y_2^2 +y_3^2)}}{6 \sqrt{2} \pi ^{3/2}} (y_1 + y_2 + y_3)^2
\]
satisfies $\bh(y_1, y_2, y_3) = \bh(y_{\sigma(1)},y_{\sigma(2)},y_{\sigma(3)})$ for all permutations $\sigma$, the value of the integral in between the first and second groups of 6 subsets is the same. 
For a fixed $\sigma$ (say, the identity), since $\bh(y_1, y_2, y_3) = \bh(-y_1, -y_2,-y_3)$, the value over 
\[\{ \by: 2y_1-y_2 \geq y_3 \geq \max\{y_1,y_2\} \}\] is the same as the value over 
\[\{ \by: -2y_1+y_2 \geq -y_3 \geq \max\{-y_1,-y_2\} \} = \{ \by: 2y_1-y_2 \leq y_3 \leq \min\{y_1,y_2\} \},
\]
so the value over the 12 sets is complete. 

It remains to be seen that for a generic $\by = (y_1, y_2, y_3)$, $y_1 \neq y_2 \neq y_3 \neq y_1$ (which can be assumed with probability 1 without affecting the integral), the point belongs to one and just one of the sets. Assume without loss of generality that $y_1$ is the greater of the three and $y_3$ is the smallest. Then since $y_1 > \max\{y_2, y_3\}$ and $y_3 < \min\{y_1, y_2\}$ the only subsets that $\by$ can belong to are: 
\begin{align*}
A &:= \{ \by: 2y_2-y_3 \geq y_1 \geq \max\{y_2, y_3\} \}\\
B &:=\{ \by: 2y_3-y_2 \geq y_1 \geq \max\{y_2, y_3\} \}\\
C &:=\{ \by: 2y_1-y_2 \leq y_3 \leq \min\{y_1, y_2\} \}\\
D &:=\{ \by: 2y_2-y_1 \leq y_3 \leq \min\{y_1, y_2\} \}.
\end{align*}
But $\by$ is not in $B$ because that would require $y_3 \geq \frac{y_1 + y_2}{2}$ but $y_3 < y_1$ and $y_3 < y_2$, and it is also not in $C$ because that would require $y_1 \leq \frac{y_2+y_3}{2}$ and $y_1 > y_2$ and $y_1 > y_3$. $\by$ will be in $A$ if $y_2 > \frac{y_1+y_3}{2}$ and in $D$ if, on the contrary, $y_2 < \frac{y_1+y_3}{2}$, both of which are possible, but not at the same time. We conclude by identifying $I_4$ as the sum of two integrals over subsets that we have defined, and therefore
\begin{align*}
        I_4 = \dfrac{2}{12} \int_{\mathbb R^3}  \frac{e^{-\frac{1}{2}(y_1^2 + y_2^2 +y_3^2)}}{6 \sqrt{2} \pi ^{3/2}} (y_1 + y_2 + y_3)^2 \, \mathrm{d}y = \dfrac 1 6 \cdot 1 = \dfrac 1 6.
\end{align*}
Here we expand the sum and use again  that the expected value of $y_iy_j$ is $\delta_{ij}$. The proof concludes by adding up
\[
    \mathbb E \bigg [\min_{\substack{\bh^\tp \bx = -1 \\ \bx \geq \bm{0}}}\Vert \bx -\by \Vert^2_2 \bigg] =I_1 + I_2 + I_3 + I_4 = \frac {13}{6}.
\]

\subsection{Proof of \Cref{lemma:dimbounding}}
\label{proof:dimbounding}
 We construct a series of counterexamples, indexed by the dimension $p$, and prove that as $p \rightarrow \infty$, the expected value of the LLR diverges. Since stochastic dominance implies inequality of expectations (when expectations are finite), we conclude that the distribution can not be stochastically dominated.
For all $p \in \mathbb{N}$, consider the example in $\mathbb{R}^p (= \mathbb{R}^m)$, $\bK = \bI_p$, $\bx^* = (0,\ldots,0, 1)$, $\bh = (1,\ldots,1,-1)$ (such that $\mu^* = -1$). Let 
\begin{equation}
    \lambda_n(\mu^* = -1, \by) = \min_{\substack{\sum_{i =1}^{p-1}x_i -x_{p} = -1\\ \bx \geq \bm{0}}} \Vert \bx -\by \Vert^2_2 - \min_{\bx \geq \bm{0} } \Vert \bx-\by \Vert^2_2.
\end{equation}
And compute
\begin{equation}    
\mathbb{E}_{\by \sim \mathcal N(\bx^*, \bI_n)}[\lambda_n(-1, \by)] = \mathbb{E}_{\by \sim \mathcal N(\bx^*, \bI_n)}\bigg [\min_{\substack{\sum_{i =1}^{p-1}x_i -x_{p} = -1\\ \bx \geq \bm{0}}} \Vert \bx -\by \Vert^2_2\bigg]  - \mathbb{E}_{\by \sim \mathcal N(\bx^*, \bI_n)}\bigg[\min_{\bx \geq \bm{0} } \Vert \bx-\by \Vert^2_2\bigg].
\end{equation}
For the second term, we have
\begin{align*}
     \mathbb{E} \bigg [\min_{\bx \geq \bm{0} } \Vert \bx -\by \Vert^2_2\bigg ] &=  \sum_{i=1}^p  \mathbb{E} \bigg [ (y_i - \max\{y_i, 0\})^2\bigg ] \\
     &= (p-1) \mathbb{E}_{z \sim \mathcal N(0,1)} \bigg[ (z-\max\{z,0\})^2 \bigg ] + \mathbb{E}_{z \sim \mathcal N(1,1)} \bigg[ (z-\max\{z,0\})^2 \bigg ] \\ 
&= (p-1)\dfrac{1}{2} + (2- \phi(-1)/\Phi(-1))\cdot(\Phi(-1)),
\end{align*}
using similar arguments as the proof in \Cref{ProofCounter}. We will lower bound the first term using duality. For simplicity, define $\bz = (y_1, \ldots, y_{p-1}, y_n-1) \sim \mathcal N(\mathbf{0}, \bI_n)$, and equivalently optimize
\begin{equation}
\min_{\substack{\sum_{i =1}^{p-1}x_i = x_{p}\\ \bx \geq \bm{0}}} \Vert \bx -\bz \Vert^2_2,
\end{equation}
where we defined the feasible $\tilde{\bx} = (x_1, \ldots, x_n-1 = \sum_{i =1}^{p-1} x_i)$ and replaced $\tilde{\bx}$ by $\bx$, abusing notation. Using Fenchel duality, we have that 
\begin{align}
     \min_{\substack{\sum_{i =1}^{p-1}x_i = x_{p}\\ \bx \geq \bm{0}}} \Vert \bx -\bz \Vert^2_2 \geq \sup_{\bm{\xi} \in \mathbb{R}^p} (- f^*(\bm{\xi}) - g^*(-\bm{\xi})), 
\end{align}
where we have noted by $f^*$ the convex conjugate of $f(\bx) := \|\bx - \bz\|_2^2$ and, letting $S$ be the feasible set, we denoted by $g^*$ the convex conjugate of 
\begin{align}
g(\bx) = \begin{cases}
    0  &\quad \text{if } \bx \in S \\ 
    \infty &\quad \text{if } \bx \notin S.
\end{cases}
\end{align}

Note that with these definitions, $\min_{\substack{\sum_{i =1}^{p-1}x_i = x_{p}\\ \bx \geq \bm{0}}} \Vert \bx -\bz \Vert^2_2 = \inf_{\bx} (f(x) + g(x))$ so the weak Fenchel duality applies. 
We compute $f^*(\bm{\xi}) = \frac{1}{4} \Vert \bm{\xi} \Vert^2_2 + \bz^\top \bm{\xi} - \bz^\top \bz$, and 
\begin{align}
    g^*(\bm{\xi}) = \begin{cases}
    0 &\quad \text{if } \xi_i+\xi_p \leq 0 \text{ for } i \in [p-1] \\ 
    \infty &\quad \text{otherwise},
\end{cases}    
\end{align}
so that 
\begin{align}
     \sup_{\bm{\xi} \in \mathbb{R}^p} (-f^*(\bm{\xi}) - g^*(-\bm{\xi})) &=  \sup_{\xi_i + \xi_n \geq 0, \text{ for all } i \in [p-1]} \bigg [-\frac{1}{4} \Vert \bm{\xi} \Vert^2_2 - \bz^\top \bm{\xi} + \bz^\top \bz \bigg] \\ &\geq \sup_{\xi_i + \xi_n \geq 0, \text{ for all } i \in [p-1]} \bigg [-\frac{1}{4} \Vert \bm{\xi} \Vert^2_2 - \bz^\top \bm{\xi} \bigg] . 
\end{align}

Since the supremum is lower bounded by any feasible point, we can further bound by picking a 
feasible $\bm{\xi}^*$ for each possible $\bz$. We define the following:
\begin{equation}
    \bm{\xi}^*(\bz) = \begin{cases}
        -\bz & \text{if }-\bz\text{ is feasible } (-z_i \geq z_n \text{ for all } i)  \\
        (-z_1, \ldots, -z_{p-1}, \max_{i \in [p-1]}{z_i}) & \text{otherwise}.
    \end{cases}
\end{equation}
Observe that
\begin{align*}
 &\min_{\substack{\sum_{i =1}^{p-1}x_i = x_{p}\\ \bx \geq \bm{0}}} \Vert \bx -\bz \Vert^2_2 \\&\qquad \geq -\frac{1}{4} \Vert \bm{\xi}^*(\bz) \Vert^2_2 - \bz^\top \bm{\xi}^*(\bz) \\&\qquad = \begin{dcases}
            \frac{3}{4} \|\bz\|^2 & \text{if }-\bz \text{ is feasible } (-z_i \geq z_n \text{ for all } i)  \\
        \frac{3}{4}\sum_{i=1}^{p-1}z_i^2 + z_n\max_{i \in [p-1]}{z_i} -\frac{1}{4}\Big(\max_{i \in [p-1]}{z_i}\Big)^2 & \text{otherwise}.
 \end{dcases}
\end{align*}
We note that $-\bz$ is feasible with probability $1/p$, by symmetry. Taking expected value over the inequality and using the law of total expectation yields
\begin{align*}
&\mathbb{E}\bigg[\min_{\substack{\sum_{i =1}^{p-1}x_i = x_{p}\\ \bx \geq \bm{0}}} \Vert \bx -\bz \Vert^2_2\bigg ] \\
&\geq \frac{1}{p} \times \frac{3}{4}\mathbb{E}\big[\|z\|_2^2\big] + \frac{p-1}{p} \bigg \{\mathbb{E}\bigg[\frac{3}{4}\sum_{i=1}^{p-1}z_i^2\bigg]+ \mathbb{E}\bigg[z_n\max_{i \in [p-1])}{z_i}\bigg] -\frac{1}{4}\mathbb{E}\bigg[\Big(\max_{i \in [p-1]}{z_i}\Big)^2\bigg] \bigg \} \\ 
&= \frac{3}{4} + \frac{p-1}{p}\bigg\{\frac{3(p-1)}{4} + 0 - \frac{1}{4}\mathbb{E}\bigg[\Big(\max_{i \in [p-1]}{z_i}\Big)^2\bigg]\bigg\}.
\end{align*}
To bound the last term, we use  
\[
    \mathbb{E}\bigg[\Big(\max_{i \in [p-1]}{z_i}\Big)^2\bigg] = \mathbb{E}\bigg[\max_{i \in [p-1]}{z_i}\bigg] + \mathrm{Var}\bigg[\max_{i \in [p-1]}{z_i}\bigg] \leq \sqrt{2\log(p-1)} + 1,
\] 
where the moment bounds are standard results: the expectation bound can be found using Jensen's inequality on $\exp({\sqrt{2\log p}\max_i{z_i}})$ and then bounding $\max_i{z_i} \leq \sum_i z_i$, and the variance bound with \Poincare's inequality applied to a smooth maximum, even though it can be refined \cite{10.1214/ECP.v17-2210}. 
Putting everything together, we obtain
\begin{align*}
\mathbb{E}_{\by \sim \mathcal N(\bx^*, \bI_n)}[\lambda_n(-1, \by)] \geq  \frac{p-1}{p}\bigg\{\frac{3(p-1)}{4} - \frac{1}{4} \sqrt{2\log(p-1)} \bigg\} - \frac{p}{2} + \mathcal{O}(1),
\end{align*}
which is $\mathcal{O}(p)$ and therefore tends to $\infty$ as $p\rightarrow\infty$.
This completes the proof.

\clearpage
\section{Additional numerical illustrations in \Cref{sec:numerical-examples}}
\label{app:sec:numerical-examples}

\subsection{Constrained Gaussian in three dimensions}

We include the analog of \Cref{fig:3d_coverage_length68} with $1-\alpha = 0.95$ in \Cref{fig:3d_coverage_length95}.

\begin{figure}[!ht]
    \centering
    \includegraphics[width = 0.9\columnwidth]{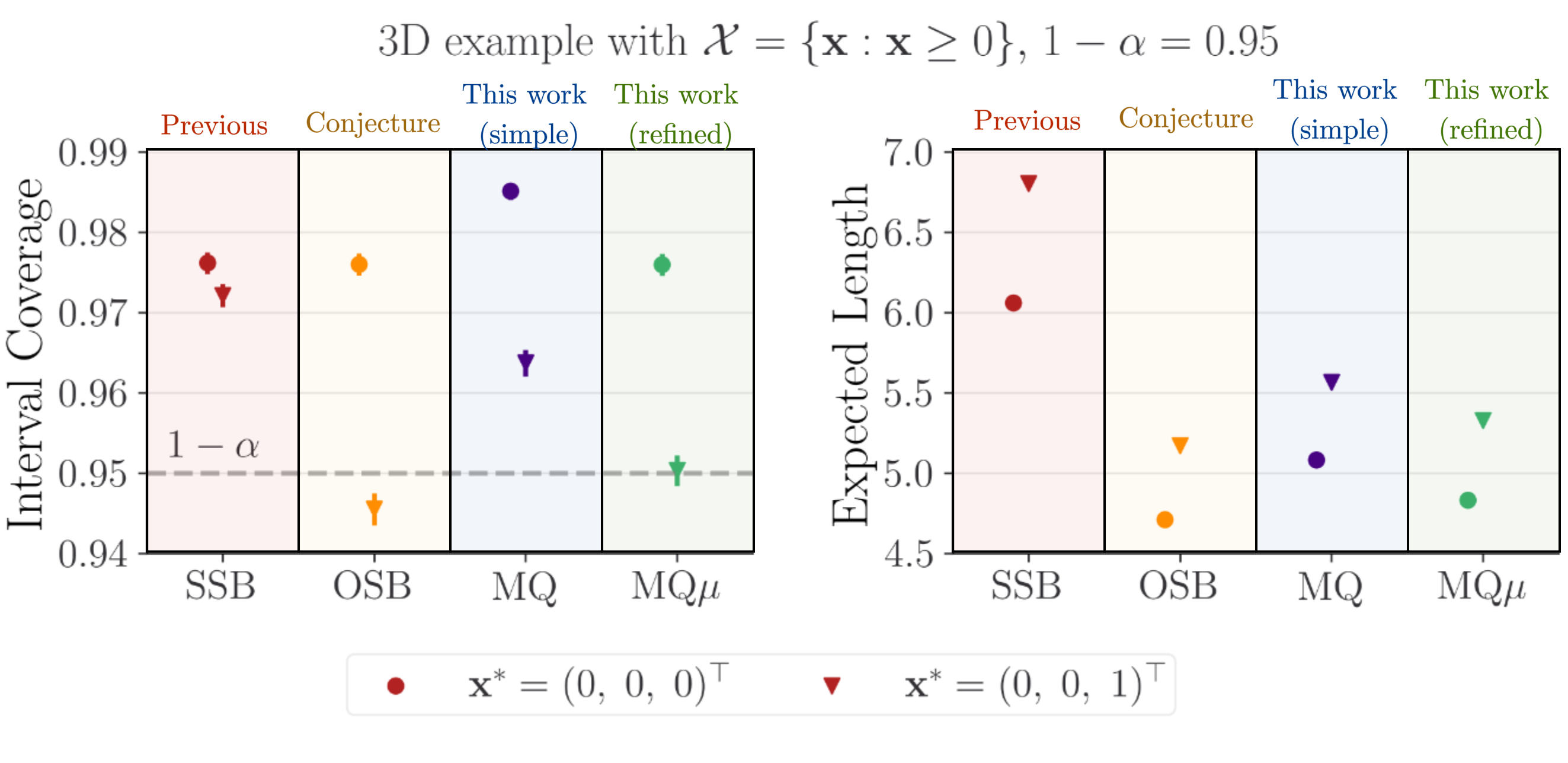}
    \caption{ Estimated interval coverage \textbf{(left)} and expected lengths \textbf{(right)} for $95\%$ intervals resulting from the SSB, OSB, MQ and MQ$
    \mu$ methods for the Gaussian linear model in \eqref{eq:lineargaussianmodel} with $\bK = \bI_3$, $\varphi(\bx) = \bh^\tp \bx = x_1 + x_2 - x_3$ and $\mathcal X = \{\bx \in \mathbb{R}^3: \bx \geq 0 \}$.
    }
    \label{fig:3d_coverage_length95}
\end{figure}

\end{document}